\documentclass[a4paper,10pt]{amsart}

\usepackage{amsmath}
\usepackage{amssymb}
\usepackage{amsthm}
\usepackage{verbatim}
\usepackage[all]{xy}
\usepackage{enumerate}

\usepackage{color}

\newtheorem{thm}{Theorem}[section]
\newtheorem{prop}[thm]{Proposition}
\newtheorem{cor}[thm]{Corollary}
\newtheorem{lem}[thm]{Lemma}

\theoremstyle{definition}

\newtheorem{dfn}[thm]{Definition}

\newtheorem{rmk}[thm]{Remark}

\numberwithin{equation}{section}
\newcommand{\cK}{\mathcal{K}}

\newcommand{\cM}{\mathcal{M}}
\newcommand{\cH}{\mathcal{H}}

\newcommand{\bG}{\mathbb{G}}

\newcommand{\cT}{\mathcal{T}}
\newcommand{\id}{\textrm{id}}

\newcommand{\Id}{\textrm{Id}}
\newcommand{\Dom}{\textrm{Dom}}

\newcommand{\Irr}{{\rm Irr}}

\newcommand{\cL}{\mathcal{L}}

\newcommand{\cB}{\mathcal{B}}

\newcommand{\minotimes}{\otimes_{{\rm min}}}

\newcommand{\cHcoarse}{\mathcal{H}_{{\rm coarse}}}

\newcommand{\cA}{\mathcal{A}}

\newcommand{\Tr}{{\rm Tr}}
\newcommand{\op}{{\rm op}}
\newcommand{\sfA}{\mathsf{A}}

\newcommand{\ONplus}{O_N^+(F)}
\newcommand{\potimes}{\otimes_\partial}

\newcommand{\cP}{\mathcal{P}}
\newcommand{\cQ}{\mathcal{Q}}

\newcommand{\sN}{\mathsf{sN}}
\newcommand{\Nor}{\mathsf{N}}

\title[Gradient forms and strong solidity of free   quantum groups]{Gradient forms and strong solidity of free    quantum groups}

\date{\noindent \today.  \\
 }

\author[M. Caspers]{Martijn Caspers}
\address{TU Delft, EWI/DIAM,
	P.O.Box 5031,
	2600 GA Delft,
	The Netherlands}
\email{m.p.t.caspers@tudelft.nl}

\begin{document}

\maketitle

\begin{abstract}
	Consider the free orthogonal quantum groups $O_N^+(F)$ and  free unitary quantum groups $U_N^+(F)$ with $N \geq 3$. In the case $F = \id_N$ it was proved both by Isono and Fima-Vergnioux that the associated finite von Neumann algebra $L_\infty(O_N^+)$ is strongly solid. Moreover, Isono obtains strong solidity also for  $L_\infty(U_N^+)$ . In this paper we prove for general $F \in GL_N(\mathbb{C})$  that the von Neumann algebras $L_\infty(O_N^+(F))$ and  $L_\infty(U_N^+(F))$ are strongly solid. A crucial part in our proof is the study of coarse properties of  gradient bimodules associated with Dirichlet forms on these algebras and constructions of derivations due to Cipriani--Sauvageot.
\end{abstract}

\section{Introduction}

In their fundamental paper \cite{OzawaPopaAnnals} Ozawa and Popa gave a new method to show that the free group factors do not possess a Cartan subalgebra, a result that was obtained earlier by Voiculescu \cite{Voiculescu} using free entropy. To achieve this,  Ozawa and Popa in fact proved a stronger result. They showed that  the normalizer of any diffuse amenable von Neumann subalgebra of the free group factors, generates a von Neumann algebra that is again amenable. This property then became known as `strong solidity'.  As free group factors are non-amenable and strongly solid they in particular cannot contain Cartan subalgebras.

The approach of \cite{OzawaPopaAnnals} splits into two important parts. The first is the notion of `weak compactness'. \cite{OzawaPopaAnnals} showed that if a von Neumann algebra has the CMAP, then the normalizer of an amenable von Neumann subalgebra acts by conjugation on the subalgebra in a weakly compact way. The second part consists in combining weak compactness with Popa's malleable deformation for the free groups and his spectral gap techniques.

After the results of Ozawa-Popa   several other strong solidity results have been obtained by combining weak compactness with   different   deformation techniques of (group-) von Neumann algebras, often coming from group geometric properties. Roughly (to the knowledge of the author) they can be divided into three categories:
\begin{enumerate}
	\item[(I.1)] The aforementioned malleble deformations;
	\item[(I.2)] \label{Item=IntroCo} The existence of proper cocycles and derivations and deformations introduced by Peterson \cite{Peterson} and further developed by Ozawa--Popa \cite{OzawaPopaAJM};
	\item[(I.3)] \label{Item=IntroCo3} The Akemann-Ostrand property, which  compares  to proper quasi-cocycles and  bi-exactness of groups; c.f. \cite{ChifanSinclair}, \cite{ChifanSinclairU},  \cite{PopaVaesCrelle}.
\end{enumerate}
For group von Neumann algebras the required property in (I.2)  is to a certain extent stronger than  (I.3) in the sense that proper cocycles are in particular quasi-cocycles.  These techniques have been applied successfully to obtain rigidity results for von Neumann algebras (in particular strong solidity results). The current paper also obtains such results and our global methods fall into category (I.2).  Note also that we shall consider derivations on quantum groups  without considering cocycles.

\vspace{0.3cm}

Recently, first examples of type III factors were given that are strongly solid \cite{BHV}, namely the free Araki-Woods factors. This strengthens the earlier results of Houdayer-Ricard  \cite{HoudayerRicard} who showed already the absence of Cartan subalgebras. A crucial result obtained in  \cite{BHV} is the introduction  of a proper notion of weak compactness for the {\it stable normalizer} of a von Neumann subalgebra. Using this notion of weak compactness strong solidity of free Araki-Woods factors is obtained by proving amenability properties of stable normalizers after passing to the continuous core.

\vspace{0.3cm}

This paper grew out of the question of whether the von Neumann algebras of  (arbitrary) free orthogonal and free unitary quantum groups are strongly solid. These free orthogonal and unitary quantum groups have been defined by Wang and Van Daele \cite{WangDaele} as operator algebraic quantum groups.

As  C$^\ast$-algebras the free orthogonal quantum groups are generated by self-adjoint operators $u_{i,j}, 1 \leq i,j \leq N$ with $N \geq 2$ satisfying the relation that the matrix
$
(u_{i,j})_{1 \leq i,j \leq N}$
is unitary. It was shown that this C$^\ast$-algebra can be equipped with a natural structure of a C$^\ast$-algebraic quantum group. Through a canonical GNS-construction this yields a von Neumann algebra $L_\infty(O_N^+)$. Parallel to this  one may also define the free unitary quantum groups with von Neumann algebras $L_\infty(U_N^+), N \geq 2$.  We refer to  Section \ref{Sect=Preliminaries} below for details. These algebras have natural deformations parametrized by an
invertible matrix $F \in GL_N(\mathbb{C})$ which yields quantum groups with non-tracial Haar weights (i.e. quantum groups that are not of Kac type).  We write $L_\infty(O_N^+(F))$ and $L_\infty(U_N^+(F))$ for the associated von Neumann algebras.

Ever since their introduction these algebras have received considerable attention and in particular over the last few years significant structural results have been  obtained for them.
In particular, recently it was proved that free quantum groups can be distinguished from the free group factors \cite{BrannanVergnioux}.  Further, the following is known if we assume $N \geq 3$ (the case $N=2$ corresponds to the amenable $SU_q(2)$ case):
 \begin{enumerate}
 	\item Factoriality results for $L_\infty(U_N^+(F))$ and $L_\infty(O_N^+(F))$ were obtained in
 	 \cite{VaesVergnioux}, \cite{CFY}. In particular for any $F \in GL_N(\mathbb{C})$ the von Neumann algebra $L_\infty(U_N^+(F))$ is a factor.  If  $F = \id_N$ the factors are of type II$_1$ and otherwise they are of type III$_\lambda$ for suitable $\lambda \in (0,1]$.
 \item  For a range of $F \in GL_n(\mathbb{C})$ the algebras $L_\infty(O_N^+(F))$ and $L_\infty(U_N^+(F))$ are non-amenable \cite{Banica}.
 	\item For any $F \in GL_n(\mathbb{C})$ the algebras $L_\infty(O_N^+(F))$ and $L_\infty(U_N^+(F))$ have the CMAP and the Haagerup property \cite{BrannanCrelle}, \cite{Freslon}, \cite{CFY}.
 	\item  $O_N^+ = O_N^+(\Id_N)$ admits a proper cocycle that is weakly contained in the adjoint representation \cite{FimaVergnioux}. That is, it satisfies a property resembling property $HH^+$ of \cite{OzawaPopaAJM}, see also \cite{Peterson}.
    \item In case $F = \id_N$  the factors $L_\infty(O_N^+)$ and $L_\infty(U_N^+)$ are strongly solid   \cite{IsonoTrams}, \cite{FimaVergnioux} (see also the bi-exactness results from \cite{IsonoIMRN}).
    \item For general $F \in GL_N(\mathbb{C})$ the algebras  $L_\infty(O_N^+(F))$ and  $L_\infty(U_N^+(F))$ do not have a Cartan subalgebra \cite{IsonoTrams}.
\end{enumerate}
 In the current context also the results by Voigt \cite{Voigt} on the Baum-Connes conjecture should be mentioned; part of the results of \cite{CFY} and therefore the current paper are based on $q$-computations from \cite{Voigt}.

\vspace{0.3cm}

 In this paper we use   quantum Markov semi-groups (i.e. semi-groups of state preserving normal ucp maps) and Dirichlet forms (i.e. their generators) to obtain strong solidity for all free orthogonal and unitary quantum groups. Dirichlet forms have been studied extensively     \cite{GoldsteinLindsay}, \cite{Cipriani}, \cite{Sauvageot}, \cite{CiprianiSauvageot}, \cite{CaspersSkalski},  \cite{DFSW}, \cite{SkalskiViselter}, \cite{CiprianiSauvageotAmenable}, ... In particular in \cite{CiprianiSauvageot} it was shown that in the tracial case  a Dirichlet form always leads to a derivation as a square root. The derivation takes values in a certain bimodule which we shall call the {\it gradient bimodule}. In this paper we show the following, yielding  a $HH^+$-type deformation as in \cite{Peterson}, \cite{OzawaPopaAJM} (see I.2 above):

   \vspace{0.3cm}

\noindent {\it Key  result (tracial case).}  Let $\bG = \ONplus, F = {\rm Id}_N$ be the tracial free orthogonal quantum group.  There is a Markov semi-group of central multipliers on $\bG$, which is naturally constructed from the results of \cite{BrannanCrelle}, \cite{CFY},   such that the associated  gradient bimodule is weakly contained in the coarse bimodule of $L_\infty(\bG)$.

\vspace{0.3cm}

In fact, the same  result is true in the non-tracial case, but a stronger property is needed in order to treat that case by passing to the continuous core of a von Neumann algebra.   The proof of the key result is based on two crucial estimates for the case $F \overline{F} \in \mathbb{R} \Id_N$: one on the eigenvalues of the Dirichlet form and the other on intertwiners of irreducible representations of $O_N^+(F)$ going back to \cite{VaesVergnioux}.

\vspace{0.3cm}

 In order to tackle all quantum groups $U_N^+(F)$ and $O_N^+(F)$ we treat the above in a more general context. We study semi-groups of state preserving ucp maps and introduce three properties: immediately gradient Hilbert-Schmidt (IGHS), gradient Hilbert-Schmidt (GHS) and gradient coarse (GC).  IGHS (as well as GHS) essentially implies GC (see Proposition \ref{Prop=IGHSimpliesGC}). The key result announced in the previous paragraphs is proved by showing that $O_N^+(F)$ with $F \overline{F} \in \mathbb{R} \Id_N$ admits a semi-group  that is IGHS. Preservation under free products and behavior under crossed products of IGHS and GC are studied in Section \ref{Sect=IGHSstable} from which we show that general free quantum groups admit semi-groups that are IGHS and their cores admit GC semi-groups.

 \vspace{0.3cm}

 These results suffice to fuel the theory as set out in the beginning of the introduction. We first recall the definition of strong solidity.

  \begin{dfn}\label{Dfn=StronglySolid}
	A von Neumann algebra $\cM$ is called {\bf strongly solid}  if for every   diffuse,   amenable von Neumann subalgebra $\cQ \subseteq \cM$   for which there exists a faithful normal conditional expectation $E_{\cQ}: \cM \rightarrow \cQ$, we have that  the von Neumann algebra   $\Nor_{\cM}(\cQ)$  generated by  the normalizer
	\[
  	 \{ u \in \cM \mid u \textrm{ unitary and } u \cQ u^\ast = \cQ \},
	\]
 is still amenable.
\end{dfn}


 We use the notion of weakly compact actions of stable normalizers from \cite{BHV} and the deformation techniques (starting from proper derivations) as introduced by Peterson  \cite{Peterson} and further developed by Ozawa-Popa \cite{OzawaPopaAJM}. Eventually this leads to strong solidity of all free orthogonal and unitary quantum groups.  The precise statement  we need from these sources does not occur in the literature (though very similar statements are claimed in \cite{Sinclair}, \cite{FimaVergnioux}, \cite{BHV}) and hence we incorporate them in the appendix.

 We  conclude:

 \vspace{0.3cm}

 \noindent {\bf Theorem.} For $F \in GL_N(\mathbb{C}), N \geq 3$ let $\bG$ be either $O_N^+(F)$ or $U_N^+(F)$. $L_\infty(\bG)$  is strongly solid.

     \vspace{0.3cm}

      Note that if a Kac type quantum group with the CMAP has the Haagerup property then our approach here shows that there is a canonincal candidate for a bimodule (i.e. the gradient bimodule) and a proper real derivation into this bimodule.  It   remains then to show that the gradient bimodule is weakly contained in the coarse bimodule to obtain good deformations. It would be interesting to know how large the class of quantum groups is to which this strategy applies.

%

  \vspace{0.3cm}

  \noindent {\it Structure.} Section \ref{Sect=Preliminaries} contains various preliminaries on quantum groups and von Neumann algebras. Section \ref{Sect=Gradient} recalls results by Cipriani-Sauvageot and some non-tracial extensions. Section \ref{Sect=IGHS} contains general results on Markov semi-groups and coarse properties of the gradient bimodule.
  Section \ref{Sect=IGHSstable} contains  stability properties of  IGHS, GHS and GC that are nedeed  to treat  $\ONplus$ for all $F \in GL_N(\mathbb{C})$.
   In Sections \ref{Sect=Solid1} and \ref{Sect=Solid2} we prove our main theorem, i.e. the strong solidity result. Finally in Section \ref{Sect=Final} we prove a compression result. The parts that are directly taken from \cite{BHV} and \cite{OzawaPopaAJM} are included in Appendix \ref{Sect=AppendixA}.

  \vspace{0.3cm}

  \noindent {\it Acknowledgements.}  The author thanks Yusuke Isono, Marius Junge, Adam Skalski, Stefaan Vaes, Mateusz Wasilewski and Makoto Yamashita for useful discussions and/or  useful comments on the paper. The author thanks the referee for suggesting several improvements to the paper.

\section{Preliminaries}\label{Sect=Preliminaries}

\subsection{Free orthogonal quantum groups} In \cite{Woronowicz} Woronowicz defined a  compact C$^\ast$-algebraic quantum group $\bG = (\sfA, \Delta_{\sfA})$ as a pair of a unital C$^\ast$-algebra $\sfA$ with a comultiplication $\Delta_{\sfA}: \sfA \rightarrow \sfA \otimes \sfA$ (minimal tensor product) such that $(\Delta_{\sfA} \otimes {\rm id}) \Delta_{\sfA} = ({\rm id} \otimes \Delta_{\sfA}) \Delta_{\sfA}$ and such that both $(\sfA \otimes 1) \Delta_{\sfA}(\sfA)$ and $(1 \otimes \sfA) \Delta_{\sfA}(\sfA)$ are dense in $\sfA \otimes \sfA$. Compact quantum groups have a unique Haar state $\varphi$ such that for $x \in \sfA$,
\begin{equation}\label{Eqn=Invariance}
(\varphi \otimes {\rm id}) \Delta_{\sfA}(x) = \varphi(x) 1 = ({\rm id} \otimes \varphi)
 \Delta_{\sfA}(x).
 \end{equation}
  Let $(\pi_\varphi, \cH_\varphi)$ be the GNS-representation of $\varphi$ with cyclic vector $\Omega_\varphi := 1 \in \cH_\varphi$ and set $L_\infty(\bG) = \pi_\varphi(\sfA)''$.  The state $\varphi$ determines a unique normal faithful state, still denoted by $\varphi$,  on $L_\infty(\bG)$ satisfying \eqref{Eqn=Invariance} for all $x \in L_\infty(\bG)$. Here $\Delta_{\bG} := \Delta_{L_\infty(\bG)}$ is then the von Neumann algebraic comultiplication, which shall not be used in this paper. The triple $(L_\infty(\bG), \Delta_{\bG}, \varphi)$ is then a von  Neumann algebraic quantum group in the Kustermans-Vaes sense, see \cite{KustermansVaesScan}. It is common to write $L_2(\bG)$ for $\cH_\varphi$.

  A finite dimensional unitary representation of $\bG$ is a unitary element $u \in L_\infty(\bG) \otimes M_n(\mathbb{C})$ such that $(\Delta_{\bG} \otimes {\rm id})(u) = u_{13} u_{23}$ with $u_{23} = 1 \otimes u$ and $u_{13} = (\Sigma \otimes {\rm id})(u_{23})$ with $\Sigma$ the flip. We denote $\Irr(\bG)$ for the set of all irreducible representations modulo equivalence. For $\alpha \in \Irr(\bG)$ we let $u^\alpha$ be a corepresentation of class $\alpha$; none of the constructions in this paper depend on the choice of the representative $u^\alpha$.  We use $\alpha \subseteq \beta$ to say that $\alpha$ is a subrepresentation of $\beta$. This means that $u^\alpha = (1 \otimes p) u^\beta$ with $u^\beta \in L_\infty(\bG) \otimes M_{n_\beta}(\mathbb{C})$  for some projection $p \in M_{n_\beta}(\mathbb{C})$  such that $1 \otimes p$ commutes with $u^\beta$.

   In the literature the terminology `corepresentation' is also common to refer to representations, but here we stay with `representation' as our terminology. Let $\alpha \in \Irr(\bG)$ and let $X_\alpha$ be the span of elements $(\id \otimes \omega)(u), \omega \in M_n(\mathbb{C})^\ast$ and let $\cH_\alpha = X_\alpha \Omega_\varphi$. $X_\alpha$ is called the space of matrix coefficients of $\alpha$.  The projection of $L_2(\bG) := L_2(L_\infty(\bG))$ onto $\cH_\alpha$ is denoted by $p_\alpha$ and is called the isotypical projection of $\alpha$.

 We say that $\bG$ is finitely generated if $\Irr(\bG)$ is finitely generated as a fusion category. That is, there exists a finite dimensional representation $\alpha$ such that for every $\beta \in \Irr(\bG)$ there exists a $k \in \mathbb{N}$ such that $\beta \subseteq \alpha^{\otimes k}$. We may assume that the trivial representation is contained in $\alpha$ and that $\alpha$ is equivalent to its contragredient representation. Then the minimal such $k$ is called the length of $\beta$ which we denote by $l(\beta)$. The length depends on $\alpha$, which at the point that we need it is implicitly fixed.

  In  \cite{WangDaele} Wang and Van Daele introduced the free orthogonal quantum groups. We recall them here. Throughout the entire paper fix an integer $N \geq 2$ and let $F$ be a invertible complex matrix of size $N \times N$.   Let $\sfA := \sfA(\ONplus)$ be the universal C$^\ast$-algebra generated by   elements $u_{i,j}, 1 \leq i, j \leq N$ subject to the relation that the matrix $u^{1} = (u_{i,j})_{i,j}$ is unitary and $u^1 = F \overline{u^{1}} F^{-1}$. Here  $\overline{u^{1}}$ is the entrywise adjoint $(u_{i,j}^\ast)_{i,j}$. It has comultiplication $\Delta_{\sfA}(u_{i,j}) = \sum_{k=1}^N u_{i,k} \otimes u_{k,j}$. We call this quantum group $\ONplus$ with von Neumann algebra $L_\infty(\ONplus)$ and Haar state $\varphi$. In case $N = 2$  the quantum group is amenable \cite{BrannanNotes}, \cite{Banica}.

    If we assume that $F \overline{F} \in \mathbb{R} {\rm Id}_N$ the quantum group $\ONplus$ is monoidally equivalent to $SU_q(2)$ where the number $0 < q < 1$ is such that  $q + q^{-1} =  {\rm Tr}(F^\ast F)$.      Also set $N_q = q + q^{-1}$ which is the quantum dimension of the fundamental representation $u^{1}$. It holds that $N_q \geq N$ and equality holds if and only if the Haar state of $\ONplus$ is tracial. Note that $q$ is the smallest   root of $x^2 - N_q x + 1 = 0$.
     In this case, i.e. when $F \overline{F} \in \mathbb{R} {\rm Id}_N$, the representation theory of  $\ONplus$ as a fusion category was described by Banica \cite{Banica}. We have $\Irr(\ONplus) \simeq \mathbb{N}$ with $0$ the trivial representation and $1$ isomorphic to $u^1$ above. In fact we will denote $u^\alpha$ for the representation of class $\alpha \in \mathbb{N}$. The fusion rules are for $\alpha \geq \beta$,
    \[
    \beta \otimes \alpha \simeq
     \alpha \otimes \beta \simeq \vert \alpha - \beta \vert \oplus \vert \alpha - \beta +2 \vert \oplus \ldots \oplus \vert \alpha + \beta \vert.
    \]
    We write $n_\alpha$ for the dimension of $\alpha \in \Irr(\ONplus)$. It satisfies the recurrence relation $N n_\alpha = n_{\alpha +1} + n_{\alpha -1}$. If we let $q_0 \in (0,1)$ be the smallest positive root of $x^2 - Nx + 1 =0$ then we have $n_\alpha \simeq q_0^{- \alpha} + O(1)$. Also $q \leq q_0$.  It follows that   $\limsup_{\alpha \rightarrow \infty} (n_\alpha)^{1/\alpha} q \leq 1$.

        \subsection{General von Neumann algebra theory}
 For von Neumann algebra theory we refer to the books by Takesaki  \cite{TakesakiI}, \cite{TakesakiII}.

\vspace{0.3cm}

\noindent {\bf Assumption.} Throughout the entire paper $\cM$ is a von Neumann algebra with fixed normal faithful state $\varphi$. In case of a compact quantum group $\varphi$ is the Haar state.

\vspace{0.3cm}

We use $\cM^{\op}$ for the opposite von Neumann algebra and write $x^\op, x \in \cM$ for elements in the opposite algebra. We also set $\overline{x} = (x^\ast)^{\op}$.
 We write $L_2(\cM)$ for the standard form Hilbert space. It has distinguished vector $\Omega_{\varphi}$ such that $x \mapsto x \Omega_\varphi$ is a GNS-map for $\varphi$ with $\varphi(y^\ast x) = \langle x \Omega_{\varphi}, y \Omega_{\varphi} \rangle$.

 \subsection{Tomita-Takesaki theory}
Let $S$ be the closure of the map $x \Omega_{\varphi} \mapsto x^\ast \Omega_{\varphi}, x \in \cM$ which has polar decomposition $S = J \nabla^{\frac{1}{2}}$. Here $J: L_2(\cM) \rightarrow L_2(\cM)$ is an anti-linear isometry with $J^2 = 1$. We set the modular group $\sigma^{\varphi}_t(x) = \nabla^{it} x \nabla^{-it}$. We define the Tomita algebra $\mathcal{T}_\varphi$ as the $\ast$-algebra of $x \in \cM$ that are analytic for $\sigma^{\varphi}$. We write $\xi x$ for $Jx^\ast J \xi, \xi \in L_2(\cM)$. Then $\Omega_\varphi x = \sigma_{-i/2}^\varphi(x) \Omega_\varphi$. We have that $\nabla^{\frac{1}{4}} x \Omega_\varphi \in L_2^+(\cM)$, where the latter set denotes the positive cone in the standard Hilbert space.
We also record that \cite[Lemma VIII.3.18]{TakesakiII},
\begin{equation}\label{Eqn=ActionEstimates}
\Vert x y \Omega_\varphi \Vert_2 \leq \Vert \sigma_{i/2}^\varphi(y) \Vert \Vert x \Omega_\varphi \Vert_2, \qquad   \Vert y x  \Omega_\varphi \Vert_2 \leq \Vert y \Vert \Vert x \Omega_\varphi \Vert_2   \qquad x \in \cM, y \in \mathcal{T}_\varphi.
\end{equation}

\subsection{Hilbert-Schmidt operators} \label{Sect=HS}
Let $H: \Dom(H) \subseteq \cM \rightarrow \cM$ be a linear map. We say that $H$ is Hilbert-Schmidt if the associated map   $L_2(\cM) \rightarrow L_2(\cM)$  that sends $x \Omega_{\varphi}$ to $H(x) \Omega_{\varphi}$ is Hilbert-Schmidt.  We   denote the extension of $H$  as a Hilbert-Schmidt map on $L_2(\cM)$ by $H^{(l,2)}$. Then $\Vert H^{(l,2)} \Vert_{HS}^2 = \sum_{i,j}  \vert \langle H^{(l,2)} e_i, e_j \rangle \vert^2$ is the Hilbert-Schmidt norm for any choice of orthonormal basis $e_i$. Let $\overline{L_2(\cM)}$ be the conjugate Hilbert space of $L_2(\cM)$.  We may identify $H^{(l,2)}$  isometrically and linearly  with a vector $\zeta_H \in L_2(\cM) \otimes \overline{L_2(\cM)}$ by means of the identification,
\[
\langle  H^{(l,2)}(\xi), \eta \rangle = \langle \xi \otimes   \overline{\eta},   \zeta_H   \rangle.
\]

 \subsection{Bimodules and weak containment}
A   $\cM$-$\cM$-bimodule is a Hilbert space $\cH$ with  normal $\ast$-representations, $\pi_1$ of $\cM$ and   $\pi_2$ of the opposite algebra $\cM^{op}$, that commute. Notation: $a \xi b = \pi_1(a) \pi_2(b) \xi$ with  $\xi \in \cH, a,b, \in \cM$. We write $_\cM \cH_\cM$ for the bimodule structure, or briefly $\cH$ if the bimodule structure is clear.
We recall the Connes-Jones definition of weak containment \cite{ConnesJones}. We also refer to the extensive treatment of bimodules by Popa \cite{PopaIncrest}.

\begin{dfn}\label{Dfn=WeakContainment}
	Let $\cK$ and $\cH$ be two $\cM$-$\cM$-bimodules. We say that $\cK$ is weakly contained in $\cH$, notation $\cK \preceq \cH$, if for every $\xi \in \cK$,   $\varepsilon >0, E, F \subseteq \cM$ finite there exist finitely many   $\eta_j \in \cK$ indexed by $j \in G$ such that for all $x, \in E$, $y \in F$,
	\[
	\vert \langle  x \xi y, \xi \rangle - \langle \sum_{j \in G} x \eta_j y, \eta_j \rangle \vert < \varepsilon.
	\]
\end{dfn}

We let $_{\cM}L_2(\cM)_{\cM}$ denote the identity bimodule $L_2(\cM)$ with actions $a \xi b$ = $a J b^\ast J \xi$. We let $_{\cM}L_2(\cM) \otimes  L_2(\cM)_{\cM}$ denote the coarse bimodule  with actions $a (\xi \otimes \eta) b = a \xi \otimes \eta b$. The following is Popa's definition of amenability \cite{PopaIncrest}, \cite{PopaBetti}.

\begin{dfn}\label{Dfn=Amenable}
	A  von Neumann algebra $\cM$ is called amenable if    $_{\cM}L_2(\cM)_{\cM} \preceq \:\: _{\cM} \! L_2(\cM) \otimes  L_2(\cM)_{\cM}$.
\end{dfn}

Let $\cM$ be a   von Neumann algebra with normal faithful state $\varphi$.
If $\Phi: \cM \rightarrow \cM$ is a completely positive map then by Stinespring's  theorem \cite[Section 5.2]{EffrosRuan}, \cite{Pisier} there exists a $\cM$-$\cM$-bimodule $\cH_\Phi$ with distinguished vector $\eta_\Phi$ such that for $x,y \in \cM$ we have $\langle x \eta_{\Phi} y, \eta_{\Phi} \rangle = \langle \Phi(x) \Omega_\varphi y, \Omega_\varphi \rangle$. Recall that the $\cH_\Phi$ can be realized as follows. Take $\cM \otimes L_2(\cM)$ with pre-inner product $\langle a \odot \xi, c \odot \eta \rangle = \langle \Phi(c^\ast a) \xi, \eta \rangle$.  Quotienting out the nilspace and taking a completion yields $\cH_\Phi$ with actions $x \cdot (a \odot \xi) \cdot y = (xa \odot \xi y), a, x,y \in \cM, \xi \in L_2(\cM)$.  Then take $\eta_\Phi = 1 \otimes \Omega_\varphi$. The following properties are well-known and now easy to check.
  $\eta_\Phi$ is cyclic in the sense that the linear span of $\cM \eta_\Phi \cM$ is dense in $\cH_{\Phi}$. For any $\xi \in \cH_\Phi$ we have that the map $\varphi_{\xi, \xi}(x \otimes y^\op) = \langle x \xi y, \xi \rangle$ is positive on $\cM \odot \cM^{\op}$.
 If $\Phi = \Id_{\cM}$ then $\cH_\Phi = L_2(\cM)$ (even as $\cM$-$\cM$-bimodules).
We write $\overline{\cH}$ for the conjugate space of a Hilbert space $\cH$. Then the modular conjugation $J: L_2(\cM) \rightarrow \overline{L_2(\cM)}: \xi \mapsto \overline{\xi^\ast}$ is a linear isometric isomorphism.
The following  was pointed out in \cite[p. 28, Sect. 1.5: comments]{PopaIncrest} (attributed to Connes), but we could not find an explicit proof. The following argument follows closely \cite[Lemma 2.15]{Avsec}.

We shall call a map $\Phi: \cM \rightarrow \cM$ a Markov map if it is normal, $\varphi$-preserving and unital, completely positive (ucp).

\begin{lem} \label{Lem=HSStinespring} Let $\cM$ be a von Neumann algebra with normal faithful state $\varphi$.
	Let $\Phi: \cM \rightarrow \cM$ be a Markov map and suppose that $\Phi^{(l,2)}: L_2(\cM) \rightarrow L_2(\cM)$ is Hilbert-Schmidt. Then   $\cH_\Phi \preceq \cHcoarse$.
\end{lem}
\begin{proof}
Let $(\cH_\Phi, \eta_\Phi)$ be the pointed Stinespring bimodule. Take $c_1, c_2, d_1, d_2 \in \cM$ arbitrary and set $\xi_1 = c_1 \eta_{\Phi} d_1$ and $\xi_2 = c_2 \eta_{\Phi} d_2$.
	Now we get for $a,b \in \cM$ that there exists a vector $\zeta_\Phi \in L_2(\cM) \otimes \overline{ L_2(\cM)}$ (see Section \ref{Sect=HS}), such that
	\[
	\begin{split}
 	 \langle a \xi_1 b, \xi_2 \rangle = &
	\langle a c_1 \eta_\Phi d_1 b, c_2 \eta_\Phi d_2 \rangle
	=
	\langle c_2^\ast a c_1 \eta_\Phi d_1 b d_2^\ast, \eta_\Phi  \rangle
	= \langle \Phi(c_2^\ast a c_1) \Omega_\varphi d_1 b d_2^\ast, \Omega_\varphi \rangle \\
= & \langle \Phi(c_2^\ast a c_1) \Omega_\varphi , \Omega_\varphi d_2 b^\ast d_1^\ast \rangle
= \langle c_2^\ast a c_1 \Omega_\varphi \otimes \overline{ \Omega_\varphi d_2 b^\ast d_1^\ast}, \zeta_\Phi \rangle
= \langle   a c_1 \Omega_\varphi \otimes \overline{ \Omega_\varphi d_2 b^\ast}, (c_2 \otimes 1) \zeta_\Phi (1 \otimes \overline{d_1}) \rangle 
	\end{split}
	\]	
	This shows that $\varphi_{\xi_1, \xi_2}(a \otimes b) = \langle a \xi_1 b, \xi_2 \rangle$ extends to a bounded functional on $\cM \minotimes \cM^\op$, moreover it is normal and thus extends to the von Neumann tensor product $\cM \otimes \cM^\op$ (by Kaplansky of the same norm). Take  finitely many vectors $\xi_i$ of the above form and put $\xi = \sum_i \xi_i$. As $\varphi_{\xi, \xi}$ is   positive on   $\cM \minotimes \cM^\op$ it  extends to a positive normal functional on $\cM \otimes \cM^\op$ by Kaplansky.
	Then, as
	$L^2(\cM) \otimes L^2(\cM)$
	is the standard form Hilbert space for $\cM \otimes \cM^{\op}$,  pick $\eta \in L^2(\cM) \otimes L^2(\cM)$ such that $\langle x \xi y, \xi \rangle = \langle x \eta y, \eta \rangle$.   As vectors $\xi$ of this form are dense in  $\cH_{\Phi}$ the lemma follows by approximation.
\end{proof}

\subsection{Quantum Dirichlet forms}\label{Sect=Dirichlet} Recall that a Markov map $\cM \rightarrow \cM$ was defined as a $\varphi$-preserving normal ucp map (the normal faithful state $\varphi$ is always implicitly fixed and usually the Haar state of a compact quantum group in this paper).  We say that a Markov map $\Phi$ is $\varphi$-modular if  $\Phi \circ \sigma^\varphi_t = \sigma_t^\varphi \circ \Phi$ for all $t \in \mathbb{R}$.
  A Markov map $\Phi: \cM \rightarrow \cM$
is called KMS-symmetric if
\[
\langle \nabla^{\frac{1}{4}} \Phi(x)\Omega_\varphi, \nabla^{\frac{1}{4}} y\Omega_\varphi \rangle = \langle \nabla^{\frac{1}{4}} x\Omega_\varphi, \nabla^{\frac{1}{4}} \Phi(y) \Omega_\varphi \rangle, \qquad x,y \in \cM.
\]
If $\Phi: \cM \rightarrow \cM$ is any Markov map then by a standard interpolation argument  there exists a contractive map $\Phi^{(2)}: L_2(\cM) \rightarrow L_2(\cM)$ acting as
\[
\nabla^{\frac{1}{4}} x \Omega_\varphi \mapsto \nabla^{\frac{1}{4}} \Phi(x) \Omega_{\varphi}, x \in \cM.
\]
KMS-symmetry is then equivalent to $\Phi^{(2)}$ being self-adjoint.
 With a  {\bf Markov semi-group} we mean a   semi-group $(\Phi_t)_{t \geq 0}$ of KMS-symmetric Markov maps $\cM \rightarrow \cM$ such that for every $x \in \cM$ the function $t \mapsto \Phi_t(x)$ is $\sigma$-weakly continuous. 

  For $\xi \in L_2(\cM)$ we may write $\xi = \sum_{k =0}^3 i^{k} \xi_k$ with $\xi_k \in L_2^+(\cM)$ (the positive cone of the standard form). Let $\xi_+ = \xi_0$.
   Then let $\xi_{\wedge} = \xi - (\xi - \Omega_\varphi)_+$.

\begin{dfn}
	A (closed, densely defined) quadratic form $Q$ with domain $\Dom(Q)\subseteq L_2(\cM)$ is said to be a {\bf conservative Dirichlet form} if (1)  $\xi \in \Dom(Q)$ implies $J\xi \in \Dom(Q)$ and $Q(J\xi) = Q(\xi)$;  (2) $\Omega_\varphi \in \Dom(Q)$ and $Q(\Omega_\varphi) = 0$; (3) For $\xi \in \Dom(Q)$ we have $\xi_+ \in \Dom(Q)$, $\xi_{\wedge} \in \Dom(Q)$ and moreover $Q(\xi_+) \leq Q(\xi), Q(\xi_{\wedge}) \leq Q(\xi)$. 	
\end{dfn}

A quadratic form $Q$ is called {\bf conservative, completely Dirichlet} if its matrix amplification $Q^{[n]}$ is a conservative Dirichlet form for every $n \geq 1$. Here $\Dom(Q^{[n]})$ are the $n\times n$-matrices with entries in $\Dom(Q)$ and  $Q^{[n]}((\xi_{i,j})_{i,j = 1}^n) = \sum_{i,j} Q(\xi_{i,j})$. If $Q$ is a quadratic form then write $\Delta = \Delta_Q \geq 0$ for the unique (closed densely defined) unbounded operator with $\Dom(\Delta^{\frac{1}{2}}) = \Dom(Q)$ such that $Q(\xi) = \langle \Delta^{\frac{1}{2}} \xi,  \Delta^{\frac{1}{2}} \xi \rangle$.  The following result was obtained independently in \cite{GoldsteinLindsay} (in terms of Haagerup $L_p$-spaces) and \cite{Cipriani} (in terms of standard forms, being the formulation we take here).

\begin{thm}\label{Thm=QSchonberg}
	$Q$ is a conservative completely Dirichlet form if and only if the semi-group $(e^{-t\Delta})_{t \geq 0}$ determines a   Markov semi-group, meaning that there is a  Markov semi-group $(\Phi_t)_{t \geq 0}$ on $\cM$ such that $(e^{-t \Delta})_{t \geq 0} = (\Phi_t^{(2)})_{t \geq 0}$.
\end{thm}	

	In view of Sch\"onberg's correspondence \cite[Appendix C]{NateTaka}, conservative completely Dirichlet forms are therefore non-commutative analogues of conditionally positive definite functions.
We could have rephrased our statements in terms of conditionally negative definite functions by considering $-\Delta$ instead of $\Delta$.	



 We need the following lifting property from \cite[Lemma 5.2]{CaspersSkalski}, the proof of which is essentially contained in \cite{OkayasuTomatsu}. We also recall that on the $L_2$-level strong continuity and weak continuity of $(\Phi_t^{(2)})_{t \geq 0}$ are equivalent (see \cite[Lemma 3.5]{CaspersBMO}). $\sigma$-weak continuity of $(\Phi_t)_{t \geq 0}$ is equivalent to weak continuity of $(\Phi_t^{(2)})_{t \geq 0}$. A map $T: L_2(\cM) \rightarrow L_2(\cM)$ is called completely positive if $\Id_n \otimes T$ maps the positive cone in the standard form $L_2^+(M_n \otimes \cM)$ into itself for every $n \in \mathbb{N}$.

\begin{lem}\label{Lem=MarkovLift}
	Let  $C = \{ \xi \in L_2(\cM) \mid 0 \leq \xi \leq \Omega_\varphi\}$.
	If $(S_t)_{t \geq 0}$ is a strongly continuous semi-group of self-adjoint completely positive operators on $L_2(\cM)$ such that $S_t(\Omega_\varphi) = \Omega_\varphi$ and such that $S_t(C)\subseteq C$.    Then there exists a   Markov semi-group $(\Phi_t)_{t \geq 0}$ on $\cM$ such that $\Phi_t^{(2)} = S_t$.
\end{lem}

\section{Gradient forms and the results by Cipriani-Sauvageot}\label{Sect=Gradient} We recall some of the work of Cipriani-Sauvageot \cite{CiprianiSauvageot} which is crucial in our approach. We need a slightly more general version going beyond the case of tracial states of their construction. Note that we do not prove the existence of a square root in the non-tracial setting however (which is one of the main results of \cite{CiprianiSauvageot}; the question is also asked for in \cite{SkalskiViselter}).

 \subsection{The gradient bimodule}
   If $Q$ is a conservative completely Dirichlet form on $L_2(\cM)$, then let $\Delta \geq 0$ be such that $\Dom(\Delta^{\frac{1}{2}}) = \Dom(Q)$ and $Q(\xi) = \langle \Delta^{\frac{1}{2}} \xi, \Delta^{\frac{1}{2}} \xi \rangle$.

   \begin{dfn} We assume that there is a $\sigma$-weakly dense unital $\ast$-subalgebra of the Tomita algebra $\mathcal{T}_\varphi$ {\bf which we call $\mathcal{A}$} such that  $\nabla^{\frac{1}{4}} \mathcal{A} \Omega_\varphi \subseteq  \Dom(\Delta)$ and $\Delta (\nabla^{\frac{1}{4}} \mathcal{A} \Omega_\varphi) \subseteq \nabla^{\frac{1}{4}} \mathcal{A} \Omega_\varphi$.  For $a \in \cA$ we write $\Delta(a) \in \cA$ for the unique\footnote{If $\nabla^{\frac{1}{4}} \Delta(a) \Omega_\varphi = 0$  then  for all $y \in \cM$ we have $0 = \langle  \nabla^{\frac{1}{4}} \Delta(a) \Omega_\varphi, y \Omega_\varphi  \rangle = \langle  \Delta(a) \Omega_\varphi,  \nabla^{\frac{1}{4}} y \Omega_\varphi  \rangle$. Hence $\Delta(a) \Omega_\varphi = 0$  and since $\Omega_\varphi$ is cyclic we have that $\Delta(a) = 0$.}
   element such that $ \nabla^{\frac{1}{4}} \Delta(a) \Omega_\varphi = \Delta( \nabla^{\frac{1}{4}}a \Omega_\varphi)$. That is, $\Delta$ on the $L_2$- and $L_\infty$-level agree under the symmetric correspondence. Finally, we assume that  for every $t \geq 0$ we have that $\Phi_t(\cA) \subseteq \cA$ and that $(\Phi_t)_{t \geq 0}$ is norm continuous on $\cA$. The latter implies that on the norm closure of $\cA$ we have that  $(\Phi_t)_{t \geq 0}$ is a C$^\ast$-Markov semi-group and fits in the framework of \cite{CiprianiSauvageot}.
   	\end{dfn}

   \begin{rmk}
   	We note here that if $\varphi$ is a trace $\tau$ then in \cite{DaviesLindsay} it was proved that $\Dom(\Delta^{\frac{1}{2}}) \cap \cM$ is a $\ast$-algebra which may serve to do analogous constructions as we do below. We work with the algebra $\cA$ however that is generally smaller in order to avoid some technicalities. In general we cannot guarantee the existence of such an algebra.  Our assumption on the existence of $\cA$ should also be compared to similar assumptions made in  \cite{JungeMei}.
   \end{rmk}

\begin{rmk}
	Suppose that $\bG$ is a compact quantum group. Let $\cA(\bG)$ be the $\ast$-algebra generated by all matrix coefficients $u_{i,j}^\alpha, \alpha \in \Irr(\bG)$. This algebra is well-known to be contained in the Tomita algebra $\cT_{\varphi}$ of the Haar state $\varphi$; in fact $\sigma^\varphi$ preserves the space spanned by $u_{i,j}^\alpha, 1 \leq i,j \leq n_\alpha$ for every single $\alpha \in \Irr(\bG)$. Now if $(\Phi_t)_{t \geq 0}$ is moreover a semi-group of central multipliers, i.e. $\Phi_t(u_{i,j}^{\alpha}) = c_{\alpha, t} u_{i,j}^{\alpha}$ for some constants $c_{\alpha, t} \in \mathbb{C}$ that form a semi-group, then it follows that $\cA = \cA(\bG)$ satisfies the criteria described above. Indeed, in this case one has $\Delta(u_{i,j}^\alpha) = c'_\alpha u_{i,j}^\alpha$ where $c'_\alpha$ is the derivative of $c_{\alpha, t}$ at $t = 0$ from which this is directly derived.
\end{rmk}

   \begin{dfn}
 For $x,y \in \mathcal{A}$ we define the {\bf gradient form}
\begin{equation}\label{Eqn=Gradient}
\langle x, y \rangle_{\Gamma} = \Delta(y)^\ast x + y^\ast \Delta(x) - \Delta(y^\ast x) \in \cA \subseteq \cM.
\end{equation}
Note that as $\cA$ is unital we have $\cA \otimes \cA \Omega_\varphi \cA = \cA \otimes \cA \sigma^{\varphi}_{-i/2}(\cA) \Omega_\varphi = \cA \otimes \cA  \Omega_\varphi$.
Further $\cA \otimes \cA \Omega_\varphi$ may be equipped with a (degenerate) inner product
\[
\langle a \otimes \xi, c \otimes \eta \rangle_{\partial} =
\frac{1}{2} \langle   \langle a, c \rangle_\Gamma  \xi, \eta \rangle.
\]
Quotienting out the degenerate part and taking the completion yields a Hilbert space $\cH_\partial$. The class of $a \otimes \xi$ will be denoted by $a \otimes_\partial \xi$.
We have contractive commuting actions (see below) determined by
\begin{equation}\label{Eqn=LeftRightAction}
x \cdot (a \potimes \xi ) = xa \otimes_\partial \xi - x \potimes a\xi, \qquad (a \potimes \xi) \cdot y = a \potimes \xi y,
\end{equation}
with $a,x,y \in \mathcal{A}$ and $\xi \in \cA \Omega_\varphi \cA = \cA \Omega_\varphi$.
\end{dfn}

The proof of the following lemma is taken from the arguments in \cite{CiprianiSauvageot}. Since our setup is non-tracial and we work with the algebra $\cA$ instead of the Dirichlet algebra of \cite{CiprianiSauvageot} we included a proof sketch.

\begin{lem}\label{Lem=ActionsContract}
	The operations  \eqref{Eqn=LeftRightAction} are (well-defined) contractive left and right actions of $\cA$ that moreover commute.
\end{lem}
\begin{proof}
	We first prove the statements for the left action.
	We need the auxiliary contractions $\cA \rightarrow \cM$,
	\[
	R_\varepsilon(\Delta) := \frac{1}{1 + \varepsilon \Delta } =  \int_{t \in \mathbb{R}_{\geq 0}} e^{-t}  \Phi_{\varepsilon t} \: dt, \qquad \Delta_\varepsilon :=  \frac{\Delta}{1 + \varepsilon \Delta } = \frac{1}{\varepsilon} \left(  1 - R_\varepsilon(\Delta) \right).
	\]
	We define an approximate gradient form by
	\begin{equation}\label{Eqn=GradientApprox}
	\langle x, y \rangle_{\Gamma, \varepsilon}
	=   \Delta_{\varepsilon}(y)^\ast x  +  y^\ast \Delta_{\varepsilon}(x) - \Delta_{\varepsilon}(y^\ast x).
	\end{equation}
	So that $\lim_{\varepsilon \searrow 0} \langle x, y \rangle_{\Gamma, \varepsilon} = \langle x, y \rangle_{\Gamma}$ weakly in $\cM$.  Exactly as in \cite[Lemma 3.1]{CiprianiSauvageot} one proves that the approximate gradient form \eqref{Eqn=GradientApprox} is positive definite and that the $\cM$-valued matrix $(\langle a_i, a_j \rangle_{\Gamma, \varepsilon})_{i,j}$ is positive. Then we further define on $\cA \odot \cA \Omega_\varphi$,
	\[
	\langle a \odot \xi, c \odot \eta \rangle_{\partial, \varepsilon} = \frac{1}{2}
	\langle \langle a, c \rangle_{\Gamma , \varepsilon} \xi, \eta \rangle.
	\]
	Troughout the rest of the proof let $a_i, b_i,   x, y \in \cA$.  	
	$R_{\varepsilon}(\Delta)$ has a   Stinespring dilation $W_\varepsilon: L_2(\cM) \rightarrow \cH_\varepsilon$ with representation $\pi_\varepsilon: \cA \rightarrow B(\cH_\varepsilon)$ satisfying $R_{\varepsilon}(x) = W_\varepsilon^\ast \pi_{\varepsilon}(x) W_\varepsilon$. Exactly as in \cite[Lemma 3.5]{CiprianiSauvageot} we get that,
	\[
	\begin{split}
	&  2 \Vert x \cdot \sum_{i} a_i \otimes b_i \Omega_{\varphi}     \Vert_{\partial, \varepsilon}^2	=2
	\Vert  \sum_{i} x a_i \otimes b_i \Omega_{\varphi}
	- \sum_{i} x \otimes a_i  b_i \Omega_{\varphi}
	\Vert_{\partial, \varepsilon}^2 \\
	= & \  \langle \sum_{i,j}
	 b_j^\ast
	\left(  \langle x a_i, x a_j   \rangle_{\Gamma, \varepsilon}
	+  a_j^\ast \langle   x,   x  \rangle_{\Gamma, \varepsilon} a_i
	-  a_j^\ast \langle x a_i, x   \rangle_{\Gamma, \varepsilon}
	- \langle x , x a_j   \rangle_{\Gamma, \varepsilon}  a_i
	\right) b_i \Omega_{\varphi},   \Omega_{\varphi}  \rangle \\
	=  & \langle \sum_{i,j}       b_j^\ast \left( W_\varepsilon a_j - \pi_{\varepsilon}(a_j) W_{\varepsilon}  \right)^\ast x^\ast x \left( W_\varepsilon a_i - \pi_{\varepsilon}(a_i) W_\varepsilon \right)  b_i  \Omega_{\varphi} ,  \Omega_{\varphi}   \rangle     \\
	\leq & \Vert x \Vert^2 \:    \langle \sum_{i,j}       b_j^\ast \left( W_\varepsilon a_j - \pi_{\varepsilon}(a_j) W_{\varepsilon}  \right)^\ast  \left( W_\varepsilon a_i - \pi_{\varepsilon}(a_i) W_\varepsilon \right)  b_i  \Omega_{\varphi},  \Omega_{\varphi}  \rangle.     		
	\end{split}
	\]
	And by the same argument backwards this yields therefore
	\[
	\Vert x \cdot \sum_{i} a_i \otimes b_i \Omega_{\varphi}    \Vert_{\partial, \varepsilon}^2 \leq \Vert x \Vert^2 	\Vert \sum_{i} a_i \otimes b_i \Omega_{\varphi}    \Vert_{\partial, \varepsilon}^2.
	\]
	Contractiveness of the left action then follows by taking the limit $\varepsilon \searrow 0$.
	Next, for the right action  we get
	\[
	\begin{split}
	& \Vert \sum_{i} a_i \otimes b_i \Omega_{\varphi}   y \Vert_\partial^2
	=  \frac{1}{2} \langle  \sum_{i,j} \langle a_i, a_j \rangle_{\Gamma} b_i \Omega_{\varphi}   y,  b_j \Omega_{\varphi}   y \rangle \\
	 	\leq &  \frac{1}{2}  \Vert y \Vert^2 \:  \langle \sum_{i,j}   b_j^\ast \langle a_i, a_j \rangle_{\Gamma} b_i \Omega_{\varphi},      \Omega_{\varphi}   \rangle  =  \Vert y \Vert^2	\Vert \sum_{i} a_i \otimes b_i \Omega_{\varphi}    \Vert_\partial^2.
	\end{split}
	\]
	So the right action is contractive. Clearly the left and right action commute.
\end{proof}

\begin{rmk}
	By Lemma \ref{Lem=ActionsContract} the left and right action of $\cA$ extend to the C$^\ast$-closure of $\cA$. It is not  clear to us whether these actions are normal in general and hence extend to actions on the von Neumann closure of $\cA$. However, in the bimodules that we require to prove our main theorem this is true, see Proposition \ref{Prop=VNAModule} below.
\end{rmk}

\begin{rmk}
	Throughout the paper we shall often use the fact that for $x,a,c \in \cA, \xi, \eta \in \cA \Omega_\varphi$ we have,
	\begin{equation}\label{Eqn=GradientInproduct}
	\begin{split}
	 \langle x \cdot (a \potimes \xi), c \potimes \eta\rangle_\partial = &
	\langle
	xa \potimes \xi - x \potimes a \xi, c \potimes \eta
	\rangle_\partial \\
	= & \frac{1}{2}
	\langle  (c^\ast \Delta(xa) + \Delta(c)^\ast xa - \Delta(c^\ast xa)  -
	c^\ast \Delta(x) a - \Delta(c)^\ast xa + \Delta(c^\ast x) a) \xi, \eta \rangle\\
	= &\frac{1}{2}
	\langle  (c^\ast \Delta(xa)  - \Delta(c^\ast xa)  -
	c^\ast \Delta(x) a   + \Delta(c^\ast x) a) \xi, \eta \rangle.
	\end{split}
	\end{equation}
 \end{rmk}

\begin{prop}\label{Prop=VNAModule}
Let $\bG$ be a finitely generated compact quantum group and let $(\Phi_t)_{t \geq 0}$ be a  Markov semi-group of central multipliers. That is, for every  $t >0$ and $\alpha \in \Irr(\bG)$  there is a $c_{\alpha,t} \in \mathbb{C}$ such that for all $1 \leq i,j \leq n_\alpha$ we have  $\Phi_t(u^\alpha_{i,j}) = c_{\alpha, t} u^\alpha_{i,j}$. Let $\cA = \cA(\bG)$.  The associated $\cA$-$\cA$-bimodule $\cH_\partial$ constructed above extends to a (normal) $L_\infty(\bG)$-$L_\infty(\bG)$-bimodule.
\end{prop}
\begin{proof}
	It suffices to show that the left and right action are $\sigma$-weakly continuous on the unit ball.  Let $a,b \in \cA(\bG)$ and assume moreover that they are coefficients of irreducible representations with length $l(a)$ and $l(b)$ respectively (see Section \ref{Sect=Preliminaries}). Consider the mapping, c.f. \eqref{Eqn=GradientInproduct},
		\begin{equation}\label{Eqn=Module}
	\begin{split}
	  \cA(\bG) \ni x \mapsto  & \langle  x \cdot  a \potimes b \Omega_\varphi  ,  a \potimes b \Omega_\varphi   \rangle_\partial = \frac{1}{2}
	    \langle  (a^\ast \Delta(xa) - \Delta(a^\ast x a)  - a^\ast \Delta(x) a   + \Delta(a^\ast x) a  ) b \Omega_\varphi, b \Omega_\varphi  \rangle.
	\end{split}
	\end{equation}
	Note that $\Delta(u^\alpha_{ij}) =  c_{\alpha}'  u^\alpha_{ij}$ with $c_{\alpha}'$ the derivative of $c_{\alpha,t}$ at $t = 0$.
	Therefore if $x$ is a coefficient of an irreducible representation with length $> 2l(a) + 2l(b)$ we get  that $\langle  x \cdot  a \potimes b \Omega_\varphi,  a \potimes b \Omega_\varphi \rangle = 0$. So that the mapping  \eqref{Eqn=Module} factors through the normal  projection   $L_\infty(\bG) \rightarrow L_\infty(\bG)$ that maps $u^\alpha_{i,j}$ to $\delta_{\alpha \leq 2l(a) + 2l(b)} u^\alpha_{i,j}$ which image is finite dimensional. Hence the functional \eqref{Eqn=Module} is normal.
	
	Now, let $(x_j)_j$ be a   net in the unit ball of $\cA(\bG)$ converging $\sigma$-weakly to $x \in \cA(\bG)$. Take vectors $\xi, \eta \in \cH_\partial$ arbitrary and let $\varepsilon > 0$. Let $\xi_0, \eta_0$ be vectors in the linear span of all vectors $a \potimes b \Omega_\varphi$ with $a,b$ as above with $\Vert \omega_{\xi,\eta} - \omega_{\xi_0, \eta_0} \Vert_{\partial} < \varepsilon$. By the previous paragraph and the polarization identity we may find $j_0$  such that for $j \geq j_0$ we have $\vert \langle (x_j - x) \cdot  \xi_0, \eta_0 \rangle \vert \leq \varepsilon$. Then also $\vert \langle (x_j - x) \xi, \eta \rangle \vert  \leq 2 \varepsilon$. This shows that the left action is $\sigma$-weakly continuous on the unit ball. For the right action the proof is similar.
\end{proof}

 \subsection{Derivations in the tracial case}
The constructions of Section \ref{Sect=Gradient} work for non-tracial $\varphi$. Now assume $\varphi$ is tracial, say $\varphi = \tau$.
Consider the linear map
\begin{equation}\label{Eqn=Derivation}
\partial: \mathcal{A} \rightarrow \cH_\partial: a \mapsto a \potimes \Omega_\tau.
\end{equation}
Because in the tracial case $a \Omega_\tau = \Omega_\tau a, a \in \cA$ we have for $a,b \in \cA$,
\[
\partial(ab) = ab \potimes \Omega_\tau = ab \potimes \Omega_\tau - a \potimes b \Omega_\tau - a \potimes \Omega_\tau b =
 a \partial(b) + \partial(a) b,
\]
 i.e. $\partial$ is a derivation. Moreover, as by conservativity of $\Delta$ we have  $\tau(\Delta(a^\ast a)) = \langle \Delta(a^\ast a) \Omega_\tau, \Omega_\tau \rangle = \langle  a^\ast a \Omega_\tau, \Delta \Omega_\tau \rangle = 0$ and we see that,
 \begin{equation}\label{Eqn=RootComp}
 \Vert \partial(a) \Vert_{\partial}^2 =
 \frac{1}{2} \left( \tau(\Delta(a)^\ast a ) + \tau(a^\ast \Delta(a)   )  - \tau(\Delta(a^\ast a))  \right)
 =
  \frac{1}{2} \left( \tau(\Delta(a)^\ast a ) + \tau(a^\ast \Delta(a)   )  \right)
 =
 \Vert \Delta^{\frac{1}{2}}(a) \Vert_2^2.
 \end{equation}
  In \cite[Section 4]{CiprianiSauvageot} it is proved that there exists a closable derivation $\partial_0: \Dom(\Delta^{\frac{1}{2}}) \rightarrow \cH_\partial$ such that $\partial_0^\ast \overline{\partial_0} = \Delta$ (so with equality of domains). By construction $\partial \subseteq \partial_0$  and so $\partial$ is preclosed  and we let   $\overline{\partial}$ be its closure. If $\cA\Omega_\tau$ is a core for $\Delta^{\frac{1}{2}}$ it follows from \eqref{Eqn=RootComp} that the $\Dom(\overline{\partial})$ contains the Dirichlet algebra of all $x \in \cM$ such that $x \Omega_\tau \in \Dom(\Delta^{\frac{1}{2}})$. So if $\cA$ is a core for $\Delta^{\frac{1}{2}}$ then the derivation $\overline{\partial}$ equals the closure of the derivation $\overline{\partial_0}$ constructed in \cite[Section 4]{CiprianiSauvageot}.

    In the cases we need it these conditions are satisfied.

   \begin{lem}\label{Lem=PartialClosure}
   	Let $(\Phi_t)_{t \geq 0}$ be a semi-group of central multipliers on a compact quantum group $\bG$ of Kac type (i.e. with tracial Haar state). Let $\Delta$ be the generator of $(\Phi_t^{(2)})_{t \geq 0}$ as before.  Then $\cA(\bG) \Omega_\tau$ is a core for $\Delta^{\frac{1}{2}}$.
   \end{lem}
   \begin{proof}
   	Let $p_\alpha$ be the isotypical projection of $L_2(\bG)$ onto the space of matrix coefficients of $\alpha \in \Irr(\bG)$. As $(\Phi_t)_{t \geq 0}$ are central multipliers there exist constants $\Delta_\alpha$ such that $\Delta p_\alpha \xi = \Delta_\alpha p_\alpha \xi$ for any $\xi \in L_2(\bG)$.    	
   	 Let $\xi \in \Dom(\Delta^{\frac{1}{2}})$. Then taking limits over increasing finite subsets $F \subseteq \Irr(\bG)$ we find $\sum_{\alpha \in F} p_\alpha \xi \rightarrow \xi$ and $\sum_{\alpha \in F} p_\alpha \Delta_\alpha^{\frac{1}{2}} \xi \rightarrow \Delta^{\frac{1}{2}} \xi$.
   \end{proof}

\begin{lem}\label{Lem=Real}
	The derivation \eqref{Eqn=Derivation} is real in the sense that for all $a,b,c \in \cA$ we have
	\[
	\langle \partial(a), \partial(b) c \rangle_{\partial} = \langle c^\ast \partial(b^\ast), \partial(a^\ast) \rangle_\partial.
	\]
\end{lem}
\begin{proof}
	We have,
	\[
	\begin{split}
	   & \langle \partial(a), \partial(b) c \rangle_{\partial}
	   = \langle a \otimes \Omega_\tau, b \otimes c \Omega_\tau \rangle_\partial
	   = \frac{1}{2} \tau \left(
	   c^\ast ( \Delta(b^\ast) a + b^\ast \Delta(a) - \Delta(b^\ast a)  )
	   \right).
	   \end{split}
	   \]
	   Using that $\tau(x^\ast \Delta(y)) = \tau(\Delta(x^\ast)  y)$ and that $\tau(\Delta(x)) = \langle x \Omega_\tau, \Delta(\Omega_\tau) \rangle = 0$ with $x,y \in \cA$ gives further,
	   \[
	   \begin{split}
	    \langle \partial(a), \partial(b) c \rangle_{\partial} = & \frac{1}{2} \tau \left(b^\ast  \Delta(a c^\ast ) + \Delta(c^\ast b^\ast) a - \Delta(c^\ast) b^\ast a   \right) =
	   \frac{1}{2} \tau \left(b^\ast  \Delta(a c^\ast ) + \Delta(c^\ast b^\ast) a - \Delta(c^\ast) b^\ast a   \right) \\
	   = & \frac{1}{2} \tau\left(  a \Delta(c^\ast b^\ast)  -  a \Delta(c^\ast) b^\ast - \Delta(a c^\ast b^\ast) + \Delta(a c^\ast) b^\ast \right) =
	  \langle  c^\ast \cdot (b^\ast \otimes \Omega_\tau), (a^\ast \otimes \Omega_\tau) \rangle_\partial
	    \\
	   = & \langle c^\ast \partial(b^\ast), \partial(a^\ast) \rangle_\partial.
	\end{split}
	\]
\end{proof}

\section{Coarse properties of the gradient bimodule: IGHS, GHS and GC}\label{Sect=IGHS}

In this section we study when the bimodule $\cH_\partial$ is weakly contained in the coarse bimodule.
We use all notation introduced in Sections \ref{Sect=Preliminaries} and \ref{Sect=Gradient}. In particular $\cM$ is a von Neumann algebra with fixed normal faithful state $\varphi$. We let $(\Phi_t)_{t \geq 0}$ be a Markov semi-group on $\cM$ and associate to it the generator $\Delta$, the algebra $\cA$, the Dirichlet form $Q$ and the gradient form $\langle \: , \: \rangle_\Gamma$. As $\cA$ is contained in $\cM$ it inherits the matrix norms of $\cM$ and therefore complete positivity of a map $\cA \rightarrow \cM$ is understood naturally as a map that sends positive operators to positive operators on each matrix level.

 We introduce three properties of semi-groups that are convenient in studying coarse properties of $\cH_\partial$.

\begin{dfn}\label{Dfn=IGHS}
	We call a Markov semi-group $(\Phi_t)_{t \geq 0}$ on a von Neumann algebra $\cM$ with fixed normal faithful state $\varphi$ {\bf immediately gradient Hilbert-Schmidt (IGHS)} if for every choice $a, b \in \cA$ we have that the following two properties hold:
\begin{itemize}
\item  For every $t > 0$   the map
	\begin{equation}\label{Eqn=PsiMappie}
	\Psi_t^{a,b}: x \mapsto  \Phi_t( \langle x a, b \rangle_\Gamma - \langle x, b \rangle_\Gamma a )
	\end{equation}
	extends to a Hilbert-Schmidt map $L_2(\cM) \rightarrow L_2(\cM)$ given by  $x \Omega_\varphi \mapsto   \Psi_t^{a,b}(x) \Omega_\varphi, x \in \cA$.
\item For $t =0$ the map \eqref{Eqn=PsiMappie} extends to a bounded map  $L_2(\cM) \rightarrow L_2(\cM)$ given by $x \Omega_\varphi \mapsto   \Psi_0^{a,b}(x) \Omega_\varphi, x \in \cA$.
\end{itemize}
We call $(\Phi_t)_{t \geq 0}$ {\bf gradient Hilbert-Schmidt (GHS)} if for $t = 0$ and any $a,b \in \cA$ the map \eqref{Eqn=PsiMappie} is Hilbert-Schmidt. 	
	We call  $(\Phi_t)_{t \geq 0}$  {\bf gradient coarse (GC)} if the left and right $\cA$-actions on $\cH_\partial$ extend to normal $\cM$-actions and  $\cH_\partial$ is weakly contained in the coarse bimodule of $\cM$.
\end{dfn}
 
Note that if $\Psi_0^{a,b} \in B(L_2(\cM))$ then $\Psi_t^{a,b} \in B(L_2(\cM)), t \geq 0$ and that $\Psi_t^{a,b} \rightarrow  \Psi_0^{a,b}$ strongly in $B(L_2(\cM))$ as $t \searrow 0$.
 
\begin{rmk}
We shall often make use of the fact that for $a,b, x \in \cA$,
  \begin{equation}
 	\Psi_0^{a,b}( x) = b^\ast \Delta(xa) - \Delta(b^\ast x a )  - b^\ast \Delta(x) a + \Delta(b^\ast x) a,
  \end{equation}
  \end{rmk}


\begin{lem}\label{Lem=Psi}
	For some $n \in \mathbb{N}$ let $a_{1}, \ldots, a_n, c_1, \ldots, c_n \in \cA$. Then for any $t \geq 0$ the map
	\begin{equation}\label{Eqn=PsiCirc}
	\Theta_t^\circ := \Theta_t^{\circ, a_{1}, \ldots, a_n}: x \mapsto   [ \Phi_t\left( \langle  x a_i, a_j   \rangle_{\Gamma} - \langle    x  , a_j   \rangle_{\Gamma} a_i \right)]_{i,j}, \qquad x \in \cA,
	\end{equation}
	is a completely positive map $\cA \rightarrow M_n(\cM)$. Set
	\begin{equation}\label{Eqn=Psi}
	\Theta_t := \Theta_t^{a_{1}, \ldots, a_n; c_1, \ldots, c_n}: x \mapsto \sum_{i,j = 1}^n    c_j^\ast \Phi_t\left( \langle  x a_i, a_j   \rangle_{\Gamma} - \langle    x  , a_j   \rangle_{\Gamma} a_i \right) c_i, \qquad x \in \cA.
	\end{equation}
	Then the mapping
	\begin{equation} \label{Eqn=AAmap}
	\cA \otimes \cA^{\op} \rightarrow \mathbb{C}: x \otimes y^{\op} \mapsto \langle \Theta_t(x) \Omega_\varphi y, \Omega_\varphi \rangle,
	\end{equation}
	is positive. Finally, if
	 $(\Phi_t)_{t \geq 0}$ is IGHS (resp. GHS) then for every $t > 0$  the map \eqref{Eqn=Psi} is Hilbert-Schmidt and converges strongly to $\Theta_0$ as $t \searrow 0$ (resp. for  $t = 0$ the map \eqref{Eqn=Psi} is Hilbert-Schmidt).
\end{lem}
\begin{proof}
	The fact that for any choice of the  $x, a_i, c_i \in \cA$ we have
	\[
	0 \leq
	\langle  x \cdot \sum_{i=1}^n  a_i \potimes  c_i   \Omega_\varphi, \sum_{i=1}^n  a_i \potimes  c_i  \Omega_\varphi   \rangle_\partial = \frac{1}{2}
     \langle \Theta^{\circ, a_1, \ldots, a_n}_0(x)
     \left(
     \begin{array}{c}
      c_1 \Omega_\varphi \\
      \vdots \\
      c_n \Omega_\varphi
     \end{array}
     \right)
     ,  \left(
     \begin{array}{c}
     c_1 \Omega_\varphi \\
     \vdots \\
     c_n \Omega_\varphi
     \end{array}
     \right) \rangle.
	\]
	shows that $\Theta^\circ_0$ is positive and the same argument on matrix levels gives complete positivity. Hence as $\Phi_t$ is completely positive also \eqref{Eqn=PsiCirc} must be  completely positive. 	
	 Let $x = (x_1, \ldots, x_n), c^\ast = (c_1^\ast, \ldots, c_n^\ast)$ be the row vectors with entries $x_i, c_i \in \cA$ and let again $a_i \in \cA$. Then $x^\ast x \in M_n(\cA)^+$ and
\[
(\id_n \otimes \Theta_t^{a_1, \ldots, a_n, c_1, \ldots, c_n})(x^\ast x) = (\id_n \otimes c\Theta_t^{\circ, a_1, \ldots, a_n}( \: \cdot \: ) c^\ast)(x^\ast x) \in M_n(\cM)^+.
\]
Further, recalling $\Theta_t := \Theta_t^{a_{1}, \ldots, a_n; c_1, \ldots, c_n}$,
	\[
	\sum_{k,l=1}^n
	\langle \Theta_t(x_k^\ast x_l) \Omega_\varphi y_l y_k^\ast, \Omega_\varphi \rangle
	=    \langle (\id_n \otimes \Theta_t)( x^\ast x )
		  \left(
	\begin{array}{c}
	 \Omega_\varphi y_1  \\
	\vdots \\
	 \Omega_\varphi y_n
	\end{array}
	\right),
	\left(
	\begin{array}{c}
	\Omega_\varphi y_1  \\
	\vdots \\
	\Omega_\varphi y_n
	\end{array}
	\right)
	\rangle \geq 0,
	\]
	so that \eqref{Eqn=AAmap} is positive. 		
	The final statement follows as if the semi-group is IGHS, then
	\[
	\begin{split}
	\mathcal{A} \Omega_\varphi \ni x \Omega_\varphi \mapsto \Theta_t(x) \Omega_\varphi =&
	 \sum_{i,j = 1}^n    c_j^\ast \Phi_t\left( \langle  x a_i, a_j   \rangle_{\Gamma} - \langle    x  , a_j   \rangle_{\Gamma} a_i \right) \Omega_\varphi \sigma_{i/2}(c_i) \\
	 = & \sum_{i,j = 1}^n    c_j^\ast\left( \Psi_t^{a_i, a_j} (x) \Omega_\varphi  \right) \sigma_{i/2}(c_i), \\
	  	 \end{split}
	\]
	 is Hilbert-Schmidt for $t > 0$ by linearity and bounded if $t=0$. Further $\Theta_t \rightarrow \Theta_0$ strongly as $t \searrow 0$.  The statement for GHS follows similarly.
\end{proof}

\begin{prop}\label{Prop=IGHSimpliesGC}
	Assume that the left and right $\cA$-actions on $\cH_\partial$ extend to normal $\cM$-actions.
	If $(\Phi_t)_{t \geq 0}$ is IGHS or GHS  then it is GC.
\end{prop}
\begin{proof}
	We give the proof for the IGHS assumption; for the GHS assumption the proof is similar and in fact easier.
		Throughout the proof fix $a_1, \ldots,  a_n, c_1, \ldots, c_n \in \cA$ and for $t \geq 0$ let $\Theta_t := \Theta_t^{ a_1, \ldots, a_n, c_1, \ldots, c_n }$ be the map defined in \eqref{Eqn=Psi}. Set $\Theta = \Theta_0$. For $x,y \in \cA$ we get,
	\[
	\begin{split}
	& \langle x \cdot ( \sum_{i=1}^n a_i \otimes_{\partial} c_i \Omega_\varphi ) \cdot y, \sum_{j=1}^n a_j \otimes_{\partial} c_j \Omega_\varphi  \rangle_\partial
	=
	\langle  \sum_{i=1}^n x a_i \otimes_{\partial} c_i\Omega_\varphi y - x  \otimes_{\partial} a_i c_i \Omega_\varphi y, \sum_{j=1}^n a_j \otimes_{\partial} c_j \Omega_\varphi  \rangle_\partial \\
	= & \frac{1}{2} \langle
	\sum_{i,j=1}^n  c_j^\ast  ( \langle x a_i, a_j \rangle_{\Gamma}   - \langle x, a_j \rangle_{\Gamma}  a_i )c_i  \Omega_\varphi, \Omega_\varphi y^\ast \rangle     =
	\frac{1}{2}  \langle \Theta(x) \Omega_\varphi, \Omega_\varphi    y^\ast \rangle.
	\end{split}
	\]
If $x,y \in \cM$ are arbitrary we may approximate them using Kaplansky's density theorem in the strong topology with bounded nets $(x_k)_k$ and $(y_k)_k$ in $\mathcal{A}$. Then $x_k \rightarrow x$ in the $\sigma$-weak topology and $x_k \Omega_\varphi \rightarrow x\Omega_\varphi$ in the norm of $L_2(\cM)$. Similarly $y_k \rightarrow y$ $\sigma$-weakly and $\Omega_\varphi y_k^\ast = J y_k \Omega_\varphi \rightarrow J y \Omega_\varphi = \Omega_\varphi y^\ast$ in norm.
The left and right $\cM$-action on $H_\partial$ are normal and the IGHS assumption gives that $\Theta$ is  bounded $L_2(\cM) \rightarrow L_2(\cM)$. We thus see that
\[
\begin{split}
 \langle x \cdot ( \sum_{i=1}^n a_i \otimes_{\partial} c_i \Omega_\varphi ) \cdot y, \sum_{j=1}^n a_j \otimes_{\partial} c_j \Omega_\varphi  \rangle_\partial = &
\lim_{k_1, k_2} \langle x_{k_1} \cdot ( \sum_{i=1}^n a_i \otimes_{\partial} c_i \Omega_\varphi ) \cdot y_{k_2}, \sum_{j=1}^n a_j \otimes_{\partial} c_j \Omega_\varphi  \rangle_\partial \\
	= & \lim_{k_1, k_2}  \frac{1}{2}  \langle \Theta(x_{k_1}) \Omega_\varphi, \Omega_\varphi    y^\ast_{k_2} \rangle
	= \frac{1}{2}  \langle \Theta(x) \Omega_\varphi, \Omega_\varphi    y^\ast \rangle.
\end{split}
\]
In turn we find by the IGHS assumption that for all $x,y \in \cM$,
\[
\langle x \cdot ( \sum_{i=1}^n a_i \otimes_{\partial} c_i \Omega_\varphi ) \cdot y, \sum_{j=1}^n a_j \otimes_{\partial} c_j \Omega_\varphi  \rangle_\partial
= \frac{1}{2} \lim_{t \searrow 0} \langle \Theta_t(x) \Omega_\varphi, \Omega_\varphi    y^\ast \rangle.
\]
	By the IGHS assumption for $t > 0$ the map  $\Theta_t$ is bounded $L_2(\cM) \rightarrow L_2(\cM)$ and moreover Hilbert-Schmidt by Lemma  \ref{Lem=Psi} and therefore we see that there exists a vector $\zeta_t  \in L_2(\cM) \otimes \overline{L_2(\cM)}$ such that,
	\[
	\langle \Theta_t   (  x) \Omega_{\varphi}, \Omega_{\varphi} y^\ast \rangle =
	\langle  x\Omega_\varphi \otimes   \overline{ \Omega_\varphi y^\ast},  \zeta_t  \rangle.
		\]
	This shows that for $t>0$ we have that
	\begin{equation}\label{Eqn=M}
	\cM \odot \cM^{{\rm op}} \ni x \otimes y \mapsto \frac{1}{2}  \langle \Theta_t(x) \Omega_\varphi y, \Omega_\varphi      \rangle
	\end{equation}
	 extends to a normal functional on $\cM \otimes \cM^{{\rm op}}$. Moreover,
		we see from Lemma \ref{Lem=Psi}   that \eqref{Eqn=M} is positive on $\cA \odot \cA^{\op}$ and hence by Kaplansky on $\cM \otimes \cM^{{\rm op}}$. 	
	 Now as $L_2(\cM) \otimes L_2(\cM)$ is the standard form Hilbert space of $\cM \otimes \cM^{{\rm op}}$ there exists $\zeta_t' \in L_2(\cM) \otimes L_2(\cM)$, still with $t > 0$, with
	\[
	\frac{1}{2}   \langle \Theta_t(x) \Omega_\varphi, \Omega_\varphi    y \rangle = \langle x \zeta_t' y, \zeta_t' \rangle.
	\]	
	Therefore, for every $x,y \in \cM$ we have
	\begin{equation}\label{Eqn=ApproxZeta}
	\langle x \cdot ( \sum_{i=1}^n a_i \otimes_{\partial} c_i) \cdot y, \sum_{i=1}^n a_i \otimes_{\partial} c_i  \rangle
	= \lim_{t \searrow 0}
	\langle  x \zeta_t' y, \zeta_t' \rangle.
	\end{equation}
We can now directly check that $\cH_\partial$ is weakly contained in the coarse bimodule of $\cM$. Indeed, let $\xi \in \cH_\partial, \varepsilon >0$ and  let $F \subseteq \cM$ be a finite subset.  Assume that $ \xi = \sum_{i=1}^n a_i \otimes_{\partial} c_i \Omega_\varphi$. Then by \eqref{Eqn=ApproxZeta} we may find $t>0$ such that for all $x,y \in F$ we have
\begin{equation}\label{Eqn=WeakCon}
\vert
\langle x \xi y, \xi \rangle - \langle  x \zeta_t' y, \zeta_t' \rangle \vert < \varepsilon.
\end{equation}
Then by approximation we find that for general $\xi \in \cH_\partial$ we can find $t>0$ such that for all $x, y \in F$ the estimate \eqref{Eqn=WeakCon} holds. We see by Definition \ref{Dfn=WeakContainment} that $\cH_\partial$ is weakly contained in the coarse bimodule of $\cM$.
\end{proof}

\section{Stability properties}\label{Sect=IGHSstable}

We prove that IGHS and GHS are properties that are preserved by free products. We also prove the necessary reduction to continuous cores.

\subsection{Free products}
For the definition of free products of von Neumann algebras we refer to \cite{Avitzour}, \cite{Voiculescu}. We also refer to \cite{CaspersFima} and adopt  its notation and terminology.
Let $(\cM_i, \varphi_i), i \in I$ be von Neumann algebras with normal faithful states $\varphi_i$. The free product $(\cM, \varphi)$ is the von Neumann algebra with normal faithful state $\varphi$ that contains each $\cM_i, i \in I$ as a subalgebra to which $\varphi$ restricts as $\varphi_i$; moreover, these algebras are freely independent in $\cM$ with respect to $\varphi$ and generate $\cM$.
Set $\cM_i^\circ$ to be the set of all $x \in \cM_i$ with $\varphi_i(x) = 0$. For $x \in \cM_i$ we set $x^\circ =  x - \varphi_i(x)$. A reduced operator in the free product   $(\cM, \varphi) = \ast_{i \in I} (\cM_i, \varphi_i)$ is an operator of the form $x_1 \ldots x_n$ with $x_i \in \cM_{X_i}^\circ$ for some $X_i \in I$ with $X_i \not = X_{i+1}$.  The word $X = X_1 \ldots X_n$ is called the type of $x_1 \ldots x_n$.
If $\Phi_i$ is a normal  $\varphi_i$-preserving ucp map on $\cM_i$ (i.e. it is Markov with respect to $\varphi_i$) then there exists a unique normal $\varphi$-preserving ucp map $\ast_{i \in I} \Phi_i$ on the free product $(\cM, \varphi)$ such that for a reduced operator $x_1 \ldots x_n$ with $x_k \in \cM_{i_k}^{\circ}$ we have $\Phi(x_1 \ldots x_n) = \Phi_{i_1}(x_1) \ldots \Phi_{i_n}(x_n)$. If $(\Phi_{i,t})_{t \geq 0}$ are Markov semi-groups on $\cM_i, i \in I$ then the maps $\Phi_t = \Phi_{i,t}, t \geq 0$ form a Markov semi-group on $\cM$.

	Let $\Delta_i$ be the generator of $(\Phi_{i,t})_{t \geq 0}$ and let $\cA_i$ be the dense unital subalgebras in $\cM_i$ as described in Section \ref{Sect=Preliminaries}. Let $\Delta$ be the generator of $(\Phi_t)_{t \geq 0}$.  Let $a_1 \ldots a_n$ be a reduced operator of type $A$ in the algebraic free product $\cA = \ast_{i \in I} \cA_i$. Then by taking $\sigma$-weak limits (which exists on these reduced operators),  we obtain the following Leibniz rule,
	\begin{equation}\label{Eqn=Leipniz}
	\begin{split}
	\Delta(a_1 \ldots a_n) = & \lim_{t \searrow 0} \frac{1}{t} \left(  a_1 \ldots a_n -\Phi_t (a_1 \ldots a_n)     \right) \\
	= & \sum_{i = 1}^n \lim_{t \searrow 0} \frac{1}{t} \left(  \Phi_{A_1,t} (a_1) \ldots  \Phi_{A_{i-1},t} (a_{i-1})  a_i    \ldots a_n -   \Phi_{A_1,t} (a_1) \ldots  \Phi_{A_{i},t} (a_{i})   a_{i+1} \ldots a_n      \right) \\
	= &  \sum_{i=1}^n a_1 \ldots a_{i-1} \Delta_{A_i}(a_i) a_{i+1} \ldots a_n.
	\end{split}
	\end{equation}
 A rather tedious computation purely based on this Leibniz rule now shows the following.

\begin{prop}\label{Prop=FreeProduct}
	Let $(\cM_1, \varphi_1), \ldots, (\cM_n, \varphi_n)$ be finitely many von Neumann algebras with normal faithful states. Suppose that each $(\cM_i, \varphi_i)$ is equipped with a Markov semi-group $(\Phi_{i,t})_{t \geq 0}$ and let $(\Phi_t)_{t\geq 0}$ be the free product Markov semi-group on the free product
	\[
	   (\cM_1, \varphi_1) \ast \ldots \ast (\cM_n, \varphi_n).
	\]
	If each $(\Phi_{i,t})_{t \geq 0}$ is IGHS (resp. GHS) then $(\Phi_t)_{t\geq 0}$  is IGHS (resp. GHS).
\end{prop}
\begin{proof} The proof splits in steps.
	
		\vspace{0.3cm}

\noindent {\bf 1. Setup: expansion into reduced words.} 	
Let $\cA_i$ and   $\cA = \ast_i \cA_i$ as in the paragraph before this proposition. In particular the unit is in $\cA_i$  so that $x^\circ \in \cA_i$ whenever $x \in \cA_i$.  Let $\varphi = \ast_{i} \varphi_i$ be the free product state. For each $i$ we let $O_i$ be a set of vectors in $\cA_i^\circ$ such that $O_i\Omega_{\varphi_i}$  forms an  orthonormal basis of $L_2(\cM_i^\circ)$.
		Take $x \in \cA$ equal to a reduced word $x = x_1 \ldots x_n$ with letters in the $\cA_i$'s.
	Also assume that both $a,b \in \cA$ are reduced words $a = a_1 \ldots a_m$ and $b= b_1 \ldots b_k$ with letters in the $\cA_i$'s. We assume moreover that all letters $a_i, b_i$ and $x_i$ come from $\cup_j O_j$.  Let $A,B$ and $X$ be the respective types of $a,b$ and $x$. To reduce the number of cases we need to consider in this proof we extend our notation as introduced above a bit.
	We shall write
	\[
	\overbrace{xy}^\circ = xy - \varphi(xy), \qquad x,y \in \cup_{i} \cA_i.
	\]
	In particular, if $x \in O_i$ and $y \in O_j$ then $\overbrace{xy}^\circ = xy$ if $i \not = j$ (this extends the notation). In case $i \not = j$ we have $\Delta(xy) = \Delta(x) y  + x \Delta(y)$ by the Leibniz rule \eqref{Eqn=Leipniz}.
	  If we start writing $b x a$  as a sum of reduced operators we find the following terms,
	\begin{equation}\label{Eqn=SumDec}
	\begin{split}
	bxa = & 	\sum_{i=0}^{n-1} \sum_{j=i+1}^m \bigg(
	\varphi(b_k x_1) \ldots \varphi(b_{k-i+2}  x_{i-1} )
	\varphi(x_n a_1) \ldots \varphi(x_{j+1}  a_{n-j} ) \\
	&	\qquad \times \quad
	b_1 \ldots b_{k-i} \overbrace{ (b_{k-i+1} x_{i} )}^{\circ} x_{i+1}\ldots  x_{j-1}  \overbrace{  ( x_{j}  a_{n-j+1} )  }^{\circ} a_{n-j+2} \ldots a_{m} \bigg) \\
	+ & 	\sum_{ i=1 }^{n} \bigg(
	\varphi(b_k x_1) \ldots \varphi(b_{k-i+2}  x_{i-1} )
	\varphi(x_n a_1) \ldots \varphi(x_{i+1}  a_{n-i} ) \\
	&	\qquad \times \quad
	b_1 \ldots b_{k-i} (b_{k-i+1} x_{i}  a_{n-i+1})  a_{n-i+2} \ldots a_{m} \bigg) \\
	= & {\rm I} + {\rm II},
	\end{split}
	\end{equation}
	where we define I and II as the big sums.
	We use the convention that $a_j = 0$ if $j >m$ and $b_j = 0$ if $j > k$. Also note that many of these terms are 0, for example if $x_1 \in O_i$ and $b_k \in O_j$ with $i \not = j$ we have that $\varphi(b_k x_1) = 0$.
	The summands in I are reduced operators, the summands in II are not necessarily reduced for the reason that $b_{k-i+1} x_{i}  a_{n-i+1}$ is not necessarily reduced. In order to treat this summand we continue our expansion into three sums and a remainder part $F(x)$. We find that,
	\begin{equation}\label{Eqn=IIexpression}
	\begin{split}
	   {\rm II}  = & 	\sum_{\substack{i=1\\ B_{k-i+1} = X_{i} \not = A_{n-i+1}} }^{n}
	 	\bigg( \varphi(b_k x_1) \ldots \varphi(b_{k-i+2}  x_{i-1} )
	 \varphi(x_n a_1) \ldots \varphi(x_{i+1}  a_{n-i} )  \\
	   & \qquad \times \quad
	  b_1 \ldots b_{k-i} \overbrace{ (b_{k-i+1} x_{i})}^\circ  a_{n-i+1}   a_{n-i+2} \ldots a_{m} \bigg) \\
	   & + 	\sum_{\substack{i=1\\  B_{k-i+1} \not = X_{i}  = A_{n-i+1}  } }^{n} \bigg(
	  \varphi(b_k x_1) \ldots \varphi(b_{k-i+2}  x_{i-1} )\\
	   & \qquad \times \quad
	   b_1 \ldots b_{k-i} b_{k-i+1}  \overbrace{ (x_{i}  a_{n-i+1}) }^\circ     a_{n-i+2} \ldots a_{m} \bigg) \\
	   & + 	\sum_{\substack{i=1\\ B_{k-i+1} = X_{i} = A_{n-i+1}  } }^{n} \bigg(
	   	\varphi(b_k x_1) \ldots \varphi(b_{k-i+2}  x_{i-1} )
	   \varphi(x_n a_1) \ldots \varphi(x_{i+1}  a_{n-i} )   \\
	   & \qquad \times \quad
	    b_1 \ldots b_{k-i} \overbrace{ (b_{k-i+1}  x_{i}  a_{n-i+1} ) }^\circ     a_{n-i+2} \ldots a_{m} \bigg) \\
	   & + F(x),
	\end{split}
	\end{equation}
	where $F: \cM \rightarrow \cM$ is the finite rank operator that collects the remaining terms of II; that is, $F(x)$ is given by the same expression \eqref{Eqn=IIexpression} but with the operation $\overbrace{ \: \cdot \:}^\circ$ replaced by taking $\varphi( \: \cdot \: )$.
	
	\vspace{0.3cm}

\noindent 	{\bf 2. Appyling the $\Psi$-map.}
	    Now we apply $\Psi^{a,b^\ast}_t$ for $t=0$ to $x$ (we prefer  $\Psi^{a,b^\ast}_t$ over $\Psi^{a,b}_t$ to keep the notation simpler; for the proof it is irrelevant). Recall that,
	    \begin{equation}\label{Eqn=LongComputation}
	    \Psi^{a,b^\ast}_0(x) =   b \Delta(xa) - \Delta(b x a) - b \Delta(x) a + \Delta(b x) a.
	    \end{equation}
	    We proceed by expanding the right hand side of this expression into a decomposition very similar to \eqref{Eqn=SumDec} and \eqref{Eqn=IIexpression}. If we do this we get the following, where   the respective terms II$_{b\Delta(xa)}$, II$_{\Delta(bxa)}$, II$_{b\Delta(x)a}$ and II$_{\Delta(bx)a}$ are described below. Write $\Delta_k^l$ for $\Delta$ if $k =l$ and for the identity operator otherwise. So,
	    \[
	    \begin{split}
	    b \Delta(xa) & = 	\sum_{i=0}^{n-1} \sum_{j=i+1}^m  \sum_{l=1}^{m-n+1-1}
	    \varphi(b_k \Delta_{l}^1(x_1)) \ldots \varphi(b_{k-i+2}  \Delta_{l}^{i-1}(x_{i-1}) )
	    \varphi(x_n a_1) \ldots \varphi(x_{j+1}  a_{n-j} ) \\
	    &	\qquad \times \quad
	    b_1 \ldots b_{k-i} \overbrace{b_{k-i+1}\Delta^i_l( x_{i})}^{\circ} \Delta^{i+1}_l(x_{i+1})\ldots  \Delta^{j-1}_l(x_{j-1}) \\
	    & \qquad \times \quad \Delta^j_l( \overbrace{ x_{j}  a_{n-j+1}   }^{\circ} )  \Delta_l^{j+1}(a_{n-j+2}) \ldots \Delta_l^{m-n+1-1}(a_{m})  + {\rm II}_{b \Delta(xa) },\\
	    \Delta(b xa) & = 	\sum_{i=0}^{n-1} \sum_{j=i+1}^m  \sum_{l=1}^{k+2j+m-n-2i+1}
	    \varphi(b_k x_1) \ldots \varphi(b_{k-i+2}  x_{i-1} )
	    \varphi(x_n a_1) \ldots \varphi(x_{j+1}  a_{n-j} ) \\
	    &	\qquad \times \quad
	    \Delta_l^1(b_1) \ldots \Delta_l^{k-i}(b_{k-i})  \Delta_l^{k-i+1}( \overbrace{b_{k-i+1} x_{i}}^{\circ}) \Delta_l^{k-i+2}( x_{i+1}) \ldots  \Delta_l^{k+j-2i}(x_{j-1}) \\
	    & \qquad \times \quad \Delta_k^{k+j-2i+1} (\overbrace{   x_{j}  a_{n-j+1}   }^{\circ}) \Delta_l^{k+j-2i+2}( a_{n-j+2}) \ldots \Delta_l^{k+2j+m-n-2i+1}(a_{m}) + {\rm II}_{\Delta(bxa)},\\
	    b \Delta(x)a & = 	\sum_{i=0}^{n-1} \sum_{j=i+1}^m \sum_{l=1}^{n}
	    \varphi(b_k \Delta_l^1(x_1)  ) \ldots \varphi(b_{k-i+2} \Delta_l^{i-1}( x_{i-1}) )
	    \varphi(\Delta_l^n(x_n) a_1) \ldots \varphi( \Delta_l^{j+1}(x_{j+1})  a_{n-j} ) \\
	    &	\qquad \times \quad
	    b_1 \ldots  b_{k-i}  \overbrace{  b_{k-i+1}\Delta_l^i(x_{i} )}^{\circ} \Delta_l^{i+1}( x_{i+1} )\ldots  \Delta_l^{j-1}(x_{j-1})  \overbrace{  \Delta_l^j( x_{j} ) a_{n-j+1}   }^{\circ}  a_{n-j+2}\ldots a_{m} + {\rm II}_{b\Delta(x) a}, \\
	    \Delta(b x)a & = 		\sum_{i=0}^{n-1} \sum_{j=i+1}^m  \sum_{l}^{k+j+n-2i+4}
	    \varphi(b_k x_1) \ldots \varphi(b_{k-i+2}  x_{i-1} )
	    \varphi(\Delta_l^{k+j+n-2i+3}(x_n) a_1) \ldots \varphi( \Delta_l^{k+j-2i+3}(x_{j+1})  a_{n-j} ) \\
	    &	\qquad \times \quad
	    \Delta_l^1(b_1) \ldots  \Delta_l^{k-i}(b_{k-i})  \Delta_l^{k-i+1}( \overbrace{  b_{k-i+1} x_{i}}^{\circ}) \Delta_l^{k-i+2}( x_{i+1} ) \ldots  \Delta_l^{k+j-2i+1}(  x_{j-1} ) \\
	    & \qquad \times \quad \overbrace{  \Delta_l^{k+j-2i+2}( x_{j} ) a_{n-j+1}   }^{\circ}  a_{n-j+2}\ldots a_{m} + {\rm II}_{\Delta(bx) a}.
	    \end{split}
	    \]
	    Therefore, as all these terms cancel,
	    \[
	       \widetilde{\Psi}^{a,b^\ast}_0(x) =   {\rm II}_{b \Delta(xa) } - {\rm II}_{\Delta(bxa)} -{\rm II}_{b\Delta(x) a} + {\rm II}_{\Delta(bx) a}.
	    \]	
	    For the `II-terms' we get the following. Again, we split this into a decomposition similar to \eqref{Eqn=IIexpression}. We get that
	    \[
	    \begin{split}
	    		 {\rm II}_{b \Delta(xa) } = & {\rm II}_{b \Delta(xa) }^{(1)} +   {\rm II}_{b \Delta(xa) }^{(2)}  +  {\rm II}_{b \Delta(xa) }^{(3)} + F_{b,a, 1 }(x),\\
	    		  {\rm II}_{  \Delta(b xa)  } = & {\rm II}_{  \Delta(b xa)  }^{(1)} +   {\rm II}_{ \Delta(b xa) }^{(2)}  +  {\rm II}_{  \Delta(b xa)  }^{(3)} +  F_{b,a, 2 }(x),\\
	    		   {\rm II}_{b \Delta(x) a } = & {\rm II}_{b \Delta(x)a }^{(1)} +   {\rm II}_{b \Delta(x)a }^{(2)}  +  {\rm II}_{b \Delta(x)a }^{(3)} + F_{b,a, 3 }(x),\\
	    		  {\rm II}_{  \Delta(b xa)  } = & {\rm II}_{   \Delta(b x)a  }^{(1)} +   {\rm II}_{ \Delta(b x) a }^{(2)}  +  {\rm II}_{  \Delta(b x)a  }^{(3)} + F_{b,a, 4 }(x),
	    \end{split}		
	    \]
	    where the $F_{b,a,i}$'s are  finite rank maps $\cM \rightarrow \cM$ and the II$^{(1)}$, II$^{(2)}$ and II$^{(3)}$ terms are specified below. Let us first examine the II$^{(1)}$-terms. We get that,
	    \[
	    \begin{split}
	      {\rm II}_{b \Delta(xa) }^{(1)} = &	\sum_{\substack{i=1\\ B_{k-i+1} = X_{i} \not = A_{n-i+1}} }^{n}    \sum_{l=1}^{2i+1+m-n}
	    \varphi(b_k \Delta_l^{1}(x_1)) \ldots \varphi(b_{k-i+2}  \Delta_l^{i-1}(x_{i-1}) )
	    \varphi(x_n a_1) \ldots \varphi(x_{i+1}  a_{n-i} )  \\
	    & \qquad \times \quad
	    b_1 \ldots b_{k-i}   \overbrace{b_{k-i+1} \Delta_l^{i}(x_{i})}^\circ  \Delta_l^{i+1}(a_{n-i+1})     \ldots \Delta_l^{2i+1+m-n}(a_{m})\\
	   {\rm II}_{\Delta(b xa) }^{(1)}= &	\sum_{\substack{i=1\\ B_{k-i+1} = X_{i} \not = A_{n-i+1}} }^{n}    \sum_{l=1}^{k+m-n+2}
	    \varphi(b_k x_1) \ldots \varphi(b_{k-i+2} x_{i-1} )
	    \varphi(x_n a_1) \ldots \varphi(x_{i+1}  a_{n-i} )  \\
	    & \qquad \times \quad
	    \Delta_l^{1}(b_1)^{(1)} \ldots \Delta_l^{k-i}(b_{k-i}) \Delta_l^{k-i+1}(   \overbrace{b_{k-i+1} x_{i}}^\circ)  \Delta_l^{k-i+2}(a_{n-i+1})   \ldots \Delta_l^{k+m-n+2}(a_{m})\\
	    {\rm II}_{b \Delta(x)a }^{(1)}= &	\sum_{\substack{i=1\\ B_{k-i+1} = X_{i} \not = A_{n-i+1}} }^{n}    \sum_{l=1}^{n}
	    \varphi(b_k \Delta_l^{1}(x_1) ) \ldots \varphi(b_{k-i+2} \Delta_l^{i-1}(x_{i-1}) )
	    \varphi(\Delta_l^{n}(x_n) a_1) \ldots \varphi( \Delta_l^{i+1}(x_{i+1})  a_{n-i} )  \\
	    & \qquad \times \quad
	    b_1 \ldots b_{k-i}  \overbrace{b_{k-i+1} \Delta_l^{i}( x_{i})}^\circ   a_{n-i+1}   \ldots  a_{m},\\
	     {\rm II}_{ \Delta(b x)a }^{(1)} = &	\sum_{\substack{i=1\\ B_{k-i+1} = X_{i} \not = A_{n-i+1}} }^{n}   \sum_{l=1}^{k+n-2i+1}
	    \varphi(  b_k x_1 ) \ldots \varphi(b_{k-i+2} x_{i-1} )
	    \varphi( \Delta_l^{k+n-2i+1}(x_n) a_1) \ldots \varphi( \Delta_l^{k-i+2}(x_{i+1})  a_{n-i} )  \\
	    & \qquad \times \quad
	    \Delta_1^l(b_1) \ldots \Delta_1^{k-i}(b_{k-i})  \Delta_l^{k-i+1}(\overbrace{  b_{k-i+1}  x_{i}}^\circ )  a_{n-i+1}   \ldots  a_{m},\\
	    \end{split}
	    \]
	    Again we see that,
	    \[
	    {\rm II}_{b \Delta(xa) }^{(1)} - {\rm II}_{\Delta(bxa)}^{(1)} - {\rm II}_{b\Delta(x) a}^{(1)} + {\rm II}_{\Delta(bx) a}^{(1)} = 0.
	    \]
	    Now for the II$^{(2)}$-terms we find,
	    \[
	    \begin{split}
	    {\rm II}_{b \Delta(xa) }^{(2)} = &	\sum_{\substack{i=1\\ B_{k-i+1} \not = X_{i}  = A_{n-i+1}} }^{n}    \sum_{l=1}^{2i+1+m-n}
	    \varphi(b_k \Delta_l^{1}(x_1)) \ldots \varphi(b_{k-i+2}  \Delta_l^{i-1}(x_{i-1}) )
	    \varphi(x_n a_1) \ldots \varphi(x_{i+1}  a_{n-i} )  \\
	    & \qquad \times \quad
	    b_1 \ldots  b_{k-i+1} \Delta_l^{i}( \overbrace{ x_{i}   a_{n-i+1} }^\circ ) \Delta_l^{i+1}( a_{n-i+2})
	        \ldots \Delta_l^{2i+1+m-n}(a_{m})\\
	    {\rm II}_{\Delta(b xa) }^{(2)}= &	\sum_{\substack{i=1\\ B_{k-i+1} \not = X_{i}  = A_{n-i+1}} }^{n}    \sum_{l=1}^{k+m-n+2}
	    \varphi(b_k x_1) \ldots \varphi(b_{k-i+2} x_{i-1} )
	    \varphi(x_n a_1) \ldots \varphi(x_{i+1}  a_{n-i} )  \\
	    & \qquad \times \quad
	    \Delta_l^{1}(b_1)  \ldots  \Delta_l^{k-i+1}(b_{k-i+1} )   \Delta_l^{k-i+2}( \overbrace{  x_{i}   a_{n-i+1}  }^\circ)  \Delta_l^{k-i+3}(a_{n-i+2})   \ldots \Delta_l^{k+m-n+2}(a_{m})\\
	    \end{split}
	    \]
	    \[
	    \begin{split}
	    {\rm II}_{b \Delta(x)a }^{(2)}= &	\sum_{\substack{i=1\\ B_{k-i+1}\not = X_{i}  = A_{n-i+1}} }^{n}    \sum_{l=1}^{n}
	    \varphi(b_k \Delta_l^{1}(x_1) ) \ldots \varphi(b_{k-i+2} \Delta_l^{i-1}(x_{i-1}) )
	    \varphi(\Delta_l^{n}(x_n) a_1) \ldots \varphi( \Delta_l^{i+1}(x_{i+1})  a_{n-i} )  \\
	    & \qquad \times \quad
	    b_1 \ldots   b_{k-i+1} \overbrace{\Delta_l^{i}( x_{i})   a_{n-i+1} }^\circ a_{n-i+2}    \ldots  a_{m},\\
	    {\rm II}_{ \Delta(b x)a }^{(2)} = &	\sum_{\substack{i=1\\ B_{k-i+1} \not = X_{i}  = A_{n-i+1}} }^{n}   \sum_{l=1}^{k+n-2i+1}
	    \varphi(  b_k x_1 ) \ldots \varphi(b_{k-i+2} x_{i-1} )
	    \varphi( \Delta_l^{k+n-2i+1}(x_n) a_1) \ldots \varphi( \Delta_l^{k-i+3}(x_{i+1})  a_{n-i} )  \\
	    & \qquad \times \quad
	    \Delta_1^l(b_1) \ldots   \Delta_l^{k-i+1} ( b_{k-i+1} ) \Delta_l^{k-i+2}(\overbrace{  x_{i}  a_{n-i+1} }^\circ )  a_{n-i+2}   \ldots  a_{m}.\\
	    \end{split}
	    \]
	    Again we get (or in fact by a symmetry argument from the II$^{(1)}$-case),
	    \[
	    	    {\rm II}_{b \Delta(xa) }^{(2)} - {\rm II}_{\Delta(bxa)}^{(2)} - {\rm II}_{b\Delta(x) a}^{(2)} + {\rm II}_{\Delta(bx) a}^{(2)} = 0.
	    \]
	    We now examine the II$^{(3)}$-terms. We find,
	    \begin{equation} \label{Eqn=YetAnotherCompI}
	    \begin{split}
	       {\rm II}^{(3)}_{b \Delta(xa)} =& 	\sum_{\substack{i=1\\ B_{k-i+1} = X_{i} = A_{n-i+1}  }  }^{n}  \sum_{l=1}^{m-n+2i}
	      \varphi(b_k \Delta_l^{1}(x_1)  ) \ldots \varphi(b_{k-i+2} \Delta_l^{i-1}( x_{i-1}) )
	      \varphi(x_n a_1) \ldots \varphi(x_{i+1}  a_{n-i} )  \\
	      & \qquad \times \quad
	      b_1 \ldots b_{k-i} (\overbrace{b_{k-i+1}  x_{i}  a_{n-i+1} }^\circ   )  \Delta_l^{i+1}(a_{n-i+2}) \ldots \Delta_l^{m-n+2i}(a_{m})
	      \\
	      & + 	\sum_{\substack{i=1\\ B_{k-i+1} = X_{i} = A_{n-i+1}  }  }^{n}
	      \varphi(b_k  x_1  ) \ldots \varphi(b_{k-i+2}   x_{i-1} )
	      \varphi(x_n a_1) \ldots \varphi(x_{i+1}  a_{n-i} )  \\
	      & \qquad \times \quad
	      b_1 \ldots b_{k-i} (\overbrace{b_{k-i+1}  \Delta(  \overbrace{ x_{i}  a_{n-i+1}}^\circ   )  }^\circ   )  a_{n-i+2} \ldots  a_{m}
      \\
	      &  + G_{b,a,1}(x),
	      \end{split}
	      \end{equation}
	      where
	      \[
	      \begin{split}
	      G_{b,a,1}(x) = &	\sum_{\substack{i=1\\ B_{k-i+1} = X_{i} = A_{n-i+1}  }  }^{n}
	      \varphi(b_k  x_1  ) \ldots \varphi(b_{k-i+2}   x_{i-1} )
	      \varphi(x_n a_1) \ldots \varphi(x_{i+1}  a_{n-i} )  \\
	      & \qquad \times \quad
	      b_1 \ldots b_{k-i} (\overbrace{b_{k-i+1}   \varphi(  x_{i}  a_{n-i+1} )  }^\circ   )  a_{n-i+2} \ldots  a_{m},
	      \end{split}
	      \]
	      is a finite rank map. Similarly, there are finite rank maps $\cM \rightarrow \cM$, say $G_{b,a,2}, G_{b,a,3}$ and $G_{b,a,4}$ (in fact $G_{b,a,3}$ being the 0 map as $x_i = \overbrace{x_i}^\circ$) such that 	
	      \begin{equation} \label{Eqn=YetAnotherCompII}
	      \begin{split}
	       {\rm II}^{(3)}_{ \Delta(b xa)} =& 	\sum_{\substack{i=1\\ B_{k-i+1} = X_{i} = A_{n-i+1}  }  }^{n}  \sum_{l=1}^{m+k-n}
	      \varphi(b_k  x_1  ) \ldots \varphi(b_{k-i+2}   x_{i-1} )
	      \varphi(x_n a_1) \ldots \varphi(x_{i+1}  a_{n-i} )  \\
	      & \qquad \times \quad
	      \Delta_l^{1}(b_1) \ldots \Delta_l^{k-i}(b_{k-i}) \overbrace{b_{k-i+1}  x_{i}  a_{n-i+1} }^\circ     \Delta_l^{k-i+1}(a_{n-i+2}) \ldots \Delta_l^{m+k-n}(a_{m})
	      \\
	      & + 	\sum_{\substack{i=1\\ B_{k-i+1} = X_{i} = A_{n-i+1}  }  }^{n}
	      \varphi(b_k  x_1  ) \ldots \varphi(b_{k-i+2}   x_{i-1} )
	      \varphi(x_n a_1) \ldots \varphi(x_{i+1}  a_{n-i} )  \\
	      & \qquad \times \quad
	      b_1 \ldots b_{k-i} \Delta(\overbrace{b_{k-i+1}  x_{i}  a_{n-i+1}  }^\circ   )  a_{n-i+2} \ldots  a_{m}\\
	      & + G_{b,a,2}(x), \\ 	
	      {\rm II}^{(3)}_{ b \Delta(x)a} =& 	\sum_{\substack{i=1\\ B_{k-i+1} = X_{i} = A_{n-i+1}  }  }^{n}  \sum_{l=1}^{n-1}
	      \varphi(b_k  \Delta_l^{1}(x_1)  ) \ldots \varphi(b_{k-i+2}   \Delta_l^{i-1}(x_{i-1}) )
	      \varphi(\Delta_l^{n-1}(x_n) a_1) \ldots \varphi(  \Delta_l^{i}(x_{i+1})  a_{n-i} )  \\
	      & \qquad \times \quad
	       b_1 \ldots b_{k-i} \overbrace{b_{k-i+1}  x_{i}  a_{n-i+1} }^\circ      a_{n-i+2} \ldots  a_{m}
	      \\
	      & + 	\sum_{\substack{i=1\\ B_{k-i+1} = X_{i} = A_{n-i+1}  }  }^{n}
	      \varphi(b_k  x_1  ) \ldots \varphi(b_{k-i+2}   x_{i-1} )
	      \varphi(x_n a_1) \ldots \varphi(x_{i+1}  a_{n-i} )  \\
	      & \qquad \times \quad
	      b_1 \ldots b_{k-i} \overbrace{b_{k-i+1}  \Delta(x_{i})  a_{n-i+1}) }^\circ     a_{n-i+2} \ldots  a_{m}, \\
	      & + G_{b,a,3}(x),\\
	        {\rm II}^{(3)}_{ \Delta(b x)a} =& 	\sum_{\substack{i=1\\ B_{k-i+1} = X_{i} = A_{n-i+1}  }  }^{n}  \sum_{l=1}^{n-1}
	      \varphi(b_k   x_1  ) \ldots \varphi(b_{k-i+2}   x_{i-1} )
	      \varphi(\Delta_l^{n-1}(x_n) a_1) \ldots \varphi(  \Delta_l^{k-i+1}(x_{i+1})  a_{n-i} )  \\
	      & \qquad \times \quad
	      \Delta_l^{1}(b_1) \ldots \Delta_l^{k-i}(b_{k-i}) \overbrace{b_{k-i+1}  x_{i}  a_{n-i+1} }^\circ      a_{n-i+2} \ldots  a_{m}
	      \\
	      & + 	\sum_{\substack{i=1\\ B_{k-i+1} = X_{i} = A_{n-i+1}  }  }^{n}
	      \varphi(b_k  x_1  ) \ldots \varphi(b_{k-i+2}   x_{i-1} )
	      \varphi(x_n a_1) \ldots \varphi(x_{i+1}  a_{n-i} )  \\
	      & \qquad \times \quad
	      b_1 \ldots b_{k-i} \overbrace{  \Delta(  \overbrace{b_{k-i+1}  x_{i}}^\circ)  a_{n-i+1} }^\circ     a_{n-i+2} \ldots  a_{m} \\
	      	      & + G_{b,a,4}(x). \\ 	
	    \end{split}
	    \end{equation}
	    As $\Delta(1) = 1$ (by conservativity of the Dirichlet form)  we have for any $y$ that $\Delta(y) = \Delta(\overbrace{y}^\circ)$.
	    We see that the first summations of the 4 terms of \label{Eqn=YetAnotherCompI} and \eqref{Eqn=YetAnotherCompII} cancel each other, so that we get a remaining term:
	    \[
	        \begin{split}
	       	    	    &{\rm II}_{b \Delta(xa) }^{(3)} - {\rm II}_{\Delta(bxa)}^{(3)} - {\rm II}_{b\Delta(x) a}^{(3)} + {\rm II}_{\Delta(bx) a}^{(3)} - (G_{a,b,1}(x) - G_{a,b,2}(x) - G_{a,b,3}(x) + G_{a,b,4}(x) ) \\
	       	    	  = &
	       	    	    	\sum_{\substack{i=1\\ B_{k-i+1} = X_{i} = A_{n-i+1}  }  }^{n}
	       	    	    \varphi(b_k  x_1  ) \ldots \varphi(b_{k-i+2}   x_{i-1} )
	       	    	    \varphi(x_n a_1) \ldots \varphi(x_{i+1}  a_{n-i} ) b_1 \ldots b_{k-i}   \\
	       	    	    & \:\: \times \:\:
	       	    	    (
	       	    	    \overbrace{  b_{k-i+1}  \Delta(x_{i}  a_{n-i+1}) }^\circ -
	       	    	    \overbrace{  \Delta( b_{k-i+1}  x_{i}  a_{n-i+1}) }^\circ -
	       	    	    \overbrace{  b_{k-i+1}  \Delta(x_{i})  a_{n-i+1} }^\circ 	+       	    	
	       	    	    \overbrace{  \Delta(b_{k-i+1}  x_{i})  a_{n-i+1} }^\circ)     a_{n-i+2} \ldots  a_{m} \\
	       	    	    = & 	\sum_{\substack{i=1\\ B_{k-i+1} = X_{i} = A_{n-i+1}  }  }^{n}
	       	    	    \varphi(b_k  x_1  ) \ldots \varphi(b_{k-i+2}   x_{i-1} )
	       	    	    \varphi(x_n a_1) \ldots \varphi(x_{i+1}  a_{n-i} ) b_1 \ldots b_{k-i}   \\
	       	    	    & \qquad \times \quad
	       	    	    \overbrace{ \Psi_{X_i,0}^{a_{n-i+1}, b_{k-i+1}^\ast}(x_i)  }^{\circ}     a_{n-i+2} \ldots  a_{m}.
	       	\end{split}    	
	    \]
	    Now if we collect all of the above terms we see that
	    \begin{equation}\label{Eqn=LastPsiComp}
    	    \begin{split}
	        \Psi^{a,b^\ast}_0(x) = &	\sum_{\substack{i=1\\ B_{k-i+1} = X_{i} = A_{n-i+1}  }  }^{n}
	        \varphi(b_k  x_1  ) \ldots \varphi(b_{k-i+2}   x_{i-1} )
	        \varphi(x_n a_1) \ldots \varphi(x_{i+1}  a_{n-i} ) b_1 \ldots b_{k-i}   \\
	        & \qquad \times \quad
	        \overbrace{  \Psi_{X_i, 0}^{ a_{n-i+1},b_{k-i+1}^\ast}(x_i)  }^{\circ}     a_{n-i+2} \ldots  a_{m} \\
	        & + F_{a,b}(x),
	        \end{split}
	    \end{equation}
	    with $F_{a,b}$ the finite rank operator
	    \[
	    F_{a,b} = ( F_{b,a, 1} - F_{b, a, 2 } - F_{b,a, 3  } + F_{b , a, 4 }) + ( G_{b,a, 1 } - G_{b, a, 2 } - G_{b,a, 3  } + G_{b , a, 4 }).
	    \]

	\vspace{0.3cm}
	
\noindent {\bf 3. Conclusion of the proof.}	Let $\Vert F_{a,b} \Vert_{HS}$ be the Hilbert-Schmidt norm of $F_{a,b}$ as a map $y \Omega_\varphi \mapsto F_{a,b}(y)\Omega_\varphi$.
	Now note that if the length $n$ of $x$ as a reduced operator is strictly longer than  $k+m-1$ then the expression \eqref{Eqn=LastPsiComp} is 0 as there must be an operator $b_{k+1}$ or $a_{m+1}$ occuring in    \eqref{Eqn=LastPsiComp} which by definition are 0.

If each $\Psi_{X_i, 0}^{ a_{n-i+1},b_{k-i+1}^\ast}: L_2(\cM_i) \rightarrow L_2(\cM_i)$ in \eqref{Eqn=LastPsiComp} is bounded then so is  $\Psi^{a,b^\ast}_0: L_2(\cM) \rightarrow L_2(\cM)$. So  we conclude that the second bullet of Definition \ref{Dfn=IGHS} holds for the free product semi-group $(\Phi_t)_{t \geq 0}$ if it holds for each individual $(\Phi_{i,t})_{t \geq 0}$. It remains to verify  the first bullet point of Definition \ref{Dfn=IGHS}.

 Set $E$ as the set of   all reduced operators of the form $e_{i_1} \ldots e_{i_n}$ with $e_{i_n} \in \cup_j \mathcal{O}_j$.  $E$ forms an orthonormal basis of $L_2(\cM^\circ) = L_2(\cM) \ominus \mathbb{C} \Omega_\varphi$. Fix $t > 0$ and  let $C' = \max_j \Vert \sigma_{i/2}( \Phi_t( a_j) ) \Vert$ and then $C = \max(1, C')$. Further set $D' = \max_j \Vert   \Phi_t( b_j)   \Vert$ and then $D = \max(1, D')$. We conclude from \eqref{Eqn=LastPsiComp} and twice Cauchy-Schwarz that,
	\[
	\begin{split}
		\Vert \Psi_t^{a, b^\ast} \Vert_{HS}^2 = & \sum_{x \in E}   \Vert \Psi_t^{a, b^\ast}(x) \Vert_{2}^2 \\
		 \leq &2 \Vert F_{a,b} \Vert_{HS}^2 +
		2 \sum_{x \in E} \Vert \sum_{\substack{i=1\\ B_{k-i+1} = X_{i} = A_{n-i+1}  }  }^{k+m-1}
		 \varphi(b_k  x_1  ) \ldots \varphi(b_{k-i+2}   x_{i-1} )
		 \varphi(x_n a_1) \ldots \varphi(x_{i+1}  a_{n-i} )   \\		
		 & \qquad \times \quad
		  \Phi_t(b_1) \ldots \Phi_t(b_{k-i}) \overbrace{  \Psi_{X_i, t}^{ a_{n-i+1}, b_{k-i+1}^\ast}(x_i)  }^{\circ}    \Phi_t( a_{n-i+2}) \ldots \Phi_t( a_{m}) \Vert^2_2\\
		\leq  &2 \Vert F_{a,b} \Vert_{HS}^2 +
		 2  (k+m-1) \!\!\!\!\!\! \sum_{\substack{i=1\\ B_{k-i+1} = X_{i} = A_{n-i+1}  }  }^{k+m-1} \!\!\!\!\!\! \!\!\!\!\!\!\sum_{x \in E}
		 \vert \varphi(b_k  x_1  ) \ldots \varphi(b_{k-i+2}   x_{i-1} )
		 \varphi(x_n a_1) \ldots \varphi(x_{i+1}  a_{n-i} )\vert^2    \\
		 & \qquad \times \quad
		 \Vert \Phi_t(b_1) \ldots\Phi_t( b_{k-i}) \overbrace{  \Psi_{X_i, t}^{ a_{n-i+1}, b_{k-i+1}^\ast}(x_i)  }^{\circ}    \Phi_t( a_{n-i+2}) \ldots  \Phi_t( a_{m}) \Vert^2_ 2
		 \end{split}
		 \]
	 For all $j$ we have
	 \[
	 \sum_{y \in O_{B_j}} \vert \varphi(b_j y) \vert^2 =
	 \sum_{y \in O_{B_j}} \vert \langle y \Omega_\varphi, b_j^\ast \Omega_\varphi  \rangle \vert^2 =
	  \Vert b_j^\ast \Vert_2^2,
	  \]
	   because $O_{B_j}$ is an orthonormal basis. Similarly,
	  \[
	  \sum_{y \in O_{A_j}}  \vert \varphi( y a_j) \vert^2 =
	  \sum_{y \in O_{A_j}}  \vert \varphi( \sigma_i(a_j) y ) \vert^2 =
	  \sum_{y \in O_{B_j}} \vert \langle y \Omega_\varphi, \sigma_i(a_j)^\ast \Omega_\varphi  \rangle \vert^2 =
	  \Vert \sigma_i(a_j)^\ast \Vert_2^2,
	  \]
	  Therefore let $K = \max_j( \Vert b_j^\ast \Vert_2^2, \Vert \sigma_i(a_j)^\ast \Vert_2^2 , 1)$. 	
	   We get  using \eqref{Eqn=ActionEstimates},	
		\[
		\begin{split}
	\Vert \Psi_t^{a, b^\ast} \Vert_{HS}^2 	\leq &2 \Vert F_{a,b} \Vert_{HS}^2 +  2(k+m-1) K^{m+k} C^{2m} D^{2k}  \sum_{\substack{i=1 \\ C_{k-i+1} = X_{i} = A_{n-i+1} } }^{k+m-1} \sum_{x_i \in E_i}  \Vert \Psi_{X_i,t}^{a_{n-i-1}, b_{k-i-1}^\ast}(x_{i})  \Vert_2^2  \\
	\leq & 2 \Vert F_{a,b} \Vert_{HS}^2 + 2(k+m-1)^2 K^{m+k} C^{2m} D^{2k}  \max_{1 \leq i \leq n} \Vert \Psi_{X_i, t}^{a_{n-i-1}, b_{k-i-1}^\ast}   \Vert_{HS}^2,
	\end{split}
	\]
	which is finite for every $t > 0$ and for every choice of $a$ and $b$ in $\cA$. 	The proof for GHS instead of IGHS follows just by using $t = 0$ instead of $t > 0$.
\end{proof}

\subsection{Crossed product extensions} We prove that IGHS semi-groups yield GC semi-groups on their continuous cores. We recall the following from \cite{TakesakiII}. As before let $\cM$ be a $\sigma$-finite von Neumann algebra with fixed faithful normal state $\varphi$.
Let $c_{\varphi}(\cM)$ be the {\bf continuous core} von Neumann algebra of $\cM$. It is the von Neumann algebra acting on $L_2(\cM) \otimes L_2(\mathbb{R}) \simeq L_2(\mathbb{R}, L_2(\cM))$ that is generated by the operators
\[
(\pi_\varphi(x) \xi)(t) = \sigma^{\varphi}_{-t}(x)\xi(t), \qquad  {\rm where } \quad x \in \cM,
\]
and the shifts
\[
(u_s \xi)(t) = \xi(t-s) \qquad {\rm  where } \quad s,t \in \mathbb{R}.
\]
 We shall write $u_f = \int_{\mathbb{R}} f(s) u_s ds$ for $f \in L_1(\mathbb{R})$. The map $\pi_\varphi$ embeds $\cM$ into $c_\varphi(\cM)$.
   We let $\cL_\varphi(\mathbb{R})$ be the von Neumann algebra generated by $u_t, t \in \mathbb{R}$.
    Let $\widetilde{\varphi}$ be the {\bf dual weight} on $c_\varphi(\cM)$ of $\varphi$. If $s \mapsto x_s$ and $s \mapsto y_s$ are compactly supported $\sigma$-weakly continuous functions $\mathbb{R} \rightarrow \cM$, it satisfies
   \[
    \widetilde{\varphi}(  \left(  \int_{\mathbb{R}} \pi_\varphi(y_s) u_s ds \right)^\ast  \int_{\mathbb{R}} \pi_\varphi(x_s) u_s   ds  ) =
       \int_{\mathbb{R}} \varphi(y_s^\ast x_s) ds.
   \]
   We call the support of $s \mapsto x_s$ the {\bf frequency support} of $\int_{\mathbb{R}} \pi_\varphi(x_s) u_s ds$.
Let $h \geq 0$ be the self-adjoint operator affiliated with  $\cL_\varphi(\mathbb{R})$ such that $h^{it} = u_t, t \in \mathbb{R}$.
     There exists a normal, faithful, semi-finite {\bf trace} $\widetilde{\tau}$ on $c_\varphi(\cM)$ such that we have cocycle derivative $(D \widetilde{\varphi}/D\widetilde{\tau})_t = h^{it}$. This is informally expressed as $\widetilde{\tau}(h^{1/2} \: \cdot \:  h^{1/2}) = \widetilde{\varphi}( \: \cdot \: )$. We write
     \[
     \mathfrak{n}_{\widetilde{\varphi}} = \left\{ x \in c_\varphi(\cM)  \mid  \widetilde{\varphi}(x^\ast x) < \infty \right\}.
        \]
For $x \in  \mathfrak{n}_{\widetilde{\varphi}}$ we write $x \Omega_{\widetilde{\varphi}}$ for its  GNS-embedding into $L_2(c_{\varphi}(\cM), \widetilde{\varphi})$.
 Let $J_{\widetilde{\varphi}}$ be the modular conjugation.
 $L_2(c_{\varphi}(\cM), \widetilde{\varphi})$ is a  $c_\varphi(\cM)$-$c_\varphi(\cM)$-bimodule with left and right actions
\[
x \cdot  (a\Omega_{\widetilde{\varphi}})   \cdot y =   x J_{\widetilde{\varphi}} y^\ast J_{\widetilde{\varphi}} (a \Omega_{\widetilde{\varphi}}), \qquad a \in \mathfrak{n}_{\widetilde{\varphi}}, x,y \in c_\varphi(\cM).
\]
The Tomita algebra $\mathcal{T}_{\widetilde{\varphi}}$ is defined as the algebra of all $x \in c_{\varphi}(\cM)$ that are analytic for $\sigma^{\widetilde{\varphi}}$ and such that for every $z \in \mathbb{C}$ we have  $\sigma^{\widetilde{\varphi}}_z(x)  \in \mathfrak{n}_{\widetilde{\varphi}} \cap \mathfrak{n}_{\widetilde{\varphi}}^\ast$.
It shall be convenient for us to identify unitarily
\begin{equation}\label{Eqn=CrossedL2}
L_2(c_\varphi(\cM), \widetilde{\varphi}) \rightarrow\!\!\!\!\!\!\!\!^{\simeq}  \:\:\: L_2(\mathbb{R}, L_2(\cM)):   \pi_{\varphi}(x) u_f \Omega_{\widetilde{\varphi}} \mapsto  (f(s) x \Omega_{\varphi})_{s \in \mathbb{R}}, \qquad f \in C_{00}(\mathbb{R}), x \in \cM.
\end{equation}

\begin{rmk}
We may similarly set
\[
   \mathfrak{n}_{\widetilde{\tau}} = \left\{ x \in c_\varphi(\cM)  \mid  \widetilde{\tau}(x^\ast x) < \infty \right\}.
\]
For $x \in \mathfrak{n}_{\widetilde{\tau}}$ we write   $x \Omega_{\widetilde{\tau}}$ for its  GNS-embedding into $L_2(c_{\varphi}(\cM), \widetilde{\tau})$. We have the  $c_\varphi(\cM)$-$c_\varphi(\cM)$-bimodule structure given by
\[
x \cdot (a\Omega_{\widetilde{\tau}}) \cdot y =   x J_{\widetilde{\tau}} y^\ast J_{\widetilde{\tau}} (a\Omega_{\widetilde{\tau}}), \qquad a \in \mathfrak{n}_{\widetilde{\tau}}, x,y \in c_\varphi(\cM).
\]
Consider the map
\begin{equation}\label{Eqn=BimodEquiv}
D \rightarrow L_2(c_{\varphi}(\cM), \widetilde{\tau}):  x \Omega_{\widetilde{\varphi}} \mapsto    [x h^{\frac{1}{2}}] \Omega_{\widetilde{\tau}},
\end{equation}
where $D \subseteq L_2(c_{\varphi}(\cM), \widetilde{\varphi})$ is the space of $x \in \mathfrak{n}_{\widetilde{\varphi}}$ such that $x h^{\frac{1}{2}}$ is bounded and the closure $[x h^{\frac{1}{2}}]$ is in  $\mathfrak{n}_{\widetilde{\tau}}$. The map \eqref{Eqn=BimodEquiv} extends to a unitary map $L_2(c_{\varphi}(\cM), \widetilde{\varphi}) \rightarrow L_2(c_{\varphi}(\cM), \widetilde{\tau})$ which is moreover an equivalence of   $c_\varphi(\cM)$-$c_\varphi(\cM)$ bimodules. We simply write $L_2(c_{\varphi}(\cM))$ for this bimodule.
\end{rmk}

\vspace{0.3cm}

Recall that a  Markov semi-group $\Phi = (\Phi_t)_{t \geq 0}$ on $\cM$ is called {\bf $\varphi$-modular} if $\sigma_s^\varphi \circ \Phi_t = \Phi_t \circ \sigma_s^\varphi$ for all $s \in \mathbb{R}$ and $t \geq 0$.
	Let $\Phi  = (\Phi_t)_{t \geq 0}$ be a $\varphi$-modular Markov map on $\cM$. Then let $\widetilde{\Phi} = (\widetilde{\Phi}_t)_{t \geq 0}$ be the crossed product extension on $c_\varphi(\cM)$  determined by
	\[
	\widetilde{\Phi}_t(\pi_\varphi(x)) = \pi_\varphi(\Phi_t(x)) \quad {\rm and }\quad  \widetilde{\Phi}_t(u_s) = u_s \quad {\rm where } \quad x \in \cM, s \in \mathbb{R}, t \geq 0.
	\]
	 If $\Phi$ is a $\varphi$-modular Markov semigroup then so is $\widetilde{\Phi}$ for both the weights $\widetilde{\varphi}$ and $\widetilde{\tau}$, meaning that it is a point-strongly continuous semi-group of ucp maps that preserves these weights. If $p \in \cL_\varphi(\mathbb{R})$ is a $\widetilde{\tau}$-finite projection then the restriction of $\widetilde{\Phi}$ to the corner $p c_\varphi(\cM) p$ is a Markov-semigroup with respect to $\widetilde{\tau}$.

\vspace{0.3cm}

\noindent {\it Convention for the rest of this subsection:} Let $\cM = L_\infty(\bG)$ for a compact quantum group $\bG$ and let $\varphi$ be the Haar state of $\bG$. Let $\cA = \cA(\bG)$ be the $\ast$-algebra of matrix coefficients of finite dimensional representations of $\bG$.

\vspace{0.3cm}

The convention is mainly made to simplify several technicalities occuring in the proofs of Lemmas \ref{Lem=CrossedQG}, \ref{Lem=CrossedBounded}  as well as   Proposition \ref{Prop=CrossedGC}.
Let $(\Phi_t)_{t \geq 0}$ be a  Markov semi-group of central multipliers. Let $\Delta \geq 0$ be a generator for  $(\Phi_t)_{t \geq 0}$, i.e. $e^{-t\Delta} = \Phi_t^{(2)}$.
Let $p \in \cL_\varphi(\mathbb{R})$ be a projection. Then  $\Delta \otimes p$ is a generator for the restriction of $(\widetilde{\Phi}_t^{(2)})_{t \geq 0}$ to $p c_\varphi(\cM) p$. Its domain is understood as all $\ell_2$-sums
	\[
	\sum_{\alpha \in \Irr(\bG), 1 \leq i,j \leq n_\alpha} f_{i,j}^\alpha \otimes u_{i,j}^\alpha
	\]
	 with $f_{i,j}^\alpha \in p L_2(\mathbb{R})$ such that also   $\sum_{\alpha \in \Irr(\bG), 1 \leq i,j \leq n_\alpha} f_{i,j}^\alpha \otimes \Delta(u_{i,j}^\alpha)$ exists as a $\ell_2$-convergent sum.

\begin{dfn}
   Let $\widetilde{\cA}$  be the $\ast$-algebra of elements $\int_{\mathbb{R}} \pi_\varphi(x_s) u_s ds \in c_\varphi(\cM)$ with $x_s \in \mathcal{A}$ $\sigma$-weakly continuous and compactly supported in $s$.
\end{dfn}

\begin{lem}\label{Lem=CrossedQG}
	Let $(\Phi_t)_{t \geq 0}$ be a Markov semi-group of central multipliers on a compact quantum group $\bG$.	Let $\cA = \cA(\bG)$ and let $\widetilde{\cA}$ be defined as above. Then $\widetilde{\cA}$ is contained in the Tomita algebra $\mathcal{T}_{\widetilde{\varphi}}$ and moreover $\widetilde{\Delta} (\nabla^{\frac{1}{4}} \widetilde{\cA} \Omega_{\widetilde{\varphi}}  ) \subseteq \nabla^{\frac{1}{4}} \widetilde{\cA} \Omega_{\widetilde{\varphi}}$. Further,   we may set (the limit being existent), 	
	\[
	\widetilde{\Delta}(x) = \lim_{t \searrow 0} \frac{1}{t} (\widetilde{\Phi}_t(x) - x), \qquad x \in \widetilde{\cA}.
	\]
	Moreover,
	\[
	\widetilde{\Delta}( \int_{\mathbb{R}} \pi_\varphi(x_s) u_s ds )
	=\int_{\mathbb{R}} \pi_\varphi(\Delta(x_s)) u_s ds.
	\]
\end{lem}
\begin{proof}
	The inclusion $\widetilde{\cA} \subseteq \mathcal{T}_{\widetilde{\varphi}}$ follows from the explicit form of the modular group of $\widetilde{\varphi}$, see \cite[Theorem X.1.17]{TakesakiII}. If $s \mapsto x_s \in \cA$ is continuous and compactly supported, it takes values in the space of matrix coefficients of a single finite dimensional representation of $\bG$. Write $x_s = \sum_{\alpha} x_{s, \alpha}$ where $\alpha$ ranges over this finite ($s$-independent) subset of $\Irr(\bG)$. Then $\Delta( \pi_\varphi(x_s) ) = \sum_{\alpha} \Delta_\alpha  \pi_\varphi(x_{s, \alpha})$ for some $\Delta_\alpha \in \mathbb{C}$.  Further, for $x = \int_{\mathbb{R}} \pi_{\varphi}(x_s) u_s ds$,
	\[
	\lim_{t \searrow 0} \frac{1}{t} (\widetilde{\Phi}_t(x) - x)
	=\lim_{t \searrow 0} \frac{1}{t} \int_{\mathbb{R}} \pi_\varphi(  \Phi_t(x_s) -  x_s ) u_s ds =  \int_{\mathbb{R}} \pi_{\varphi}(\Delta(x_s)) u_s ds.
	\]	
\end{proof}

\begin{dfn}
Assume $(\Phi_t)_{t \geq 0}$ is a Markov  semi-group of central multipliers on a compact quantum group $\bG$.
Set,
\[
\langle a, c \rangle_{\widetilde{\Gamma}} = \widetilde{\Delta}(c)^\ast a + c^\ast \widetilde{\Delta}(a) - \widetilde{\Delta}(c^\ast a), \qquad a,c \in \widetilde{\cA}.
\]
 And further,
\[
\langle   a \otimes \xi, c \otimes \eta  \rangle_{\widetilde{\partial}} =
\frac{1}{2}
\langle  \langle a, c \rangle_{\widetilde{\Gamma}}  \xi, \eta \rangle, \qquad a,c \in \widetilde{\cA}, \xi, \eta \in L_2(c_\varphi(\cM)).
\]
Just as in the state case this defines an inner product on $\widetilde{\cA} \otimes L_2(c_\varphi(\cM))$. Quotienting out the degenerate part and taking a completion yields a Hilbert space $\cH_{\partial, c_\varphi}$ with contractive left and right $\widetilde{\cA}$-actions given by
\[
x \cdot (z \otimes \xi) = xz \otimes \xi - x \otimes z \xi, \qquad (z \otimes \xi) \cdot y = z \otimes \xi y, \qquad x,y, z \in \widetilde{\cA}, \xi \in L_2(c_{\varphi}(\cM)).
\]
We also set the map
\[
\widetilde{\Psi}^{z, z'}_t(x) = \widetilde{\Phi}_t(\langle x z, z' \rangle_{\widetilde{\Gamma}}  - \langle x, z \rangle_{\widetilde{\Gamma}} z'), \qquad x,z,z' \in \widetilde{\cA}, t \geq 0,
\]
and set $\widetilde{\Psi}^{z, z'} = \widetilde{\Psi}^{z, z'}_0$.
So that
\[
\langle x \cdot (z \otimes \xi) \cdot y, (z' \otimes \eta) \rangle_{\widetilde{\partial}}
= \langle \widetilde{\Psi}^{z, z'}(x)\xi y, \eta \rangle, \qquad x,y,z,z' \in \widetilde{\cA}, \xi, \eta \in L_2(c_\varphi(\cM)).
\]
\end{dfn}

\begin{rmk}
Let $a, c \in \cA$. By mild abuse of notation we shall write $\widetilde{\Psi}^{\pi_\varphi(a), \pi_\varphi(c)}$ for the $L_2$-map
\[
L_2(c_\varphi(\cM)) \rightarrow L_2(c_\varphi(\cM)): x \Omega_{\widetilde{\varphi}} \mapsto \widetilde{\Psi}^{\pi_\varphi(a), \pi_\varphi(c)}(x) \Omega_{\widetilde{\varphi}},
\]
 in case this map is bounded and say that $\widetilde{\Psi}^{\pi_\varphi(a), \pi_\varphi(c)} \in B(L_2(c_\varphi(\cM)))$.  
  Similarly we write  $\Psi^{ a,  c}$ for the map $L_2(\cM) \rightarrow L_2(\cM): x \Omega_\varphi \mapsto \Psi^{ a,  c}(x)\Omega_\varphi$ in case this map is bounded and say $\Psi^{ a,  c} \in B(L_2(\cM))$.
\end{rmk}

\begin{lem}\label{Lem=CrossedBounded}
Let $(\Phi_t)_{t \geq 0}$ be a Markov  semi-group of central multipliers on a compact quantum group $\bG$. Assume that for all $a,c \in \mathcal{A}$  we have  $\Psi^{a,c} \in B(L_2(c_\varphi(\cM)))$.   Then $\widetilde{\Psi}^{\pi_\varphi(a), \pi_\varphi(c)} \in B(L_2(c_\varphi(\cM)))$ and under the correspondence \eqref{Eqn=CrossedL2} we have  $\widetilde{\Psi}^{\pi_\varphi(a), \pi_\varphi(c)}  \simeq (\Psi^{\sigma_s^\varphi(a), c})_{s \in \mathbb{R}}$.
\end{lem}
\begin{proof}
Take $a,c \in \mathcal{A}$. Assume that $a$ is a matrix coefficient of the finite dimensional representation $u$. Let $\{ u_{i,j} \mid i,j\}$ form a linear basis of all matrix coefficients of $u$. The modular group preserves matrix coefficients of a fixed representation, see \cite{Woronowicz} (or \cite[Theorem 4.8]{CaspersKoelink}).
 So decompose $\sigma_s^\varphi(a) = \sum_{i,j} f_{i,j}(s) u_{i,j}$ with $f_{i,j}(s) \in \mathbb{C}$. Then each $f_{i,j}$ is bounded and continuous since $\sigma^\varphi$ is a $\sigma$-weakly continuous automorphism group. We see that
 \[
 \Psi^{\sigma_s^\varphi(a), c} =
\sum_{i,j} f_{i,j}(s) \Psi^{u_{i,j}, c},
\]
 and by assumption  $\Psi^{u_{i,j}, c}  \in B(L_2(\cM))$. This shows that $\Psi^{\sigma_s^\varphi(a), c} \in B(L_2(\cM))$ with uniform bound in $s$.

Now take $a,c, x \in \cA$ and $f \in L_1(\mathbb{R})$. We have
\[
\pi_\varphi(c)^\ast \pi_\varphi(x) u_f \pi_\varphi(a)  =
 \int_{\mathbb{R}} f(s) \pi_\varphi(  c^\ast x \sigma_s^\varphi(a)  )   u_s ds.
\]
 We have $\pi_\varphi(x) u_f \in \mathfrak{n}_{\widetilde{\varphi}}$ and using Lemma \ref{Lem=CrossedQG},
\[
\begin{split}
 \widetilde{\Psi}^{\pi_\varphi(a), \pi_\varphi(c)}(\pi_\varphi(x) u_f)
  = &  \int_{\mathbb{R}} f(s) \pi_\varphi(  \Delta( c^\ast x \sigma_s^\varphi(a) ) +
  c^\ast  \Delta( x) \sigma_s^\varphi(a) -
   c^\ast  \Delta( x \sigma_s^\varphi(a) ) -
    \Delta(c^\ast  x) \sigma_s^\varphi(a)
      )   u_s ds \\
      = &  \int_{\mathbb{R}} f(s) \pi_\varphi(  \Psi^{\sigma_s^\varphi(a), c}(x)
      )   u_s ds.
\end{split}
\]
Under the identification \eqref{Eqn=CrossedL2} we see that $\widetilde{\Psi}^{\pi_\varphi(a), \pi_\varphi(c)}$ corresponds to $(\Psi^{\sigma_s^\varphi(a), c})_{s \in \mathbb{R}} \in L_\infty(\mathbb{R}, B(L_2(\cM)))$.
\end{proof}

\begin{prop}\label{Prop=CrossedGC}
Let $(\Phi_t)_{t \geq 0}$ be a Markov  semi-group of central multipliers on a compact quantum group $\bG$. Let $p \in  \cL_\varphi(\mathbb{R})$ be a  projection. Then:
\begin{enumerate}
\item \label{Item=Crossed1} The $\widetilde{\cA}$-$\widetilde{\cA}$-bimodule $\cH_{\partial, c_\varphi}$ extends to a normal $c_\varphi(\cM)$-$c_\varphi(\cM)$-bimodule. Moreover, $p \cH_{\partial, c_\varphi} p$ is a normal $pc_\varphi(\cM)p$-$pc_\varphi(\cM)p$-bimodule.
\item \label{Item=Crossed2} If  $(\Phi_t)_{t \geq 0}$ on $(\cM, \varphi)$ is IGHS then  the Markov semi-group $(c_\varphi(\Phi_t))_{t \geq 0}$ on $p c_\varphi(\cM) p$  is GC.
\end{enumerate}
\end{prop}
\begin{proof}
	To keep the notation simple we will identify $\cA$ as a subalgebra of $c_\varphi(\cM)$ through the embedding $\pi_\varphi$ and further supress $\pi_\varphi$ in the notation.  We prove the statements for the projection $p =1$ and then justify how the general statements follow from this.
	Throughout the entire proof let $f_1, f_2,   g_1, g_2  \in C_{00}(\mathbb{R})$, $a,b,c,d \in \cA$.

	\vspace{0.3cm}

\noindent	
	 \noindent {\it Proof of  \eqref{Item=Crossed1} for $p =1$. }
Let $x\in \widetilde{\cA}$.
  We have,
\begin{equation}\label{Eqn=CrossPsi}
\begin{split}
\widetilde{\Psi}^{au_{f_1},  cu_{f_2}}(x ) = &
u_{f_2}^\ast c^\ast  \widetilde{\Delta}(x   au_{f_1}  ) - u_{f_2}^\ast c^\ast  \widetilde{\Delta}(x ) a u_{f_1}  - \widetilde{\Delta}(u_{f_2}^\ast  c^\ast   x  a u_{f_1}  )
+ \widetilde{\Delta}(  u_{f_2}^\ast c^\ast x )  a u_{f_1}  \\
= &
u_{f_2}^\ast c^\ast  \widetilde{\Delta}( x   a  )u_{f_1} - u_{f_2}^\ast c^\ast \widetilde{\Delta}(   x  ) au_{f_1}  - u_{f_2}^\ast \widetilde{\Delta}(  c^\ast  x  a  ) u_{f_1}
+ u_{f_2}^\ast \widetilde{\Delta}(  c^\ast  x  ) a u_{f_1}    \\
= & u_{f_2}^\ast  \widetilde{\Psi}^{a ,c }(  x ) u_{f_1}.
\end{split}
\end{equation}
We have
\begin{equation}\label{Eqn=CrossedComp0}
\begin{split}
\langle x  ( a u_{f_1} \otimes b u_{g_1} \Omega_{\widetilde{\varphi}} ) ,  c u_{f_2}  \otimes d u_{g_2} \Omega_{\widetilde{\varphi}}  \rangle
= & \langle     \widetilde{\Psi}^{ au_{f_1} ,  c u_{f_2}}(  x )    b u_{g_1}  \Omega_{\widetilde{\varphi}} , d u_{g_2}  \Omega_{\widetilde{\varphi}}   \rangle \\
= &
\langle    u_{f_2}^\ast  \widetilde{\Psi}^{a,c}(  x  )u_{f_1}   b u_{g_1}   \Omega_{\widetilde{\varphi}} , d u_{g_2}  \Omega_{\widetilde{\varphi}}   \rangle.
\end{split}
\end{equation}
Now to show that the left $\widetilde{\cA}$-action on $\cH_{\partial, c_\varphi}$ is normal it suffices to show that it is $\sigma$-weakly continuous on the unit ball of $c_\varphi(\cM)$. So suppose that $x_k$ is a net in the unit ball of  $\widetilde{\cA}$ converging $\sigma$-weakly to $x$.   We get that $\widetilde{\Psi}^{a,c}(  x_k  ) \in \widetilde{\cA}$ and  may be written as $\widetilde{\Psi}^{a,c}(  x_k  ) = \int \pi_{\varphi}(y_{k,s}) u_s ds$, with integral ranging over some compact set. Similarly write $\widetilde{\Psi}^{a,c}(  x ) = \int \pi_{\varphi}(y_{s}) u_s ds$.
 Let $u_{i,j}^\alpha$ be a matrix coefficient of $\alpha \in \Irr(\bG)$. Then $(\langle y_{k,s} \Omega_\varphi, u_{i,j}^\alpha \Omega_\varphi \rangle)_{s \in \mathbb{R}}$ is an element of $L_\infty(\mathbb{R})$ that is $\sigma$-weakly convergent to $(\langle y_{s} \Omega_\varphi, u_{i,j}^\alpha \Omega_\varphi \rangle)_{s \in \mathbb{R}}$. It follows then from the expression \eqref{Eqn=CrossedComp0} that
 \[
  \langle (x - x_k)  ( a u_{f_1} \otimes b u_{g_1} \Omega_{\widetilde{\varphi}} ),  c u_{f_2}  \otimes d u_{g_2} \Omega_{\widetilde{\varphi}}  \rangle \rightarrow 0.
  \]
  Since $x_k$ is bounded it follows that for $\xi \in \cH_{\partial, c_\varphi}$ arbitrary we get that $\langle (x - x_k) \xi, \xi \rangle \rightarrow 0$. This concludes the claim on the left action; the right action goes similarly.

\vspace{0.3cm}

\noindent	
	 \noindent {\it Proof of  \eqref{Item=Crossed2} for $p =1$. } Assume that $(\Phi_t)_{t \geq 0}$ is IGHS. For $x \in \widetilde{\cA}$ we have as in \eqref{Eqn=CrossPsi} that $\widetilde{\Psi}^{u_{f_1}a,  u_{f_2}c}(x ) =
 \widetilde{\Psi}^{a ,c }( u_{f_2}^\ast x  u_{f_1})$.
Hence for $x,y \in \widetilde{\cA}$ we have
\begin{equation}\label{Eqn=DensityCross}
\begin{split}
\langle x  (u_{f_1} a  \otimes b u_{g_1}  \Omega_{\widetilde{\varphi}} ) y   , u_{f_2} c  \otimes d u_{g_2}  \Omega_{\widetilde{\varphi}}  \rangle
= & \langle     \widetilde{\Psi}^{u_{f_1} a, u_{f_2} c}(  x )    b u_{g_1}  \Omega_{\widetilde{\varphi}} y, d u_{g_2}  \Omega_{\widetilde{\varphi}}   \rangle \\
= & \langle     \widetilde{\Psi}^{ a,  c}(  u_{f_2}^\ast x u_{f_1} )    b u_{g_1} \Omega_{\widetilde{\varphi}} y, d u_{g_2}  \Omega_{\widetilde{\varphi}}   \rangle \\
= & \langle     \widetilde{\Psi}^{ a,  c}(  u_{f_2}^\ast x u_{f_1}  \Omega_{\widetilde{\varphi}}  )      \sigma_{i/2}^{\widetilde{\varphi}}(b u_{g_1} )  y, d u_{g_2}  \Omega_{\widetilde{\varphi}}   \rangle \\
= & \langle     \widetilde{\Psi}^{ a,  c}(  u_{f_2}^\ast x u_{f_1}  \Omega_{\widetilde{\varphi}}  )     , d u_{g_2}  \Omega_{\widetilde{\varphi}}  y^\ast \sigma_{i/2}^{\widetilde{\varphi}}(b)^\ast u_{g_1}^\ast     \rangle \\
\end{split}
\end{equation}
We argue that in fact \eqref{Eqn=DensityCross} holds for all $x,y \in c_\varphi(\cM)$. Indeed, $\widetilde{\cA}$ is strong-$\ast$ dense in $c_\varphi(\cM)$ so by Kaplansky's density theorem we may take   bounded nets $x_k$ and $y_k$ in $\widetilde{\cA}$ converging in the strong-$\ast$ topology (hence $\sigma$-weakly) to $x \in c_\varphi(\cM)$ and $y \in c_\varphi(\cM)$ respectively. By Step 1 the left and right actions are normal (meaning $\sigma$-weakly continuous)  so that
\begin{equation}\label{Eqn=Limit1}
\lim_{k_1, k_2}  \langle x_{k_1}  (u_{f_1} a  \otimes b u_{g_1}  \Omega_{\widetilde{\varphi}} ) y_{k_2}   , u_{f_2} c  \otimes d u_{g_2}  \Omega_{\widetilde{\varphi}}  \rangle =
\langle x  (u_{f_1} a  \otimes b u_{g_1}  \Omega_{\widetilde{\varphi}} ) y   , u_{f_2} c  \otimes d u_{g_2}  \Omega_{\widetilde{\varphi}}  \rangle.
\end{equation}
Since   $\widetilde{\Psi}^{ a,  c}$ is bounded by Lemma \ref{Lem=CrossedBounded} and   $x_k u_{f_1}  \Omega_{\widetilde{\varphi}} \rightarrow x u_{f_1}  \Omega_{\widetilde{\varphi}}$ in norm we find
\begin{equation}\label{Eqn=Limit2}
\lim_{k_1, k_2} \langle     \widetilde{\Psi}^{ a,  c}(  u_{f_2}^\ast x_{k_1} u_{f_1}  \Omega_{\widetilde{\varphi}}  )     , d u_{g_2}  \Omega_{\widetilde{\varphi}}  y_{k_2}^\ast \sigma_{i/2}^{\widetilde{\varphi}}(b)^\ast u_{g_1}^\ast \rangle = \langle     \widetilde{\Psi}^{ a,  c}(  u_{f_2}^\ast x u_{f_1}  \Omega_{\widetilde{\varphi}}  )     , d u_{g_2} \Omega_{\widetilde{\varphi}}   y^\ast \sigma_{i/2}^{\widetilde{\varphi}}(b)^\ast u_{g_1}^\ast   \rangle.
\end{equation}
The limits \eqref{Eqn=Limit1} and \eqref{Eqn=Limit2} show that \eqref{Eqn=DensityCross} holds for all  $x,y \in c_\varphi(\cM)$.
Further, by strong continuity of the semi-group we find  for all  $x,y \in c_\varphi(\cM)$,
\begin{equation}\label{Eqn=CrossedComp1bBroken}
\begin{split}
&
\langle x  (u_{f_1} a  \otimes b u_{g_1}  \Omega_{\widetilde{\varphi}} ) y   , u_{f_2} c  \otimes d u_{g_2}  \Omega_{\widetilde{\varphi}}  \rangle =
\lim_{t \searrow 0} \langle     \widetilde{\Psi}_t^{a,c}( u_{f_2}^\ast  x  u_{f_1})   \Omega_{\widetilde{\varphi}}, d u_{g_2}  \Omega_{\widetilde{\varphi}}  y^\ast    \sigma_{i/2}^{\widetilde{\varphi}}(b)^\ast u_{g_1}^\ast  \rangle.
\end{split}
\end{equation}
By the unitary identification \eqref{Eqn=CrossedL2} there exist   $z_s, z_s' \in \cM$  such that,
\[
           u_{f_2}^\ast  x  u_{f_1}   \Omega_{\widetilde{\varphi}}  \simeq (z_s \Omega_\varphi)_{s \in \mathbb{R}}  \in L_2(\mathbb{R}, L_2(\cM)), \qquad
d u_{g_2}  \Omega_{\widetilde{\varphi}}  y^\ast    \sigma_{i/2}^{\widetilde{\varphi}}(b  )^\ast \simeq (z_s' \Omega_\varphi)_{s \in \mathbb{R}}  \in L_2(\mathbb{R}, L_2(\cM)).
\]
We have
\[
 d u_{g_2}  \Omega_{\widetilde{\varphi}}  y^\ast    \sigma_{i/2}^{\widetilde{\varphi}}(b)^\ast u_{g_1}^\ast \simeq
 (\int_{\mathbb{R}} g_1(r) z_{s+r}' \Omega_\varphi dr)_{s \in \mathbb{R}}.
\]
We may express the limiting terms on the  right hand side of \eqref{Eqn=CrossedComp1bBroken} as follows by using Lemma \ref{Lem=CrossedBounded},
\begin{equation}\label{Eqn=ThreeOneSeries}
\begin{split}
 \langle     \widetilde{\Psi}_t^{a,c}( u_{f_2}^\ast  x  u_{f_1})   \Omega_{\widetilde{\varphi}}, d u_{g_2}  \Omega_{\widetilde{\varphi}}  y^\ast   \sigma_{i/2}^{\widetilde{\varphi}}(b)^\ast u_{g_1}^\ast \rangle
 =&  \langle      \Psi_t^{a,c}(z_s \Omega_\varphi)_{s \in \mathbb{R} },    (\int_{\mathbb{R}} g_1(r) z_{s+r}' \Omega_\varphi dr)_{s \in \mathbb{R}}  \rangle \\
= &
\langle     (\Psi_t^{\sigma_s^\varphi(a),c}(z_s) \Omega_\varphi)_{s \in \mathbb{R} },    (\int_{\mathbb{R}} g_1(r) z_{s+r}' \Omega_\varphi dr)_{s \in \mathbb{R}}  \rangle.
\end{split}
\end{equation}
Because $(\Phi_t)_{t \geq 0}$ is IGHS there exists for every $t> 0, s \in \mathbb{R}$ a vector $\zeta_t^{\sigma_s^\varphi(a), c} \in L_2(\cM) \otimes \overline{L_2(\cM)}$ such that for all $z, z' \in \cM$,
\[
\langle      \Psi_t^{\sigma_s^\varphi(a),c}(z) \Omega_\varphi,        z' \Omega_\varphi   \rangle = \langle z \Omega_\varphi \otimes \overline{z' \Omega_{\varphi}}, \zeta_t^{\sigma_s^\varphi(a), c} \rangle.
\]
It follows in particular that for all $s, r \in \mathbb{R}$ we have
\begin{equation}\label{Eqn=InPart}
\langle      \Psi_t^{\sigma_s^\varphi(a),c}(z_s) \Omega_\varphi,        z'_{s+r} \Omega_\varphi   \rangle = \langle z_s \Omega_\varphi \otimes \overline{z'_{s+r} \Omega_{\varphi}}, \zeta_t^{\sigma_s^\varphi(a), c} \rangle.
\end{equation}
Further   $\zeta_t^{\sigma_s^\varphi(a), c}$ is continuous in $s$ and in particular integrable (see the proof of Lemma \ref{Lem=CrossedBounded}).

 Now fix some $n \in \mathbb{N}$ and assume that the frequency support of $x$ is contained in $[-2n, 2n]$.
 Let $K$ be the product of the sets
 \begin{equation}\label{Eqn=K'}
 {\rm ssupp}(f_1), {\rm ssupp}(f_2),   \qquad {\rm and } \qquad [-2n, 2n].
 \end{equation}
 where the symmetric support is defined as ${\rm ssupp}(f) = {\rm supp}(f) \cup {\rm supp}(f)^{-1} \cup \{ 0 \}$.
Set for $t>0$,
\[
 \widetilde{\zeta}_t^{a,c}(s,s+r) =
 \left\{
 \begin{array}{ll}
 g_1(r)  \zeta_t^{\sigma_s^\varphi(a), c}, &  s \in K, r \in {\rm sup}(g_1), \\
 0 & {\rm otherwise}.
 \end{array}
 \right.
 \]
 Then $\widetilde{\zeta}_t^{a,c}$ defines an element of $L_2(\mathbb{R}^2, \cM \otimes \cM) \simeq L_2(c_{\varphi}(\cM)) \otimes L_2(c_{\varphi}(\cM))$. Note that $(z_s)_{s \in \mathbb{R}}$, hence $(\Psi_t^{\sigma_s^\varphi(a),c}(z_s))_{s \in \mathbb{R}}$,  is supported on the product of the sets ${\rm ssup}(f_1), {\rm ssup}(f_2)$ and $[-2n, 2n]$. Hence
 \begin{equation}\label{Eqn=XorYisOne}
 \begin{split}
\eqref{Eqn=ThreeOneSeries} = &
\int_{\mathbb{R}}  \int_{\mathbb{R}}   \langle    \Psi_t^{\sigma_s^\varphi(a),c}(z_s) \Omega_\varphi ,   g_1(r) z_{s+r}' \Omega_\varphi    \rangle dr ds \\
 = &
\int_{K}  \int_{{\rm sup}(g_1)}   \langle    \Psi_t^{\sigma_s^\varphi(a),c}(z_s) \Omega_\varphi ,   g_1(r) z_{s+r}' \Omega_\varphi    \rangle dr ds \\
= &
\int_{K} \int_{{\rm sup}(g_1)}   \langle  z_s \Omega_\varphi  \otimes \overline{ z_{s+r}' \Omega_\varphi },  g_1(r) \zeta_t^{\sigma_s^\varphi(a),c} \rangle  dr ds \\
=&  \langle  (z_s \Omega_\varphi)_{s \in \mathbb{R}} \otimes \overline{(z_{s}' \Omega_\varphi)_{r \in \mathbb{R}}},  \widetilde{\zeta}_t^{a,c} \rangle \\
= &  \langle
u_{f_2}^\ast  x u_{f_1}   \Omega_{\widetilde{\varphi}} \otimes \overline{    d u_{g_2}  \Omega_{\widetilde{\varphi}}  y^\ast    \sigma_{i/2}^{\widetilde{\varphi}}(b)^\ast  }, \widetilde{\zeta}_t^{a,c} \rangle \\
= &  \langle
 x u_{f_1}   \Omega_{\widetilde{\varphi}} \otimes \overline{    d u_{g_2}  \Omega_{\widetilde{\varphi}}  y^\ast  }, (u_{f_2}  \otimes 1) \widetilde{\zeta}_t^{a,c} (  \overline{ 1\otimes  \sigma_{i/2}^{\widetilde{\varphi}}(b) }  ) \rangle.
\end{split}
\end{equation}
Now let $\chi_n = \frac{1}{\sqrt{2n } }\chi_{[-n, n]} \in L_\infty(\mathbb{R})$. Let $m_n = \chi_n \ast \chi_n$, which is positive definite and converges to 1 uniformly on compacta. Let $T_{m_n}: L_\infty(\mathbb{R}) \rightarrow L_\infty(\mathbb{R})$ be the Fourier multiplier with symbol $m_n$ and by the Bozejko-Fendler theorem \cite{BozejkoFendler}, \cite{JungeNeufangRuan} let $M_{m_n}: B(L_2(\mathbb{R})) \rightarrow  B(L_2(\mathbb{R}))$ be its extension to $B(L_2(\mathbb{R}))$ as a normal $L_\infty(\mathbb{R})$-bimodule map (i.e. the so-called Herz-Schur multiplier).
Then $\Id_{\cM} \otimes M_{m_n} \rightarrow \Id_{\cM} \otimes \Id_{B(L_2(\mathbb{R}))}$ in the point-$\sigma$-weak topolgy.  Restricting $\Id_{\cM} \otimes M_{m_n}$ from $\cM \otimes B(L_2(\mathbb{R}))$ to $c_\varphi(\cM)$ yields a normal completely positive map
\[
R_n: c_\varphi(\cM) \rightarrow c_\varphi(\cM) \qquad \textrm{ given by } \qquad \pi_\varphi(v) u_s \mapsto m_n(s) \pi_\varphi(v) u_s, v \in \cM, s \in \mathbb{R}.
\]
The range of $R_n$ is contained in the  elements with frequency support in $[-2n, 2n]$.
 Fix $n$ and put $K$ as before \eqref{Eqn=K'}.
 It follows from \eqref{Eqn=XorYisOne} that for $t>0$ the inner product functional
\begin{equation}\label{Eqn=BdTensor}
 c_\varphi(\cM) \odot c_\varphi(\cM)^{\op} \ni
x  \otimes y^{\op}  \mapsto \langle     \widetilde{\Psi}_t^{a,c}( u_{f_2}^\ast R_n( x)  u_{f_1})   \Omega_{\widetilde{\varphi}}, d u_{g_2}  \Omega_{\widetilde{\varphi}}  y^\ast    \sigma_{i/2}^{\widetilde{\varphi}}(b)^\ast u_{g_1}^\ast  \rangle
\end{equation}
extends to a normal bounded map on the von Neumann algebraic tensor product $c_\varphi(\cM) \otimes c_\varphi(\cM)^{\op}$. Now let  $\xi = \sum_i u_{f_i} a_i  \otimes b_i u_{g_i} \Omega_{\widetilde{\varphi}}$ where the sum is finite and  $a,b \in \cA, f_i, g_i \in C_{00}(\mathbb{R})$. By \eqref{Eqn=BdTensor} we see that the positive map (for positivity, see the proof of Lemma \ref{Lem=Psi})
\[
\omega_{n,t}:c_\varphi(\cM) \odot c_\varphi(\cM)^{\op} \ni x  \otimes y^{\op}  \mapsto \sum_{i,j}   \langle     \widetilde{\Psi}_t^{a_i,a_j}( u_{f_j}^\ast R_n( x)  u_{f_i})   \Omega_{\widetilde{\varphi}}, b_j u_{g_j}  \Omega_{\widetilde{\varphi}}  y^\ast    \sigma_{i/2}^{\widetilde{\varphi}}(b_i)^\ast u_{g_i}^\ast  \rangle,
\]
extends to a normal bounded map on the von Neumann algebraic tensor product $c_\varphi(\cM) \otimes c_\varphi(\cM)^{\op}$. Moreover, by Kaplansky this extension is positive. Since $L_2(c_\varphi(\cM)) \otimes L_2(c_\varphi(\cM))$ is the standard Hilbert space of $c_\varphi(\cM) \otimes c_\varphi(\cM)^{\op}$ there exists $\eta_{n,t} \in L_2(c_\varphi(\cM)) \otimes L_2(c_\varphi(\cM))$ such that for every $x, y \in c_\varphi(\cM)$ we have
\begin{equation}\label{Eqn=Est1}
\omega_{n,t}(x \otimes y^{\op}) = \langle x \eta_{n,t} y, \eta_{n,t} \rangle.
\end{equation}

We can now conclude the proof as follows.
Now let $\varepsilon > 0$ and let $F$ be a finite subset in the unit ball of $c_{\varphi}(\cM)$. Let $\xi = \sum_i u_{f_i} a_i  \otimes b_i u_{g_i} \Omega_{\widetilde{\varphi}}$ where the sum is finite and  $a,b \in \cA, f_i, g_i \in C_{00}(\mathbb{R})$. Since $R_n \rightarrow {\rm Id}_{c_\varphi(\cM)}$ in the point-$\sigma$-weak topology we may take $n \in \mathbb{N}$ large such that  for all $x,y \in F$ we have
\begin{equation}\label{Eqn=Est2}
\vert \langle x \xi y, \xi \rangle_{\widetilde{\partial}} - \langle R_n(x) \xi y, \xi \rangle_{\widetilde{\partial}} \vert < \varepsilon.
\end{equation}
Recall from \eqref{Eqn=DensityCross} that $\omega_{n,0}(x \otimes y^{\op}) = \langle R_n(x) \xi y, \xi \rangle_{\widetilde{\partial}}$.
We may take $t>0$ small such that for all $x,y \in F$,
\begin{equation}\label{Eqn=Est3}
\vert \omega_{n,0}(x \otimes y^{\op}) - \omega_{n,t}(x \otimes y^{\op})  \vert < \varepsilon
\end{equation}
Combining \eqref{Eqn=Est1}, \eqref{Eqn=Est2}, \eqref{Eqn=Est3} we find that for all $x,y \in F$ we have
\[
\vert \langle x \xi y, \xi \rangle_{\widetilde{\partial}} -  \langle x \eta_{n,t} y, \eta_{n,t} \rangle \vert < 2 \varepsilon.
\]
As the vectors $\xi$ as above are dense in $\cH_{\partial, c_\varphi}$ it follows that the bimodule $\cH_{\partial, c_\varphi}$ is weakly contained in the coarse bimodule of $c_\varphi(\cM)$.

\vspace{0.3cm}

	 \noindent {\it Proof of \eqref{Item=Crossed1} and \eqref{Item=Crossed2} for arbitrary $p \in \mathcal{L}_\varphi(\mathbb{R})$. }
 	Now let $p \in \cL_\varphi(\mathbb{R})$ be a   projection.
 	 Then we see that we have a weak containment of the $p c_{\varphi}(\cM) p$-$p c_{\varphi}(\cM) p$-bimodules $p \cH_{c_\varphi, \partial} p$ in  $p L_2(c_\varphi(\cM)) \otimes  L_2(c_\varphi(\cM)) p$. The latter is in turn weakly contained in  $p L_2(c_\varphi(\cM))p \otimes  pL_2(c_\varphi(\cM)) p$, which is justified by the following. If $c_\varphi(\cM)$ were to be a factor we write $1 = \vee_n p_n$ with $p_n$ projections with $\widetilde{\tau}(p_n) = \widetilde{\tau}(p)$; by comparison of projections there are unitaries $u_n$ such that $u_n^\ast u_n = p_n$ and $u_n u_n^\ast = p$. Then $\xi \mapsto \xi u_n$ (resp. $\xi \mapsto u_n^\ast \xi$) intertwines the left (resp. right) action of $c_\varphi(\cM)$ on $L_2(c_\varphi(\cM)) p$ and $L_2(c_\varphi(\cM)) p_n$ (resp. $p L_2(c_\varphi(\cM))$ and $p_n L_2(c_\varphi(\cM))$).  	
 	  From this the weak containment follows in the factorial case. In general it follows from desintegration to factors of $c_\varphi(\cM)$.

\end{proof}

\section{The quantum group $O_N^+(F)$ with $F \overline{F} \in \mathbb{R} {\rm Id}_N$ admits an IGHS Markov  semi-group}\label{Sect=Solid1}

In this section we make an analysis of semi-groups associated with $L_\infty(O_N^+(F))$ and its associated gradient bimodule. The idea is based on results from \cite{CFY} where De Commer, Freslon and Yamashita have obtained the Haagerup property for $O_N^+(F)$. We use general results from \cite{JolissaintMartin}, \cite{CaspersSkalski}  to construct a semi-group for such $O_N^+(F)$ that is IGHS.




 \subsection{Semi-groups and Dirichlet forms, case $F \overline{F} \in \mathbb{R} \id_N$}

  Set  $\cA(O_N^+(F))$ to be  the underlying Hopf algebra of coefficients of finite dimensional representations of $O_N^+(F)$.  Recall that in case  $F \overline{F} \in \mathbb{R} \Id_N$ we have $\Irr(\ONplus) \simeq \mathbb{N}$.
Let $U_\alpha, \alpha \in \mathbb{N}$ be the (dilated) {\it Chebyshev polynomials} of the second kind. They are defined by $U_0(x) = 1, U_1(x) = x$ and the recursion relation
\[
x U_\alpha(x) = U_{\alpha-1}(x) + U_{\alpha+1}(x), \qquad \alpha \geq 1.
\]
 Let $U'_\alpha$ be the derivative of $U_\alpha$.

The arguments in the proof of Proposition \ref{Prop=GeneratorCheby} below are close to constructions from \cite{JolissaintMartin}, \cite{Sauvageot} and its non-tracial generalization    \cite[Proposition 5.5]{CaspersSkalski}. We use these ideas to obtain   a specific generator of a Markov semi-group  that can   be expressed in terms of the Chebyshev polynomials.

We need the fact that if $P$ is a  function that is smooth in a neighbourhood of 0 then,
\begin{equation}\label{Eqn=PolDeri}
\lim_{k \rightarrow \infty}  k \left( - P(0) +   \frac{1}{k} \sum_{l= k+1}^{2k} P\left(\frac{1}{l}  \right)   \right) = \log(2) P'(0).
\end{equation}
Recall that throughout the entire paper we made the convention that $0 < q \leq 1$ is fixed by the property $q + q^{-1} = {\rm Tr}(F^\ast F) = N_q$.
Define,
\begin{equation}\label{Eqn=DeltaD}
\Delta_\alpha = \frac{U_\alpha'(q + q^{-1})}{U_\alpha(q + q^{-1})}, \qquad \alpha \in \mathbb{N}.
\end{equation}

\begin{lem} \label{Lem=DeltaExplicit}
	We have,
	\[
	\Delta_\alpha = \frac{\alpha}{\sqrt{N_q^2 - 4}}  \left( \frac{1 + q^{-2\alpha-2}}{1 - q^{-2\alpha-2}} \right) + \frac{2}{(1 - q^2)  \sqrt{N_q^2 - 4}} .
	\]
	where $N_q = q + q^{-1}$ is the quantum dimension of the fundamental representation of $\ONplus$.
\end{lem}
\begin{proof}
	This is shown in \cite[Lemma 4.4]{FimaVergnioux} (in fact it can be derived rather directly from the recursion relation of $U_\alpha$).
\end{proof}

\begin{prop}\label{Prop=GeneratorCheby}
	Assume that $F \overline{F} \in \mathbb{R} \Id_N$.
	There exists a Markov semi-group $(\Phi_t)_{t \geq 0}$ on $L_\infty(\ONplus)$ determined by $\Phi_t(u_{i,j}^\alpha) = \exp( -t \Delta_\alpha) u_{i,j}^\alpha$.
\end{prop}
\begin{proof}
 	In \cite[Theorem 17]{CFK} it was proved that for every $-1 < t < 1$ we have that,
	\[
	\Upsilon_t(u_{ij}^\alpha) =  c_d(t)   u_{ij}^\alpha,  \quad  \textrm{ with }  \quad c_\alpha(t) = \left(\frac{U_\alpha( q ^t + q^{-t})}{U_\alpha(q + q^{-1})}  \right)^3,\qquad \alpha \in \mathbb{N}, 1 \leq i,j \leq n_\alpha,
	\]
	determines a normal unital completely positive multiplier on $L_\infty(\ONplus)$. Note that  the maps $\Upsilon_t$ with $-1 < t <1$ mutually commute. Moreover, for $x \in L_\infty(\ONplus)$ we have $\Upsilon_t(x) \rightarrow x$ $\sigma$-weakly as $t \nearrow 1$.  	
	 Put $\gamma_k = \frac{k}{\log(2)}  (1 -\frac{1}{k} \sum_{j=k+1}^{2k} \Upsilon_{1 - j^{-1}}^{(2)} )$. The proof of \cite[Proposition 5.5]{CaspersSkalski} argues that we may define semi-groups of completely positive contractions $S_{t,k} = \exp(-t \gamma_k)$ on $L_2(\cM)$. Further,
	\[
	S_{t, k}(u_{i,j}^\alpha) = \exp(-t \gamma_k)(u_{i,j}^\alpha) =
	  \exp\left( \frac{-tk}{\log(2)} \left( 1 -  \frac{1}{k} \sum_{l = k+1}^{2k} c_\alpha( 1 - \frac{1}{l} ) \right)  \right) (u_{i,j}^\alpha).
	\]
	Taking the limit $k \rightarrow \infty$ of this expression and using \eqref{Eqn=PolDeri} gives
	$\lim_{k \rightarrow \infty} S_{t,k}(u_{i,j}^\alpha)$
	 $= \exp( -t c_\alpha'(1))$
	  $u_{i,j}^\alpha$.
	By density we may conclude that for every $\xi \in L_2(O^{+}_N(F))$ we have that $S_{t,k}(\xi)$ is convergent say to $S_t(\xi)$. Furthermore $(S_t)_{t\geq 0}$ is a semi-group that is moreover strongly continuous (again this follows by comparing actions on $\mathcal{A}(\ONplus)$ in $L_2(\ONplus)$ and then using density).
	
	  Consider the closed convex sets  in $L_2(\ONplus)$  given by
     $C_0  =  \{ x \in L_2(\cM) \mid 0 \leq x \leq \Omega_{\varphi}   \}$ and the positive cone in the $i$-th matrix amplification $C_i =  M_i(L_2(\cM))^+$ where $i \in \mathbb{N}_{\geq 1}$. As for each $t,n$ and $i$ we have  $S_{t,n}(C_i) \subseteq C_i$  we get  $S_t(C_i) \subseteq C_i$. Further $S_t(\Omega_{\varphi}) = \Omega_{\varphi}$.
     Lemma \ref{Lem=MarkovLift} then shows that  there exists a Markov semi-group $(\Phi_t)_{t \geq 0}$ on $L_\infty(\ONplus)$ such that   $\Phi_t^{(2)} = (S_t)_{t \geq 0}$. As,
	\[
	c'_\alpha(t) = \frac{3  U_\alpha(q^t + q^{-t})^2 U_\alpha'(q^t + q^{-t}) (q^t \log(q) + q^{-t} \log(q^{-1})) }{U_\alpha(q + q^{-1})^3},
	\]
	we see that
	\[
	c_\alpha'(1) = \frac{3  U_\alpha'(q^1 + q^{-1}) (q \log(q) - q^{-1} \log(q)) }{U_\alpha(q + q^{-1})}.
	\]
	So the proposition follows by scaling the generator of the semi-group $(S_t)_{t \geq 0}$.
\end{proof}

The following is now another example of \cite[Theorem 6.7]{CaspersSkalski}.

\begin{cor}\label{Cor=ChebyForm}
Assume that	$F \overline{F} \in \mathbb{R} \Id_N$. There exists a conservative completely Dirichlet form $Q_N$ associated with $\ONplus$ with domain,
	\[
	\Dom(Q_N) = \{ \xi \in L_2(O_N^+(F))  \mid  \sum_{\alpha =0}^\infty \sum_{i,j = 1}^{n_\alpha} \Delta_\alpha \vert  \langle e_{i,j}^\alpha, \xi   \rangle \vert^2 < \infty  \},
	\]
	that is given by $
	Q_N(\xi) =  \sum_{\alpha=1}^\infty \sum_{i,j = 1}^{n_\alpha} \Delta_\alpha \vert  \langle e_{i,j}^\alpha, \xi   \rangle \vert^2$. Here $\Delta_\alpha$ is defined in \eqref{Eqn=DeltaD}.
 \end{cor}
\begin{proof}
This is a direct consequence of the correspondence between conservative Dirichlet forms and Markov semi-groups, see Section \ref{Sect=Dirichlet} and \cite[Section 6]{CaspersSkalski}.
\end{proof}

 \subsection{Properties IGHS and GHS} We prove that the Markov semi-group constructed in Proposition \ref{Prop=GeneratorCheby} is IGHS and even GHS in the non-tracial case.

 \begin{lem}\label{Lem=DeltaGradientEstimate}
 	For $\alpha,\beta,\gamma \in \mathbb{N}$ with $\vert \gamma \vert \leq \max(\alpha, \beta)$ we have
 	\begin{equation}\label{Eqn=DeltaFirstEstimate}
 	\vert \Delta_{\alpha+\gamma} - \Delta_\alpha - (\Delta_\beta - \Delta_{\beta-\gamma}) \vert \preceq  \gamma (q^{2\alpha + 2\gamma}  -  q^{2\beta +2\gamma}) + \beta  (q^{2\beta} - q^{2\beta - 2\gamma})   + \alpha (q^{2\alpha} - q^{2\alpha + 2 \gamma}),
 	\end{equation}
 	where $\preceq$ stands for an inequality that holds up to a constant that does not depend on $\alpha, \beta$ and $\gamma$.
 \end{lem}
 \begin{proof}
 	For each $m,n \in \mathbb{Z} \backslash \{ 0 \}$ we have that,
 	\[
 	\frac{1 + q^{-2 m }}{ 1-q^{-2m }} - 	   \frac{1 + q^{-2 n }}{ 1-q^{-2n }}
 	= \frac{2 (q^{-2m}  - q^{-2n} )}{(1-q^{-2m})(1-q^{-2n})} = \frac{2 (q^{2n}  - q^{2m} )}{(q^{2m} -1)(q^{2n} - 1)}.
 	\]
 	Then we have from Lemma \ref{Lem=DeltaExplicit},
 	\[
 	\begin{split}
 	&	\sqrt{N_q^2 - 4} \vert  (\Delta_{\alpha+\gamma} - \Delta_\alpha) - (\Delta_\beta - \Delta_{\beta-\gamma})  \vert \\
 	\leq & \left|  (\alpha+\gamma) \frac{1 + q^{-2\alpha-2\gamma-2}}{1 - q^{-2\alpha-2\gamma-2}}
 	-  \alpha \frac{1 + q^{-2\alpha-2}}{1 - q^{-2\alpha-2}}
 	+ (\beta-\gamma)   \frac{1 + q^{-2\beta+2\gamma-2}}{1 - q^{-2\beta+2 \gamma-2}}
 	- \beta  \frac{1 + q^{-2\beta-2}}{1 - q^{-2\beta-2}}
 	\right| \\
 	= &  \gamma \left|   \frac{1 + q^{-2\alpha-2\gamma-2}}{1 - q^{-2\alpha-2\gamma-2}}  - \frac{1 + q^{-2\beta+2\gamma-2}}{1 - q^{-2\beta+2\gamma-2}}     \right|
 	+ \beta  \left|  \frac{1 + q^{-2\beta-2}}{1 - q^{-2\beta-2}}  -   \frac{1 + q^{-2\beta+2\gamma-2}}{1 - q^{-2\beta+2\gamma-2}} \right|  \\
 	&  +  \alpha \left|  \frac{1 + q^{-2\alpha-2\gamma-2}}{1 - q^{-2\alpha-2\gamma-2}} -\frac{1 + q^{-2\alpha-2}}{1 - q^{-2\alpha-2}}    \right| \\
 	\preceq &  \gamma \vert q^{2 \alpha + 2\gamma}  -  q^{2\beta +2\gamma} \vert + \beta  \vert q^{2\beta} - q^{2\beta - 2\gamma} \vert    + \alpha \vert q^{2\alpha} - q^{2\alpha + 2\gamma} \vert. \\
 	\end{split}
 	\]
 	This shows  \eqref{Eqn=DeltaFirstEstimate}.
 \end{proof}

 Assume $F \overline{F} \in \mathbb{R} \id_N$. Then let $\Delta^{\frac{1}{2}}$ be the unique unbounded operator with domain $\Dom(Q_N)$  such that $Q_N(\xi) = \langle \Delta^{\frac{1}{2}} \xi, \Delta^{\frac{1}{2}} \xi \rangle$ (c.f. Corollary \ref{Cor=ChebyForm}).  In other words
 \[
 \Delta^{\frac{1}{2}} = \bigoplus_\alpha \Delta_\alpha^{\frac{1}{2}} p_\alpha
 \]
 where $p_\alpha$ is the projection of $L_2(\cM)$ onto the isotypical component of $\alpha \in \Irr(\ONplus) \simeq \mathbb{N}$. Then let $\langle \: \cdot \:, \: \cdot \: \rangle_\Gamma$ be the gradient form defined in \eqref{Eqn=Gradient} with respect to this $\Delta$.
 Let $\cH_\partial$ be the $L_\infty(\ONplus)$-$L_\infty(\ONplus)$ bimodule constructed in Section \ref{Sect=Gradient} starting from the semi-group $(\Phi_t)_{t \geq 0}$ and corresponding Dirichlet form of Proposition \ref{Prop=GeneratorCheby} and Corollary \ref{Cor=ChebyForm}. The algebra $\mathcal{A}$ in Section \ref{Sect=Gradient} is then understood as $\mathcal{A}(\ONplus)$.

 The following lemma is directly based on estimates of  Jones-Wenzl projections. The  estimate we need was precisely proved in \cite[Appendix A]{VaesVergnioux} already, c.f. \eqref{Eqn=JonesWenzl}. For $\alpha \in \mathbb{N}$ write $P_\alpha(x) = p_\alpha x p_\alpha$ for the isotypical cut-down.  For $\alpha, \beta, \gamma \in \mathbb{N}$ the fusion rules of $\ONplus$ imply that if $\gamma \leq \vert \alpha - \beta \vert$  and $\gamma - \alpha + \beta \in 2 \mathbb{Z}$ then $\gamma$ is contained in $\alpha \otimes \beta$ with multiplicity 1. We shall write $V^{\alpha, \beta}_\gamma: \cH_\gamma \rightarrow \cH_\alpha \otimes \cH_\beta$ for the isometry that intertwines $\gamma$ with $\alpha \otimes \beta$. By Peter-Weyl theory  $V^{\alpha, \beta}_\gamma$ is uniquely determined up to a complex scalar of modulus 1. For the next lemma let $u_{i,j}^\alpha$ denote the matrix unit of $u^\alpha$ with respect to some orthogonal basis vectors which we simply denote by $1 \leq i,j \leq n_\alpha$. We have Peter-Weyl orthogonality relations
 \[
 \Vert u_{i,j}^\alpha \Vert_2 = \Vert Q_\alpha i \Vert_2 \Vert j \Vert_2,
 \]
 for some positive matrix $Q_\alpha \in M_{n_\alpha}(\mathbb{C})$ which may assumed to be diagonal after possibly changing the basis (see \cite[Proposition 2.1]{Daws}). Moreover we have,
 \[
 V^{\alpha, \beta}_\gamma Q_\gamma = (Q_\alpha \otimes Q_\beta)  V^{\alpha, \beta}_\gamma.
 \]

 \begin{lem}\label{Lem=JonesWenzl}
 	Assume $F \overline{F} \in \mathbb{R} \id_N$.
 	Take matrix coefficients $x = u_{i,j}^\alpha, a = u^r_{m',n'}, c = u^s_{m,n}$ where $\alpha, r,s \in \mathbb{N}$. Assume $r,s \leq \alpha$ and let $k, l \in \mathbb{Z}$ be such that $\vert k \vert \leq r$ and $\vert l \vert \leq s$ wit $k -r \in 2 \mathbb{Z}$ and $l -s \in 2\mathbb{Z}$. We have,
 	\begin{equation}\label{Eqn=JonesCom}
 	\Vert P_{\alpha + k +l}( P_{\alpha+k}(c  x) a) - P_{\alpha + k + l}( c  P_{\alpha + l}(x a)) \Vert_2 \preceq q^{\alpha} \Vert x \Vert_2  .
 	\end{equation}
 	Here $\preceq$ is an inequality that holds up to a constant only depending on $a,c$ and $q$.
 \end{lem}
 \begin{proof}
 	We prove this by induction on $s$ and $r$. If either $s=0$ or $r = 0$ the statement is clear as the left hand side of \eqref{Eqn=JonesCom} is 0

 \vspace{0.3cm}

 \noindent   {\bf Step 1. Case   $r=1$ and $s = 1$.}  We get the following.  	
 	We have,
 	\begin{equation}\label{Eqn=UEstimate1}
 	\begin{split}
 	&  P_{\alpha + k +l}( P_{\alpha + k}(c  x) a)
 	=   \langle u^{\alpha + k +l}  (V^{\alpha+k, r}_{\alpha + k +l})^\ast  (V_{ \alpha + k  }^{s,\alpha} \otimes 1_r)^\ast (m \otimes i \otimes m'), (V^{\alpha+k, r}_{\alpha+k+l})^\ast  (V_{\alpha+k  }^{s,\alpha} \otimes 1_r)^\ast (n \otimes j \otimes n') \rangle.
 	\end{split}
 	\end{equation}
 	Similarly,
 	\begin{equation}\label{Eqn=UEstimate2}
 	P_{\alpha + k +l}( c  P_{\alpha + l}(x a)) = \langle u^{\alpha+k+l}   (V^{s, \alpha +l}_{\alpha+k+l})^\ast (1_s \otimes V_{\alpha +l  }^{\alpha,r} )^\ast  (m \otimes i \otimes m' ),   (V^{s, \alpha +l}_{\alpha+k+l})^\ast (1_s \otimes V_{\alpha +l  }^{\alpha,r} )^\ast  (n \otimes j \otimes n') \rangle.
 	\end{equation}
 	By \cite[Lemma A.1, Eqn. (A.2)]{VaesVergnioux} we see that in case $l =1$ and $k \in \{-1, 1\}$, we have
 	\begin{equation}\label{Eqn=JonesWenzl}
 	\Vert  (1_s \otimes V_{\alpha +l  }^{\alpha,r}  )  V^{s, \alpha +l}_{\alpha+k+l}    -
 	(V_{ \alpha + k  }^{s,\alpha} \otimes 1_r)  V^{\alpha+k, r}_{\alpha + k +l}
 	\Vert \leq q^{\alpha + (k-r)/2}.
 	\end{equation}
 	In fact by \cite[Lemma A.2, Eqn. (A.5)]{VaesVergnioux} the left hand side of \eqref{Eqn=JonesWenzl} may also be estimated by $q^{\alpha + (l-r)/2}$ in case $l \in \{ -1, 1\}$ and $k = 1$. Therefore for any $k,l \in \{-1, 1\}$ except for $k = l = -1$ we may continu as follows.  We get,
 	\begin{equation}\label{Eqn=PLeft}
 	\begin{split}
 	&	 \Vert 	  \langle \: u^{\beta}  (  (V^{\alpha+k, r}_{\alpha + k +l})^\ast  (V_{ \alpha + k  }^{s,\alpha} \otimes 1_r)^\ast   -    (V^{s, \alpha +l}_{\alpha+k+l})^\ast (1_s \otimes V_{\alpha +l  }^{\alpha,r} )^\ast )(m \otimes i \otimes m'), \\
 	& \qquad   (V^{\alpha+k, r}_{\alpha + k +l})^\ast  (V_{ \alpha + k  }^{s,\alpha} \otimes 1_r)^\ast (n \otimes j \otimes n') \: \rangle
 	\Vert_2^2 \\
 	= & \Vert Q_{\beta} (  (V^{\alpha+k, r}_{\alpha + k +l})^\ast  (V_{ \alpha + k  }^{s,\alpha} \otimes 1_r)^\ast   -    (V^{s, \alpha +l}_{\alpha+k+l})^\ast (1_s \otimes V_{\alpha +l  }^{\alpha,r} )^\ast ) (m \otimes i \otimes m')
 	\Vert^2_2 \\
 	& \quad \times \quad \Vert  (V^{\alpha+k, r}_{\alpha + k +l})^\ast  (V_{ \alpha + k  }^{s,\alpha} \otimes 1_r)^\ast (n \otimes j \otimes n') \Vert^2_2 \\
 	\leq &  \Vert (  (V^{\alpha+k, r}_{\alpha + k +l})^\ast  (V_{ \alpha + k  }^{s,\alpha} \otimes 1_r)^\ast   -    (V^{s, \alpha +l}_{\alpha+k+l})^\ast (1_s \otimes V_{\alpha +l  }^{\alpha,r} )^\ast )  (Q_s m \otimes Q_\alpha i \otimes Q_r m')
 	\Vert^2_2   \Vert   (n \otimes j \otimes n' ) \Vert^2_2 \\
 	\leq &   q^{2\alpha} \Vert    (Q_s m \otimes Q_\alpha i \otimes Q_r m' )
 	\Vert^2_2  \Vert   n \otimes j \otimes n'  \Vert^2_2 \\
 	= & q^{2\alpha}   \Vert x \Vert_2^2 \Vert a \Vert_2^2 \Vert c \Vert_2^2.
 	\end{split}
 	\end{equation}
 	Similarly,
 	\begin{equation}\label{Eqn=PLeft}
 	\begin{split}
 	&	 \Vert 	  \langle \: u^{\beta}    (V^{s, \alpha +l}_{\alpha+k+l})^\ast (1_s \otimes V_{\alpha +l  }^{\alpha,r} )^\ast    (m \otimes i \otimes  m' ), \\
 	& \qquad   (  (V^{\alpha+k, r}_{\alpha + k +l})^\ast  (V_{ \alpha + k  }^{s,\alpha} \otimes 1_r)^\ast   -    (V^{s, \alpha +l}_{\alpha+k+l})^\ast (1_s \otimes V_{\alpha +l  }^{\alpha,r} )^\ast )(n \otimes j \otimes n') \: \rangle
 	\Vert_2^2 \leq q^{2\alpha}  \Vert x \Vert_2^2 \Vert a \Vert_2^2 \Vert c \Vert_2^2.
 	\end{split}
 	\end{equation}
 	Combining all the above estimates yields, still with $k,l \in \{ -1, 1\}$ but not $k = l = -1$,
 	\begin{equation}\label{Eqn=KLEstimate}
 	\Vert P_{\alpha + k +l}( P_{\alpha + k}(c  x) a)  -  	P_{\alpha + k +l}( c  P_{\alpha + l}(x a)) \Vert_2
 	\leq 2 q^{ \alpha }   \Vert x \Vert_2 \Vert a \Vert_2 \Vert c\Vert_2.
 	\end{equation}
 	But as we have that
 	\[
 	cxa = \sum_{k,l \in \{ -1, 1 \} }  P_{\alpha + k +l}( P_{\alpha + k}(c  x) a)  = \sum_{k,l \in \{ -1, 1 \} } 	P_{\alpha + k +l}( c  P_{\alpha + l}(x a))
 	\]
 	We can estimate the complementary case $k = l = -1$ through \eqref{Eqn=KLEstimate} by
 	\[
 	 	\Vert P_{\alpha + k +l}( P_{\alpha + k}(c  x) a)  -  	P_{\alpha + k +l}( c  P_{\alpha + l}(x a)) \Vert_2
 	\leq 6 q^{ \alpha }   \Vert x \Vert_2 \Vert a \Vert_2 \Vert c\Vert_2.
 	\]
 	This proves the lemma in case $s = r = 1$.
 	
 	\vspace{0.3cm}
 	
\noindent  	{\bf Step 2. Induction.} We prove that if the lemma holds for some   $r$ and $s-1$ as in the lemma, then it also holds for $r$ and $s$.  In particular it then follows that the lemma holds for $r =1$ and $s$ arbitrary.

 Let $a_1 = u^{s-1}_{m',n'}$ and let $a_2 = u^{1}_{m'', n''}$. Take $l_1 \in \mathbb{Z}$ with $\vert l_1 \vert \leq s-1$ and   $l_1 - s+1 \in 2\mathbb{Z}$. Further let $l_2 \in \{-1, 1\}$. Write $\preceq$ for an inequality that holds up to a constant that only depends on $a_1, a_2, c$ and $q$. We get by Step 1 that,
\[
\begin{split}
	& \Vert 	P_{\alpha + k +l_1+l_2}( P_{\alpha+l_1+k}( c   P_{\alpha + l_1}( x a_1))   a_2 )  -  	P_{\alpha + k +l_1+l_2}( c  P_{\alpha + l_1 + l_2}( P_{\alpha + l_1}( x a_1) a_2)) \Vert_2 \\
	\preceq & q^{\alpha+l_1} \Vert P_{\alpha + l_1}( x a_1) \Vert_2 \Vert c \Vert_2 \Vert a_2 \Vert_2
	\preceq q^{\alpha} \Vert x \Vert_2,
\end{split}
\]
 and
 \[
 \begin{split}
 & \Vert 	P_{\alpha + k +l_1+l_2}( P_{\alpha+l_1+k}( c P_{\alpha + l_1}( x a_1))   a_2 )  -  	P_{\alpha + k +l_1+l_2}( P_{\alpha+l_1+k}( P_{\alpha + k}(c    x )a_1)   a_2 ) \Vert_2 \\
 \preceq & q^{\alpha} \Vert x  \Vert_2  \Vert c \Vert_2 \Vert a_1 \Vert_2  \Vert a_2 \Vert \preceq q^{\alpha} \Vert x \Vert_2.
 \end{split}
 \]
 Hence,
 \[
 \begin{split}
 & \Vert 	P_{\alpha + k +l_1+l_2}( P_{\alpha+l_1+k}( P_{\alpha + k}(c    x )a_1)   a_2 )  -  	P_{\alpha + k +l_1+l_2}( c  P_{\alpha + l_1 + l_2}( P_{\alpha + l_1}( x a_1) a_2)) \Vert_2
 \preceq  q^{\alpha} \Vert x \Vert_2.
 \end{split}
 \]	
 Fix $l \in \mathbb{Z}$ with $\vert l \vert \leq s$ and $l -s \in 2\mathbb{Z}$. Taking the sum over all $l_1$ and $l_2$ with  $l_1 + l_2 = l$ we see
 \begin{equation}\label{Eqn=A1A2}
 \begin{split}
 & \Vert 	P_{\alpha + k +l}(   P_{\alpha + k}(c  x )a_1   a_2 )  -  	P_{\alpha + k +l}( c  P_{\alpha + l}(   x a_1 a_2)) \Vert_2
 \preceq  q^{\alpha} \Vert  x \Vert_2.
 \end{split}
  \end{equation}
  Since $a_1$ and $a_2$ were arbitrary coefficients of $u^{s-1}$ and $u^1$ respectively we get that \eqref{Eqn=A1A2} holds with $a_1 a_2$ replaced by any $a$ that is a matrix coefficient of $u^{(s-1) \otimes 1}$.
    Since we have an inclusion of irreducible representations $s \subseteq (s-1) \otimes 1$ we conclude our claim.

 	\vspace{0.3cm}

\noindent  	{\bf Step 3. Case $r$ and $s$ arbitrary as in the lemma.} One may proceed as in Step 2 to conclude the proof. Alternatively, assume the lemma is proved for $r-1$ and $s$. We want to show that it holds for $r$ and $s$. We have,
\[
\Vert P_{\alpha + k +l}( P_{\alpha+k}(c  x) a) - P_{\alpha + k + l}( c  P_{\alpha + l}(x a)) \Vert_2
=
\Vert P_{\alpha + k +l}( P_{\alpha+l}(c^\ast  x^\ast) a^\ast) - P_{\alpha + k + l}( a^\ast  P_{\alpha + k}(x^\ast c^\ast)) \Vert_2
\]
 Recall that every element in $\Irr(O_N^+(F))$ is equivalent to its contragredient representation.
 So by the inductive step in Step 2 of the proof with the roles of $s$ and $r$ interchanged we see that the right hand side may be estimated by a constant only depending on $a, c$ and $q$ times $q^{\vert \alpha \vert}$.

\end{proof}

The next lemma is now crucial. The fact that in the non-tracial case the Hilbert-Schmidt properties of the maps $\Psi_t$ in this lemma are better comes from the fact that the intertwining properties of Lemma \ref{Lem=JonesWenzl} are stronger.

 \begin{lem}\label{Lem=HS}
 	Assume that $F \overline{F} \in \mathbb{R} \id_N$.
 	Let $a, b  \in \cA(\ONplus)$. For $t \geq 0$ consider the linear map $\cA(\ONplus) \rightarrow \cA(\ONplus)$ defined by
 		\[
 		\Psi_t : = \Psi_t^{a, b}:  x \mapsto   \Phi_t\left(    \langle    xa  , b   \rangle_{\Gamma} - \langle    x  , b   \rangle_{\Gamma} a  \right).
 		\]
 		For $t \geq 0$ consider the map    $\Psi_t^{(l,2)}: L_2(\cM) \rightarrow L_2(\cM): x \Omega_\varphi \mapsto \Psi_t(x) \Omega_\varphi, x \in \cA$.    If $t > 0$ then $\Psi_t^{(l,2)}$ extends to a Hilbert-Schmidt map. Moreover, if $F \not = \Id_N$ then $\Psi_t^{(l,2)}$ extends to a Hilbert-Schmidt map also for $t = 0$.
 \end{lem}
 \begin{proof}
    Let $a$ and $b$  in $\cA$ be coefficients of respectively irreducible representations $u^r$ and $u^s$ with $r,s \in \mathbb{N}$. By linearity it suffices to show that for $t >0$ (and $t=0$ if $F \not = \Id_N$) the map,
    \[
     \Psi_t(x) = \Phi_t\left(    \langle    xa  , b   \rangle_{\Gamma} - \langle    x  ,b  \rangle_{\Gamma} a   \right)
    \]
    is Hilbert-Schmidt.
 	Let $x = u_{i,j}^{\alpha}$ with  $\alpha \in \mathbb{N}$. Firstly, we have
 	\[
 	\begin{split}
 	\langle    xa  , b   \rangle_{\Gamma} - \langle    x  , b   \rangle_{\Gamma}a
 	= &  b^\ast \Delta(xa) - \Delta(b^\ast x a)
 	- b^\ast \Delta(x)a  + \Delta(b^\ast x) a.
 	\end{split}
 	\]
 	Note that each isotypical projection $P_\gamma, \gamma \in \mathbb{N}$  commutes with $\Delta$ which we may naturally view as a map $\cA(\ONplus) \rightarrow \cA(\ONplus)$.   From  the fusion rules of $\ONplus$ we conclude the following for numbers $\gamma \in \mathbb{N}$. If  $\alpha+\gamma \subseteq \alpha \otimes r$ then  $\vert \gamma \vert \leq r$. If $\alpha+\gamma \subseteq s \otimes \alpha$ then $\vert \gamma \vert \leq s$. For $\vert \gamma \vert \leq r$ and $\beta \subseteq s \otimes (\alpha + \gamma)$ we have $\vert \beta -\alpha \vert \leq r+s$. Finally for $\vert \gamma \vert \leq s$ and $\beta \subseteq (\alpha-\gamma) \otimes r$ we have $\vert \beta -\alpha \vert \leq r+s$. These observations show that we get the following sum decomposition. Some summands can be zero; in fact all that matters is that the summation is finite. So,
 	\[
 	\begin{split}
 	&	\langle    xa  , b   \rangle_{\Gamma} - \langle    x  , b   \rangle_{\Gamma}a \\
 	= & \sum_{\substack{  \alpha-r-s \leq \beta \leq  \alpha  + r+s \\  - \max(r,s)  \leq \gamma \leq     \max(r,s) }}
 	P_{\beta} \left(b^\ast \Delta( P_{\alpha+\gamma} ( xa)  ) - \Delta( P_{\alpha-\gamma} (b^\ast x) a)
 	- b^\ast P_{\alpha+\gamma} (\Delta(x)a)   + \Delta( P_{\beta-\gamma} (b^\ast x)) a \right) \\
 	= & \sum_{\substack{  \alpha-r-s \leq \beta \leq  \alpha  + r+s \\  - \max(r,s)  \leq \gamma \leq     \max(r,s) }}
 	(\Delta_{\alpha+\gamma} - \Delta_\alpha)	
 	P_{\beta} (b^\ast  P_{\alpha+\gamma} ( xa)  ) - (\Delta_\beta - \Delta_{\beta-\gamma})  P_{\beta} (  P_{\alpha-\gamma} (b^\ast x) a).
 	\end{split}
 	\]
 	We therefore obtain for $t > 0$ that,
 	\begin{equation}\label{Eqn=LongEstimate}
 	\begin{split}
 	&\Vert \Phi_t \left(  \langle    xa  , b   \rangle_{\Gamma} - \langle    x  , b   \rangle_{\Gamma}a  \right)\Vert_2 \\
 	\leq &
 	\sum_{\substack{  \alpha -r-s \leq \beta \leq  \alpha  + r+s \\  - \max(r,s)  \leq \gamma \leq     \max(r,s) }}
 	\vert \Delta_{\alpha+\gamma} - \Delta_\alpha  - \Delta_\beta + \Delta_{\beta-\gamma} \vert	
 	\Vert \Phi_t (P_{\beta} (b^\ast  P_{\alpha+\gamma} ( xa)  ) ) \Vert_2 \\
 	& \qquad \qquad  +  \vert \Delta_\beta - \Delta_{\beta-\gamma} \vert \Vert \Phi_t\left(  P_{\beta} (b^\ast  P_{\alpha+\gamma} ( xa)  ) -
 	P_{\beta} (  P_{\alpha-\gamma} (b^\ast x) a)  \right) \Vert_2 \\
 	\end{split}
 	\end{equation}
 	We write $\preceq$ for an inequality that holds up to some constant independent of   $\alpha$.  Let $\gamma, \alpha, \beta$ be such that
 	$\vert \beta - \alpha \vert \leq  r+s$ and  $\vert \gamma \vert \leq  \max(r,s)$.
 	Lemma \ref{Lem=DeltaGradientEstimate} shows that,
 	\begin{equation}\label{Eqn=Estimate2}
 	\vert \Delta_{\alpha+\gamma} - \Delta_\alpha  - \Delta_\beta + \Delta_{\beta-\gamma} \vert \preceq q^{2\alpha}.
 	\end{equation}
 	As the eigenvalues of $\Delta$ grow asymptotically linear, more precisely Lemma \ref{Lem=DeltaExplicit}, we have the following.
 	\begin{equation}\label{Eqn=Estimate1}
 	\vert \Delta_\beta - \Delta_{\beta-\gamma} \vert \preceq \max(r,s) \preceq 1 \quad {\rm and } \quad
 	\exp(-t \Delta_\beta) \preceq \exp(-t (\Delta_\alpha - r - s )) \preceq \exp(-t  \alpha).
 	\end{equation}
 	By Lemma \ref{Lem=JonesWenzl} (note that $b^\ast$ is a coefficient of the contragredient of $u^s$ which is equivalent to $u^s$ itself),
 	\begin{equation}\label{Eqn=Estimate3}
 	\Vert   P_{\beta} (b^\ast  P_{\alpha+\gamma} ( xa)  ) -
 	P_{\beta} (  P_{\alpha-\gamma} (b^\ast x) a)   \Vert_2 \preceq q^{\alpha} \Vert x \Vert_2.
 	\end{equation}
 	Combining \eqref{Eqn=LongEstimate} with the estimates from \eqref{Eqn=Estimate2}, \eqref{Eqn=Estimate1} and \eqref{Eqn=Estimate3} we see that,
 	\[
 	\begin{split}
 	\Vert \Phi_t \left(  \langle    xa  , b   \rangle_{\Gamma} - \langle    x  , b   \rangle_{\Gamma}a  \right)\Vert_2
 	\preceq 	&   q^{2 \alpha} \exp(-t \alpha)  \Vert x \Vert_2 +  q^\alpha \exp(-t \alpha)  \Vert x \Vert_2
 	\end{split}
 	\]

  Now let $\xi \in  \oplus_{\alpha = 0}^N P_\alpha(L_\infty(\bG))$ and let $\xi_\alpha = P_\alpha(\xi)$. Then
 \[
 \begin{split}
 & \Vert \Psi_0(\xi) \Vert_2 \leq
 \sum_{\alpha = 0}^N \Vert \Psi_0(  \xi_\alpha ) \Vert_2
 \preceq  \sum_{\alpha = 0}^N q^{\alpha} \Vert \xi_\alpha \Vert_2  \leq \left( \sum_{\alpha = 0}^\infty q^{2 \alpha} \right)^{\frac{1}{2}}
 \left( \sum_{\alpha = 0}^\infty \Vert \xi_\alpha \Vert_2^2  \right)^{\frac{1}{2}} = \frac{1}{1-q^2} \Vert \xi \Vert_2.
 \end{split}
 \]
 Hence $\Psi_0$ is bounded $L_2(\cM) \rightarrow L_2(\cM)$.
 	Further we get that,
 	\[
 	\begin{split}
 	\Vert \Psi_t \Vert_{HS}^2 = &
 	\sum_{\alpha \in \mathbb{N}, 1 \leq i,j \leq n_\alpha}
 	\frac{ \Vert \Psi_t( u_{i,j}^\alpha) \Vert_2^2}{\Vert u_{i,j}^\alpha \Vert_2^2} \preceq
 	\sum_{\alpha \in \mathbb{N}, 1 \leq i,j \leq n_\alpha}
 	(q^{2 \alpha} + q^{\alpha})^2\exp(- 2t \alpha)	\\
 	\leq & \sum_{\alpha \in \mathbb{N}} n_\alpha^2 q^{2\alpha} \exp(-2t \alpha)	= \sum_{\alpha \in \mathbb{N}}  (   n_\alpha^{\frac{2}{\alpha}}  q^{2}  \exp(-2t))^\alpha.
 	\end{split}
 	\]
 	As $n_\alpha^{\frac{2}{\alpha}}  q^{2}$ converges to a number $\leq 1$  (see Section \ref{Sect=Preliminaries}) for $\alpha \rightarrow \infty$ this summation is finite as soon as $t>0$  which concludes the proof. Moreover, if $F \not = \Id_N$ then $n_\alpha^{\frac{2}{\alpha}}  q^{2}$ converges to a number $<1$  (see Section \ref{Sect=Preliminaries})  which concludes that the latter summation is finite if $t = 0$.

 \end{proof}
As a direct consequence we get the following.
\begin{thm}\label{Thm=IGHSfirst}
	Assume that $F \overline{F} \in \mathbb{R} \id_N$. The semi-group of  Proposition \ref{Prop=GeneratorCheby} on $L_\infty(\ONplus)$ is IGHS. If moreover $F \not = \Id_N$ then this semi-group is GHS.
\end{thm}

\section{Strong solidity}\label{Sect=Solid2}

\subsection{$HH^+$-type properties and strong solidity in the tracial case}
At this point we collect results for the tracial case, i.e. $F = \id_N$. Write $O_N^+ = O_N^+(\id_N)$. We first obtain the following result, which is closely related to Property   (HH)$^+$ from \cite{OzawaPopaAJM} and its quantum version which was first studied in  \cite{FimaVergnioux}. In fact Corollary \ref{Cor=StrongHH} was already proved in \cite[Corollary 4.7]{FimaVergnioux} based on different methods.

\begin{dfn}
Let $\cM$ be a von Neumann algebra and let $\partial: \Dom(\partial) \rightarrow \cH$ be a derivation where $\Dom(\partial)$ is a subalgebra of $\cM$ and $\cH$ is an $\cM$-$\cM$-bimodule. $\partial$ is called {\bf closable} if the operator $\partial_2: \Dom(\partial) \Omega_\varphi \rightarrow \cH: x \Omega_\varphi \mapsto \partial(x)$ is closable as an (unbounded) operator $L_2(\cM, \varphi) \rightarrow \cH$. A closable derivation $\partial$ is called {\bf proper} if $\partial^\ast_2 \overline{\partial}_2$ has compact resolvent. With slight abuse of notation we will write $\partial$ for $\partial_2$ as was also done in \cite{OzawaPopaAJM} and \cite{Peterson}.
\end{dfn}

\begin{cor}\label{Cor=StrongHH}
	Assume that $F = \id_N$. There exists a proper closable derivation $\partial$ on $\cA(\ONplus)$ into a $L_\infty(O_N^+)$-$L_\infty(O_N^+)$-bimodule $\cH$ that is weakly contained in the coarse bimodule of $L_\infty(O_N^+)$.
\end{cor}
\begin{proof}
	By Proposition \ref{Prop=VNAModule} we see that the left and right $\cA(O_N^+)$-actions extend to normal actions of $L_\infty(O_N^+)$. 	 Theorem \ref{Thm=IGHSfirst}  and Proposition \ref{Prop=IGHSimpliesGC} imply that the gradient bimodule $\cH_\partial$ is weakly contained in the coarse bimodule. Then because we are in the tracial case the constructions from \cite{CiprianiSauvageot} which are recalled in Section \ref{Sect=Preliminaries} show that there exists a derivation $\partial: \cA(O_N^+) \rightarrow \cH_\partial$. Lemma \ref{Lem=PartialClosure} shows that this derivation is closable with suitable domain so that $\Delta = \partial^\ast \overline{\partial}$. Then Lemma \ref{Lem=DeltaExplicit} shows that $\partial$ is proper.
\end{proof}

The following Corollary \ref{Cor=StrongSolidityTrace} follows by a  modification of the arguments in \cite{OzawaPopaAJM} from groups to quantum groups. This fact was also suggested in the final remarks of \cite{FimaVergnioux}. For completeness and the fact that in the non-tracial case we also require this result (even for stable normalizers),  we included the proof in the appendix.

\begin{cor}\label{Cor=StrongSolidityTrace}
	Assume that $F = \id_N$ and $N \geq 3$. Then $L_\infty(O_N^+)$ is strongly solid.
\end{cor}
\begin{proof}
	This follows from the methods in \cite[Theorem B]{OzawaPopaAJM} (see Appendix \ref{Sect=AppendixA}) in combination with Corollary \ref{Cor=StrongHH} and the fact that  $L_\infty(O_N^+)$ has the CMAP \cite{CFY}.
\end{proof}

\subsection{Strong solidity for $\ONplus$ and $U_N^+(F)$, case of general $F$}
  Recall that for a matrix $F \in GL_n(\mathbb{C})$ the free unitary quantum group $U_N^+(F)$ is defined as follows.  As a C$^\ast$-algebra it is the algebra $\sfA$ freely generated by elements $u_{i,j}, 1 \leq i,j \leq N$ subject to the relation  that the matrix $u^{1} = (u_{i,j})_{i,j}$ is unitary and $u^1 = F u^{1} F^{-1}$. The comultiplication is then given by $\Delta_{\sfA}(u_{i,j}) = \sum_{k=1}^N u_{i,k} \otimes u_{k,j}$.
  When $F \overline{F} \in \mathbb{R} \id_N$ we have that $U_N^+(F)$ is a quantum subgroup of $\mathbb{Z} \ast \ONplus$ with Hopf $\ast$-algebra homomorphism
 \begin{equation}\label{Eqn=UnPlus}
 U_N^+(F) \rightarrow \mathbb{Z} \ast \ONplus: u_{i,j} \mapsto z u_{i,j},
 \end{equation}
 where $z$ denotes the identity function on $\mathbb{T} = \widehat{\mathbb{Z}}$. Further, Wang  \cite{WangJOT} proved the following decomposition results.
 For any $F \in GL_N(\mathbb{C})$ we have an isomorphism of quantum groups
\begin{equation}\label{Eqn=FreeDecomposition1}
U_N^+(F) \simeq U_{N_1}^+(D_1) \ast \ldots \ast U_{N_m}^+(D_m),
\end{equation}
and
\begin{equation}\label{Eqn=FreeDecomposition2}
O_N^+(F) \simeq U_{N_1}^+(D_1) \ast \ldots \ast U_{N_m}^+(D_m)  \ast O_{M_1}^+(E_1) \ast \ldots \ast O_{M_n}^+(F_n),
\end{equation}
for certain matrices $D_i$ and $E_i$ of dimension $N_i$  and $M_i$ smaller than $N$ respectively with the property that $D_i \overline{D}_i \in \mathbb{R} \id_{N_i}$ and $E_i \overline{E_i} \in \mathbb{R} \id_{M_i}$.

\begin{rmk}\label{Rmk=IGHS}
Recall that in Proposition \ref{Prop=GeneratorCheby} we constructed a Markov semigroup $(\Phi_t)_{t \geq 0}$ on $\ONplus$ in case $F \overline{F} \in \mathbb{R} \Id_N$. Then taking the free product with the identity semi-group on $L_\infty(\widehat{\mathbb{Z}})$ yields a semi-group on $\mathbb{Z} \ast \ONplus$ which restricts to  $U_N^+(F)$ under the embedding \eqref{Eqn=UnPlus}. The gradient module of the identity semi-group is the zero module which is clearly IGHS. Therefore by Proposition \ref{Prop=FreeProduct} the free product semigroup is IGHS on $\mathbb{Z} \ast \ONplus$ and hence on  $U_N^+(F)$. Take the free product of the latter semi-group on the $U_N^+$-factors in \eqref{Eqn=FreeDecomposition1} and \eqref{Eqn=FreeDecomposition2}   and of the semigroup  $(\Phi_t)_{t \geq 0}$ on the $O_N^+$-factors. This yields a semi-group of central multipliers on an aribtrary quantum group $U_N^+(F)$ or $O_N^+(F)$   that is moreover IGHS.
\end{rmk}

In the following proposition we collect some results from \cite{IsonoIMRN}, \cite{VaesVergnioux}, \cite{OzawaSolid} that were not stated explicitly.  We refer to \cite{IsonoIMRN} and \cite{VaesVergnioux} for the definition of bi-exactness and the Akemann-Ostrand property which shall not be used further in this paper.

\begin{dfn}
A von Neumann algebra $\cM$ is called {\bf solid}  if for any diffuse, amenable von Neumann subalgebra $\cQ \subseteq \cM$ with faithful normal conditional expectation $E_{\cQ}: \cM \rightarrow \cQ$ we have that  $\cQ' \cap \cM$ is amenable.
\end{dfn}

\begin{dfn}
	A von Neumann algebra $\cM$ is said to have the {\bf completely contractive approximation property (CMAP)} if there exists a net of normal completely contractive finite rank maps $(\Upsilon_i)_i$ on $\cM$ such that for every $x \in \cM$ we have $\Upsilon_i(x) \rightarrow x $ $\sigma$-weakly.
\end{dfn}

\begin{prop}\label{Prop=Solid}
	For any $F \in GL_N(\mathbb{C})$ and $N \geq 3$ the von Neumann algebras $L_\infty(O_N^+(F))$ and $L_\infty(U_N^+(F))$ are solid. Further, free products of such algebras are solid.
\end{prop}
\begin{proof}
 By \cite[Theorem 24]{CFY} the reduced C$^\ast$-algebras $C_r(O_N^+(F))$ and  $C_r(U_N^+(F))$ have the CMAP and hence so do their free products   \cite{RicardXu}, \cite{FreslonCRAS}. This shows that such C$^\ast$-algebras are   locally reflexive by \cite[Chapter 18]{Pisier}, \cite{NateTaka}. 	By \cite[Theorem C]{IsonoIMRN}   the (separable) quantum groups $O_N^+(F)$, $U_N^+(F)$ and their free products are bi-exact so that  by \cite[Theorem 2.5]{VaesVergnioux} (see also \cite{OzawaSolid}) they are solid.
\end{proof}

The following proposition is essentially   \cite[Main Theorem]{BHV}.
 Let $Z(\cM) = \cM \cap \cM'$ denote the center of a von Neumann algebra. Suppose that $\cQ$ is a von Neumann subalgebra of $\cM$. Then we set the   stable normalizer,
\begin{equation}\label{Eqn=StableNormalizer}
\sN_{\cM}(\cQ) = \{ x \in \cM \mid x \cQ x^\ast \subseteq \cQ, x^\ast \cQ x \subseteq \cQ  \}.
\end{equation}
 For two faithful normal states $\varphi$ and $\psi$ on $\cM$ we set
\[
\pi_{\varphi, \psi}: c_\psi(\cM) \rightarrow c_\varphi(\cM),
\]
to be the $\ast$-homomorphism given by $\pi_{\varphi, \psi}(u_s) = u_s$ and  $\pi_{\varphi, \psi}(\pi_\psi(x)) = \pi_\varphi(x)$ where $s \in \mathbb{R}, x \in \cM$.

\begin{dfn}
A Markov semi-group $(\Phi_t)_{t\geq 0}$ with  $(\Phi_t^{(2)} = e^{-t \Delta})_{t \geq 0}$ with $\Delta \geq 0$ is called {\bf immediately $L_2$-compact} if the generator $\Delta$ has compact resolvent.
\end{dfn}

Recall that in \cite{CaspersSkalski} it was proved that the existence of an immediately $L_2$-compact semi-group on a separable von Neumann algebra $\cM$ is equivalent to $\cM$ having the Haagerup property.

\begin{prop}\label{Prop=MStronglySolid}
	Let $\bG$ be a compact quantum group and let $\cM = L_\infty(\bG)$.   Suppose that $\cM$ is solid with the CMAP. Suppose moreover that $\cM$ posseses a Markov semi-group of central multipliers that is both IGHS and immediately $L_2$-compact. Then $\cM$ is strongly solid.
\end{prop}
\begin{proof}
	We follow the proof of \cite[Main Theorem]{BHV}. Let $\cQ \subseteq \cM$ be a diffuse amenable von Neuman subalgebra with expectation. We need to prove that $\cP = \Nor_{\cM}(\cQ)''$ is amenable. Fix a faithful state $\psi$ on $\cM$ such that $\cQ$ is globally invariant under $\sigma^\psi$.  The second paragraph of the proof of \cite[Main Theorem]{BHV} shows that by solidity of $\cM$  we may replace $\cQ$ by the amenable $\psi$-expected  von Neumann subalgebra $\widetilde{\cQ} = \cQ \bigvee (\cQ' \cap \cM)$ and prove that $\Nor_{\cM}(\widetilde{\cQ})'' $ is amenable. This shows that without loss of generality we can assume that $\cQ' \cap \cM = Z(\cQ)$. From this property it follows that $c_{\psi}(\cP) \subseteq \Nor_{c_\psi(\cM)}(c_\psi(\cQ))''$, see \cite[Section 4, Claim]{BHV}, and this inclusion is $\widetilde{\psi}$-expected where $\widetilde{\psi}$ was the dual weight of $\psi$. Hence we need to prove that $\Nor_{c_\psi(\cM)}(c_\psi(\cQ))''$ is amenable. Set $\cP_0 = \pi_{\varphi, \psi}(\Nor_{c_\psi(\cM)}(c_\psi(\cQ))'')$,  $\cQ_0 = \pi_{\varphi, \psi}(c_\psi(\cQ))$ and $\cM_0 = c_\varphi(\cM) = \pi_{\varphi, \psi}(c_\psi(\cM))$.
	We have $\cP_{0} = \Nor_{\cM_0}(\cQ_0)''$.

	 To show that $\cP_0$ is amenable it suffices to show that for every $\tau$-finite projection $p \in \cL_{\varphi}(\mathbb{R})$ the von Neumann algebra $p \cP_0 p$ is amenable. Let $p \in \cL_{\varphi}(\mathbb{R})$ be such a $\tau$-finite projection.    $p \cP_0 p$ is contained with expectation in $\sN_{p \cM_0 p}(p \cQ_0 p)''$. So we need to show that $\sN_{p \cM_0 p}(p \cQ_0 p)''$ is amenable.
	
	  Note that $p \cQ_0 p$ is   amenable \cite{Anantharaman}. Further, By \cite[Lemma 2.5]{HoudayerUeda} we see that as $\cQ$ is diffuse and $p$ is $\tau$-finite,  we have $p \cQ_0 p \not \prec_{p \cM_0 p} \cL_{\varphi}(\mathbb{R})p$.
	
	  As $\cM$   is equipped with a Markov semi-group of central multipliers that is IGHS, it follows that $p \cM_0 p$ carries a GC semi-group, see Proposition \ref{Lem=CrossedQG}. Moreover, by the same Proposition \ref{Lem=CrossedQG} and the discussion at the end of Section \ref{Sect=Preliminaries} (see \cite{CiprianiSauvageot})  we see that on $p \cM_0 p$ there exists a closable derivation $\partial$ into a $p \cM_0 p$-$p \cM_0 p$ bimodule that is weakly contained in its coarse bimodule of $p \cM_0 p$. Moreover the derivation is real (Lemma \ref{Lem=Real}) and satisfies $\partial^\ast \overline{\partial} = \widetilde{\Delta}$ where $\widetilde{\Delta}$ is the generator of the GC semi-group constructed in Definition \ref{Dfn=CrossSemigroup}, which on $(p  L_2(\mathbb{R})) \otimes L_2(\cM)$ is given by $p \otimes \Delta$ with $\Delta$ the generator of the IGHS semi-group on $\cM$. By assumption $\Delta$ has compact resolvent so that $\partial$ as a derivation on $p \cM_0 p$ satisfies the properness assumption \eqref{Eqn=ProperAssumption} with $\cL = p \cL_\varphi(\mathbb{R})$. As further  $\cM$ hence $p \cM_0 p$ has the CMAP, we may apply   Theorem \ref{Thm=StronglySolidGen} to the triple $(p c_\varphi(\cM)p, p \cL_{\varphi}(\mathbb{R})p, p \cQ_0 p)$ to conclude that   $\sN_{p \cM_0 p}(p \cQ_0 p)''$ is amenable.
\end{proof}

\begin{thm}\label{Thm=StrongSolidGeneral}
	Let $N \geq 3$ and take any $F \in GL_N(\mathbb{C})$. $L_\infty(U_N^+(F))$ and $L_\infty(O_N^+(F))$  are strongly solid.
\end{thm}
\begin{proof}
	Let $\bG$ be either $U_N^+(F)$ or $O_N^+(F)$ with conditions as stated in the lemma. By \cite{CFY} $L_\infty(\bG)$ has the CMAP.
	 Remark \ref{Rmk=IGHS} and Proposition \ref{Prop=FreeProduct} shows that $L_\infty(\bG)$ posseses an IGHS Markov semi-group of central multipliers. Futher, by Corollary \ref{Cor=ChebyForm} and the Leipniz rule \eqref{Eqn=Leipniz} this semi-group is immediately $L_2$-compact.
	By Proposition \ref{Prop=Solid} $\cM$ is solid. The result then follows from Proposition \ref{Prop=MStronglySolid}.
\end{proof}

  \begin{rmk}
  	Anywhere in this paper the usage of semi-groups of central multipliers can be replaced  by more general semi-groups of modular multipliers, i.e. multipliers $\Phi: L_\infty(\bG) \rightarrow  L_\infty(\bG)$ that commute with the modular group $\sigma^\varphi$ of the Haar state. In turn one may apply averaging techniques to assure the existence of such semi-groups associated with quantum groups, c.f. \cite[Proposition 4.2]{CaspersSkalski}.

  	The reason that we must work with  multipliers in this paper is to assure that the gradient bimodules $\cH_\partial$   extend from $\cA$-$\cA$-bimodules to normal $L_\infty(\bG)$-$L_\infty(\bG)$-bimodules. It would be nice to have a more conceptual understanding in the general context of von Neumann algebras for when this happens.
  \end{rmk}

  \section{Related results: amenability and equivariant compressions}\label{Sect=Final}

We collect some final corollaries.   Firstly, we recall the following result from    \cite{CiprianiSauvageotAmenable}. We give their  proof in terms of Stinespring dilations.

  \begin{thm}[Theorem 3.15 of \cite{CiprianiSauvageotAmenable}]\label{Thm=Amenable}
  	Let $\cM$ be a von Neumann algebra and  suppose that there exists a conservative completely Dirichlet form $Q$ associated with $\cM$ such that $\Delta_Q$ has a complete set of eigenvectors  with  eigenvalues $0 \leq \lambda_1 \leq \lambda_2 \leq \lambda_3 \leq \ldots$ (with multiplicity 1, so an eigenvalue may occur multiple times). If
  	\begin{equation}\label{Eqn=LGrowth}
  	\liminf_{n \in \mathbb{N}} \frac{\lambda_n}{\log(n)} = \infty,
  	\end{equation}
  	then $\cM$ is amenable.
  \end{thm}
  \begin{proof}
  	As $Q$ is a conservative completely Dirichlet  form there exists a Markov semi-group $(\Phi_t)_{t \geq 0}$ on $\cM$ such that $\Phi_t^{(2)} = e^{-t \Delta_Q}$.   \eqref{Eqn=LGrowth} implies that for any $K >0$ we find for large $n$ that  $e^{\lambda_n} > n^K$. So if $K >   t^{-1}$ we see that for large $n$ we get  $e^{-t \lambda_n} < n^{-1}$. So $e^{-t\Delta_Q}$ is Hilbert-Schmidt.
  	Let $(\cH_t, \eta_t)$ be the pointed Stinespring $\cM$-$\cM$-bimodule of $\Phi_t$. 	
  	By Lemma \ref{Lem=HSStinespring}  for every $t > 0$  we have $\cH_t$ is weakly contained in the coarse bimodule of $\cM$. As $t \mapsto \Phi_t$ is strongly continuous we get  that  $\cH_0$ is weakly contained in the coarse bimodule. Then as
    $\Phi_0 = \Id_{\cM}$,  $\cH_0$ is the identity bimodule  and so  $\cM$ is amenable.
  \end{proof}

As an application we give a von Neumann algebraic proof of a compression result.
  Recall \cite{GuentnerKaminker} that if $\Gamma$ is a finitely generated discrete group then the equivariant compression $s^\sharp(\Gamma)$ of $\Gamma$ is the supremum over all $s\geq 0$ such that there exists a cocycle $c: \Gamma \rightarrow \cH_\pi$ into some Hilbert space representation $(\pi, \cH_\pi)$ with $d(\gamma_1, \gamma_2)^s \leq \Vert c(\gamma_1) - c(\gamma_2) \Vert_2$. Recall we say that $c$ is a cocycle if it satisfies $c(\gamma_1 \gamma_2) = c(\gamma_1) + \pi(\gamma_1) c(\gamma_2)$ for all $\gamma_1, \gamma_2 \in \Gamma$.  Necessarily $s^\sharp(\Gamma) \leq 1$ as cocycles are Lipschitz.

  \begin{cor}[Theorem 5.3 \cite{GuentnerKaminker}]
  	Let $\Gamma$ be a finitely generated discrete group. If  for the equivariant compression we have  $s^\sharp(\Gamma) > \frac{1}{2}$ then $\Gamma$ is amenable.
  \end{cor}
  \begin{proof}
  	For $\delta > 0$ small there is a cocycle $c: \Gamma \rightarrow \cH_\pi$  for some representation $(\pi, \cH_\pi)$ of $\Gamma$ such that \begin{equation}\label{Eqn=cgrows}
  	\Vert c(\gamma) \Vert_{\cH_\pi} = \Vert c(\gamma) - c(0)\Vert_{\cH_\pi}  \geq l(\gamma)^{\frac{1}{2} + \delta}.
  	\end{equation}
  	Then $\psi(\gamma) =  \Vert c(\gamma)\Vert_{\cH_\pi}^2$ is conditionally positive definite and so the semi-group $(e^{-t \psi})_{t >0}$ of multiplication operators on $\ell_2(\Gamma)$ yields an $L_2$-implementation of a Markov semi-group and thus a Dirichlet form $Q(\xi) = \langle \psi^{\frac{1}{2}} \xi, \psi^{\frac{1}{2}} \xi \rangle$.  As \eqref{Eqn=cgrows} gives $\psi(\gamma) \geq l(\gamma)^{1 + 2 \delta}$ we see that this Dirichlet form satisfies \eqref{Eqn=LGrowth}. So the proof is concluded from Theorem \ref{Thm=Amenable}.
  \end{proof}

  \appendix

\section{Stable strong solidity and derivations in the tracial case}\label{Sect=AppendixA}

The final part of the proof of Proposition \ref{Prop=MStronglySolid} requires a generalization of some of the main results of \cite{OzawaPopaAJM}.
 The proof is almost identical to \cite[Corollary B]{OzawaPopaAJM}. Similar observations were made in  \cite[Remark 3.3]{Sinclair}  and it was suggested in the context of Kac type quantum groups in \cite[Remarks after Theorem 4.10]{FimaVergnioux}. Since we need both a von Neumann version of the group theoretical results from \cite{OzawaPopaAJM} as well as a stable version we present the proof here.

\subsection{Weak compactness}\label{Sect=A11}
Troughout all of the appendix, let $(\cM, \cL, \cQ)$ be triple of a finite von Neumann algebra $\cM$ with normal faithful tracial state $\tau$, an amenable von Neumann subalgebra $\cQ \subseteq \cM$ and  a von Neumann subalgebra $\cL$ of $\cM$ with the property that $\cQ$ does not embed into $\cL$ inside $\cM$ in the sense of Popa (notation $\cQ \not \prec_\cM \cL$), see \cite{PopaStrong}, \cite{PopaBetti} for details.

 Recall that the stable normalizer $\sN_{\cM}(\cQ)$ of $\cQ$ in $\cM$ was defined in  \eqref{Eqn=StableNormalizer}.
Next we introduce the object $\sN_{\cM}^\circ(\cQ)$ below for which we need the following terminology. We refer  to the discussion before \cite[Proposition 3.6]{BHV} for further details.
 Let $E_{Z}$ be the conditional expectation of $\cQ$ onto $Z(\cQ)$, the center of $\cQ$. For $x \in \sN_{\cM}(\cQ)$ set $z_x^r$ to be the support of $E_Z(x^\ast x)$ and set $z_x^l$ to be the support of $E_Z(x x^\ast)$. Suppose that $x = v \vert x \vert$ is the polar decomposition of $x$. Then, we denote by $\alpha_v$ the unique $\ast$-homomorphism $Z(\cQ) z^x_r \rightarrow Z(\cQ) z^x_l$ determined by $va = \alpha_v(a) v$ with $a \in Z(\cQ) z^x_r$. Then set $\alpha_x = \alpha_v$ and we let $\Delta_x$ be the Radon-Nykodym derivative between $\tau$ and $\tau \circ \alpha_x$.
 We  set,
\[
\sN_{\cM}^\circ(\cQ) = \{ x \in \sN_{\cM}(\cQ) \mid \exists \delta >0 \textrm{ such that }  E_{\mathcal{Z}}(x^\ast x) \geq \delta z_x^r \textrm{ and }   E_{\mathcal{Z}}(x x^\ast) \geq \delta z_x^l \}.
\]
As explained in  the discussion before \cite[Proposition 3.6]{BHV}, is that we have an equality $\sN_{\cM}^\circ(\cQ)''  = \sN_{\cM}(\cQ)''$. We shall call the latter von Neumann algebra   $\cP$. Moreover, the partial isometries in $\sN_{\cM}^\circ(\cQ)$ generate $\cP$.

Now  we need the following weak compactness type property obtained in \cite[Proposition 3.6]{BHV}.

\begin{prop}\label{Prop=WkCpt}
	Let $(\cM, \tau)$  be a finite  von Neumann algebra with the CMAP. Let $\cQ$ be an amenable von Neumann subalgebra of $\cM$. Then there exists a net of   positive vectors $\eta_n \in L_2(\cQ \otimes \cQ^{{\rm op}})$ such that
\begin{enumerate}
\item\label{Item=WkCptI} $\lim_n \Vert (a \otimes 1) \eta_n - (1 \otimes a^{\op}) \eta_n \Vert_2 = 0$, for all $a \in \cQ$.
\item $\lim_n \Vert (x \otimes 1) \eta_n (x^\ast \Delta_x^{\frac{1}{2}}  \otimes 1 ) - (1 \otimes x^{\op}) \eta_n (1 \otimes \overline{x}) \Vert_2 = 0$ for all $x \in \sN_{\cM}^\circ(\cQ)$.
\item\label{Item=WkCptII}  $\langle (x \otimes 1) \eta_n, \eta_n \rangle = \tau(x)$, for all $x \in \cM$.
\end{enumerate}
Moreover for every partial isometry $v \in \sN_{\cM}^\circ(\cQ)$ there exists a sequence of  elements $T(v,k)$ in the unit ball of the algebraic tensor product $\cM \odot \cM^{\op}$ such that
\[
\begin{split}
\lim_{k} \left(  \limsup_{n}  \Vert  (v \otimes 1) \eta_n - (1 \otimes v^{\op}) \eta_n T(v,k) \Vert_2  \right) = & 0,\\
\lim_{k} \left(  \limsup_{n}  \Vert  (v^\ast \otimes 1) \eta_n - (1 \otimes \overline{v}) \eta_n T(v,k)^\ast \Vert_2  \right) = & 0.
\end{split}
\]
Here $T(v,k)$ are elements in the unit ball of $\cM \odot \cM^{\op}$.
\end{prop}

\subsection{Derivations}
Now suppose that $\partial$ is a closable derivation on some $\sigma$-weakly dense $\ast$-subalgebra $\subseteq \cM$ into a $\cM$-$\cM$-bimodule $\cH$. Moreover, assume that $\partial$ is real.  Let $\overline{\partial}$ be its closure. By \cite{DaviesLindsay}, \cite{SauvageotHeidel} (so in the tracial case) we have that $\Dom(\overline{\partial}) \cap \cM$ is still a $\sigma$-weakly dense $\ast$-subalgebra on which $\overline{\partial}$ satisfies the Leibniz rule. Replacing $\partial$ by $\overline{\partial}$ we may assume without loss of generality that $\partial$ is closed. We introduce notation (see \cite{OzawaPopaAJM}, \cite{Peterson}), for $\alpha > 0$,
\[
\Delta = \partial^\ast \partial, \qquad \zeta_\alpha = \sqrt{\frac{\alpha}{\alpha + \Delta}}, \qquad \partial_\alpha = \alpha^{-\frac{1}{2}} \partial \circ \zeta_\alpha, \qquad \Delta_\alpha = \sqrt{\frac{\Delta}{\alpha + \Delta}} \quad {\rm and } \quad \theta_\alpha = 1 - \Delta_\alpha.
\]
Let $e_\cL$ be the Jones projection of $\cM$ onto $\cL$; it is the map $x \Omega_\tau \mapsto E_{\cL}(x) \Omega_\tau$ with $E_\cL: \cM \rightarrow \cL$ the $\tau$-preserving conditional expectation.
We further assume the type of properness assumption:
\begin{equation}\label{Eqn=ProperAssumption}
\theta_\alpha \in C^\ast(  \cM e_{\cL} \cM), \qquad \forall \alpha > 0,
\end{equation}
where  $C^\ast(  \cM e_{\cL} \cM)$ is the C$^\ast$-algebra generated by $\cM e_{\cL} \cM$. This suffices to still get the following result from \cite[Lemma 5.2]{OzawaPopaAJM} (recall that we assumed that $\cQ \not \prec_{\cM} \cL$).

\begin{lem}\label{Lem=52}
		Let $(\cM, \cL, \cQ)$ be as before with $\cM$ having the CMAP. Let $\partial$ be a real closable derivation on $\cM$ satisfying \eqref{Eqn=ProperAssumption}. Consider further notation as above.
	For every $\alpha > 0$ and $a \in \cM$ we have
	\[
	\lim_n \Vert (\partial_\alpha \otimes 1) (a \otimes 1)  \eta_n \Vert = \Vert a \Vert_2.
	\]
\end{lem}
\begin{proof}
	For an element $x$ in the Jones construction  $\langle \cM, e_{\cL} \rangle$ set
	\[
	\varphi_0(x) = \lim_n \langle  (x \otimes 1) \eta_n, \eta_n \rangle.
	\]
	From property (1) of Lemma \ref{Prop=WkCpt} it follows that $\varphi_0$ is a $\cQ$-central state on $\langle \cM, e_{\cL} \rangle$. Further by (3) of Lemma \ref{Prop=WkCpt} we see that $\varphi_0$ restricts as $\tau$ on $\cM$.  As $\cQ \not \prec_{\cM} \cL$ and by the Assumption \ref{Eqn=ProperAssumption} we find from \cite[Lemma 3.3]{OzawaPopaAJM} that $\varphi_0(a^\ast \theta_\alpha^\ast \theta_\alpha a  ) = 0$. Therefore
	\[
	\lim_n \Vert  (\partial_\alpha \otimes 1) (a \otimes 1)  \eta_n \Vert^2
	= \varphi_0( a^\ast \partial_\alpha^\ast \partial_\alpha a ) =
	 \varphi_0( a^\ast  a -  a^\ast  \theta_\alpha^\ast \theta_\alpha a  ) = \varphi_0(a^\ast a) = \Vert a \Vert_2^2.
	\]
\end{proof}

We put $\cK = \cH \otimes L_2(\cM)$ as a $\cM$-$\cM$-bimodule and we denote $\rho: \cM^\op \rightarrow B(\cK)$ for the right action.
For $\alpha > 0$ and $p \in \cQ' \cap \cM$ a projection we set
\begin{equation}\label{Eqn=EtaVectors}
\eta_n^{p, \alpha} = ( \partial_\alpha \otimes 1)((p \otimes 1) \eta_n) \in \cK.
\end{equation}
We proceed as in the proof of \cite[Proposition 3.7]{BHV}.

\begin{lem}\label{Lem=Subnet}
		Let $(\cM, \cL, \cQ)$ be as before with $\cM$ having the CMAP. Let $\partial$ be a real closable derivation on $\cM$ satisfying \eqref{Eqn=ProperAssumption}. Let $p \in \cQ' \cap \cM$ be any projection.  There exists a subnet, say $\eta^p_i = \eta_{n_i}^{p, \alpha_i}$, of the vectors $\eta_{n}^{p, \alpha}$ and elements $S(v,i)$ in the unit ball of $\cM \odot \cM^{\op}$ indexed by partial isometries $v \in \sN_{\cM}^\circ(\cQ)$  with the property that for every $v \in \sN_{\cM}^\circ(\cQ)$,
	\[
	\begin{split}
	& \lim_i \Vert (\zeta_{\alpha_i}( v) \otimes 1) \eta_i^p - (1 \otimes v^{\op}) \eta_i^p (\zeta_\alpha \otimes \id)( S(v,i))  \Vert_2 = 0,\\
	& \lim_i \Vert (\zeta_{\alpha_i}( v)^\ast \otimes 1) \eta_i^p - (1 \otimes \overline{v}) \eta_i^p (\zeta_\alpha \otimes \id)( S(v,i))  \Vert_2 = 0,\\
	\end{split}
	\]
	and further   $\lim_{i} \Vert   \eta_i^p \Vert_2 = \Vert p \Vert_2$. Moreover, for every $x \in \cM$,
	\begin{equation}\label{Eqn=Oz5Dot2}
	\Vert   ( \partial_\alpha \otimes 1 )(x \otimes 1) \eta_{\alpha_i} \Vert = \Vert x \Vert_2.
	\end{equation}
\end{lem}
\begin{proof}
	Let $\eta_n$ be the vectors constructed in Proposition \ref{Prop=WkCpt}.
	Let the net be indexed by tuples $(F, G, \delta)$ with $F$ a finite subset of partial isometries in $\sN_{\cM}^\circ(\cQ)$, $G \subseteq \cM$ finite and $\delta > 0$. Given such $i = (F, G, \delta)$ we apply Proposition \ref{Prop=WkCpt} to find $k$ large such that for all $v \in F$,
	\[
	\begin{split}
	& \limsup_n \Vert (v \otimes 1) \eta_n - (1 \otimes v^{\op}) \eta_n T(v,k) \Vert_2 < \delta, \\
	&   \limsup_n \Vert (v^\ast \otimes 1) \eta_n - (1 \otimes \overline{v}) \eta_n T(v,k)^\ast \Vert_2 < \delta.
	\end{split}
	\]
	By \cite[Lemma 4.2]{OzawaPopaAJM} we may find $\alpha$ large such that for all $v \in F$ we get,
	\[
	\limsup_n \Vert (\zeta_\alpha(v) \otimes 1) (\partial_\alpha \otimes \id)((p \otimes 1) \eta_{n})  -  (\partial_\alpha \otimes \id)(( vp \otimes 1) \eta_{n}) \Vert_2 < \delta,
	\]
	and for all $v \in F$,
	\[
	    \limsup_n  \Vert  (\partial_\alpha \otimes \id) ((p \otimes 1) \eta_n  T(v,k) ) -
	    (\partial_\alpha \otimes \id) ((p \otimes 1) \eta_n  )  (\zeta_\alpha \otimes \id)(T(v,k) )
	      \Vert_2 < \delta,
	\]
	and similarly with $T(v,k)$ replaced by $T(v,k)^\ast$.
	Combining these estimates and recalling the definition of $\eta_{n}^{p, \alpha}$ from \eqref{Eqn=EtaVectors} yields
	\[
	\begin{split}
	& \Vert (\zeta_\alpha(v) \otimes 1) \eta_{n}^{\alpha, p} - (1 \otimes v^\op) \eta_n^{\alpha, p} (\zeta_\alpha \otimes \id)(T(v,k)) \Vert_2 < 3\delta,\\
	& \Vert (\zeta_\alpha(v^\ast) \otimes 1) \eta_n^{\alpha, p} - (1 \otimes \overline{v}) \eta_n^{\alpha, p} (\zeta_\alpha \otimes \id)(T(v,k)) \Vert_2 < 3 \delta,
	\end{split}
	\]
and moreover, these estimates hold uniformly in $n$. Then Lemma \ref{Lem=52} shows that for any $\alpha > 0$ we may take $n$ so large that $\vert \Vert  \eta_n^{p,\alpha} \Vert - \Vert p \Vert_2 \vert < \delta$. Moreover, the same Lemma \ref{Lem=52} shows that for $x \in \cM$ we get \eqref{Eqn=Oz5Dot2}.
\end{proof}

Now again let $p \in \cQ' \cap \cM$ be arbitrary and let $\eta_{n_i}^{p, \alpha_i}$ be the net of vectors constructed in Lemma \ref{Lem=Subnet}. Set the functional on $\rho(\cM^{\op})' \cap B(\cK)$ given by,
\[
\varphi_{p,i}(x) = \Vert p \Vert_2^{-2} \langle (x \otimes 1) \eta_{n_i}^{p, \alpha_i}, \eta_{n_i}^{p, \alpha_i} \rangle.
\]
By Lemma \ref{Lem=Subnet} define  the state $\Omega_p$  by $\Omega_p(x) = \lim_i \varphi_{p,i}(x), x \in \rho(\cM^{\op})' \cap B(\cK)$ for some ultralimit.

\begin{lem}\label{Lem=A3}
	Let $(\cM, \cL, \cQ)$ be as before with $\cM$ having the CMAP. Let $\partial$ be a real closable derivation on $\cM$ satisfying \eqref{Eqn=ProperAssumption}.Consider further notation as above.
	Fix $p \in \cQ' \cap \cM$ and let $\cP = \sN_{\cM}(\cQ)''$. For $a \in \cP$ we have
	\begin{equation}\label{Eqn=PhiConvergence}
	\lim_i \vert  \varphi_{p, i}(   x \zeta_{\alpha_i}(a)) - \varphi_{p, i}(  \zeta_{\alpha_i}(a) x ) \vert = 0,
	\end{equation}
	uniformly for $x \in \cB(\cK) \cap \rho(\cM^{\op})'$ with $\Vert x \Vert \leq 1$. Further if $a = up$ for a unitary $u$ in $\cP$ then
	\begin{equation}\label{Eqn=PhiConvergenceII}
	\lim_i \vert  \varphi_{p, i}( \zeta_{\alpha_i}(a)^\ast x \zeta_{\alpha_i}(a)) - \varphi_{p, i}(   x ) \vert = 0,
	\end{equation}
	uniformly for $x \in \cB(\cK) \cap \rho(\cM^{\op})'$
\end{lem}
\begin{proof}
Firstly, let $x \in \cB(\cK) \cap \rho(\cM^{op})'$.  By Lemma \ref{Lem=Subnet} we get for partial isometries $v \in \sN_{\cM}^\circ(\cQ)$   that
\[
\begin{split}
&  \lim_i \varphi_{p, i}(x \zeta_{\alpha_i}(v)) =
\lim_i \langle   (x \zeta_{\alpha_i}(v) \otimes 1) \eta_{i}^p, \eta_i^p \rangle
= \lim_i
\langle   (x  \otimes v^\op) \eta_{i}^p  (\zeta_\alpha \otimes \id) (S(v, i)) , \eta_i^p \rangle\\
 = &  \lim_i
 \langle   (x  \otimes 1) \eta_{i}^p   , (1 \otimes \overline{v})  \eta_i^p (\zeta_\alpha \otimes \id) (S(v, i))^\ast \rangle
 = \lim_i
 \langle   (x  \otimes 1) \eta_{i}^p   , (\zeta_{\alpha_i}( v)^\ast  \otimes 1)  \eta_i^p   \rangle\\
 = & \lim_i \langle   (\zeta_{\alpha_i}( v) x  \otimes 1) \eta_{i}^p   ,    \eta_i^p   \rangle
 =  \lim_i \varphi_{p, i}(\zeta_{\alpha_i}(v) x ).
\end{split}
\]
Moreover, as these limits actually hold on the vector level this argument shows that
\[
  \lim_i \vert \varphi_{p, i}(x \zeta_{\alpha_i}(v)) - \varphi_{p, i}( \zeta_{\alpha_i}(v)  x) \vert = 0.
  \]
uniformly for $x$ in the unit ball of  $\cB(\cK) \cap \rho(\cM^{op})'$.

We wish to extend this commutation type property beyond  partial isometries $v \in \sN_{\cM}^\circ(\cQ)$ as follows. Take $v,w \in \sN_{\cM}^\circ(\cQ)$ partial isometries.
By Cauchy-Schwarz  we see
\[
\begin{split}
& \vert \varphi_{p, i}(x \zeta_{\alpha_i}(v)) - \varphi_{p, i}(x \zeta_{\alpha_i}(w)) \vert^2 \leq
\varphi_{p, i}( (\zeta_{\alpha_i}(v) - \zeta_{\alpha_i}(w))^\ast  (\zeta_{\alpha_i}(v) - \zeta_{\alpha_i}(w)) ) \varphi_{p,i}(x x^\ast).
\end{split}
\]
Taking limits (using \cite{OzawaPopaAJM}[Lemma 4.2] and \eqref{Eqn=Oz5Dot2}) we find
\[
\lim_i \vert \varphi_{p, i}(x \zeta_{\alpha_i}(v)) - \varphi_{p, i}(x \zeta_{\alpha_i}(w)) \vert^2
\leq \Omega(x x^\ast) \Vert v - w \Vert_2^2.
\]
 Similarly,
 $\lim_i  \vert \varphi_{p, i}(\zeta_{\alpha_i}(v) x ) - \varphi_{p, i}(\zeta_{\alpha_i}(w) x) \vert^2 \leq \Omega( x^\ast x) \Vert v-w\Vert_2^2$. Therefore by $L_2$-density in $\cP$ of the span of partial isometries in $\sN_{\cM}^0(\cQ)$ we see that in fact for all $a \in \cP$ we get that
 \[
 \lim_i \vert \varphi_{p, i}(x \zeta_{\alpha_i}(a)) - \varphi_{p, i}(\zeta_{\alpha_i}(a)  x ) \vert = 0,
 \]
 uniformly on the unit ball of $\cB(\cK) \cap \rho(\cM^{op})'$.  This yields the first claim.
  Now if $u \in \cP$ is a unitary we get again from \cite{OzawaPopaAJM}[Lemma 4.2] and \eqref{Eqn=Oz5Dot2} that
 \[
 \begin{split}
& \lim_i \varphi_{p, i}( \zeta_{\alpha_i}(up)^\ast \zeta_{\alpha_i}(up)   ) =
 \Vert p \Vert_2^{-2}
 \lim_i \Vert  (\zeta_{\alpha_i}(up)  \otimes 1)  (\widetilde{\delta}_{\alpha_i}  \otimes 1)((p \otimes 1)  \eta_i)     \Vert^2 \\
 = &
 \Vert p \Vert_2^{-2}
 \lim_i \Vert   (\widetilde{\delta}_{\alpha_i}  \otimes 1)((up \otimes 1)  \eta_i)     \Vert^2
 = \Vert p \Vert_2^{-2} \Vert u p \Vert_2^{2} = 1.
 \end{split}
 \]

\end{proof}

We now follow the proof of \cite{OzawaPopaAJM} to obtain the following.

\begin{thm}\label{Thm=StronglySolidGen}
Let $(\cM, \cL, \cQ)$ be as before with $\cM$ having the CMAP. Let $\partial$ be a real closable derivation on $\cM$ satisfying \eqref{Eqn=ProperAssumption}.  Then $\sN_{\cM}(\cQ)''$ is amenable.
\end{thm}
\begin{proof}
	Take a non-zero projection $p \in \sN_{\cM}(\cQ)' \cap \cM$ (so certainly $p \in \cQ' \cap \cM$) and let $F \subseteq \sN_{\cM}(\cQ)''$ be a finite set of unitaries. By \cite[Lemma 2.2]{Haagerup} to show that $\sN_{\cM}(\cQ)''$ is amenable it suffices to show that
	\begin{equation}\label{Eqn=HaagerupCriterium}
	 \Vert \sum_{u \in F}  u p \otimes \overline{up} \Vert_{\cM \otimes  \cM^\op} = \vert F \vert.
	\end{equation}	
	 By Lemma \ref{Lem=A3} we find that
	 \begin{equation}\label{Eqn=PhiConvergence}
	 \lim_i \vert  \varphi_{p, \alpha_i}( \zeta_{\alpha_i}(up)^\ast x \zeta_{\alpha_i}(up)) - \varphi_{p, \alpha_i}(   x ) \vert = 0,
	 \end{equation}
	 uniformly for $x \in B(\cK) \cap \rho(\cM^{\op})'$ with $\Vert x \Vert \leq 1$. As $\cH$ is weakly contained in the coarse bimodule we get that the left $\cM$ action on $\cH \otimes L_2(\cM)$ extends to a ucp map $\Psi: B(L_2(\cM)) \rightarrow \rho(\cM^{\op})' \cap B(\cH \otimes L_2(\cM))$. Then $\psi_{p,i} = \varphi_{p, \alpha_i} \circ \Psi$ satisfies
	 \[
	 	 \lim_i \vert  \psi_{p, i}( \zeta_{\alpha_i}(up)^\ast x \zeta_{\alpha_i}(up)) - \psi_{p, i}(   x ) \vert = 0,
	 \]
	 uniformly for $x \in B(\cK) \cap \rho(\cM^{\op})'$ with $\Vert x \Vert \leq 1$. From Powers-Stormer we get that
	 \begin{equation}\label{Eqn=Last}
\lim_i \Vert \sum_{u \in F}  \zeta_{\alpha_i}(u p) \otimes \overline{\zeta_{\alpha_i}(up)} \Vert_{\cM \otimes  \cM^\op} = \vert F \vert.
	 \end{equation}
	 Since $\zeta_{\alpha_i}$ is ucp we see that \eqref{Eqn=HaagerupCriterium} is larger than \eqref{Eqn=Last} which concludes the clam.

\end{proof}

\end{document}